%% file: main.tex
\documentclass[nonblindrev]{informs3}

\usepackage{booktabs}       % professional-quality tables
\usepackage{bbm}
\usepackage[colorlinks,allcolors=black]{hyperref}
\usepackage{mathrsfs}
\usepackage{setspace}
\usepackage{array, makecell}
\usepackage{multirow}
\usepackage{mathtools}
\usepackage{wrapfig}
\usepackage[ruled,vlined]{algorithm2e}
\usepackage[compact]{titlesec}
\usepackage[compatible]{algpseudocode}
\usepackage{soul}
\usepackage{fix-cm}   % enables scalable CM fonts, including cmsy

\usepackage{rotating}
\usepackage{fancyvrb}

\OneAndAHalfSpacedXII % Current default line spacing

\usepackage{tikz}

\usepackage{natbib}
 \bibpunct[, ]{(}{)}{,}{a}{}{,}%
 \def\bibfont{\small}%

\newcommand*{\eg}{\emph{e.g.}{}}
\newcommand*{\ie}{\emph{i.e.}{}}

\newcommand{\norm}[1]{\left\lVert#1\right\rVert}

\EquationsNumberedThrough    % Default: (1), (2), ...

\TheoremsNumberedThrough     % Preferred (Theorem 1, Lemma 1, Theorem 2)
\ECRepeatTheorems  %

\begin{document}

\RUNAUTHOR{Xing, Zhu, and Xie}
\RUNTITLE{Generalized Hypercube Queueing Models with Overlapping Service Regions}

\TITLE{Generalized Hypercube Queueing Models with Overlapping Service Regions}

\ARTICLEAUTHORS{%
\AUTHOR{Wenqian Xing}
\AFF{Department of Management Science and Engineering, Stanford University, \EMAIL{wxing@stanford.edu}}
\AUTHOR{Shixiang Zhu}
\AFF{Heinz College of Information Systems and Public Policy, Carnegie Mellon University, \EMAIL{shixianz@andrew.cmu.edu}}
\AUTHOR{Yao Xie}
\AFF{H. Milton Stewart School of Industrial and Systems Engineering, Georgia Institute of Technology, \EMAIL{yao.xie@isye.gatech.edu}}
}

\ABSTRACT{%
Motivated by the operations of the Atlanta Police Department, where heavy workloads and staffing shortages increasingly require units to patrol across overlapping service regions, we develop a generalized hypercube queueing model, extending \cite{larson1974hypercube}, for spatial service systems with overlapping coverage.
Designing effective service regions requires capturing both workload and the operations of mobile servers such as police units. The classical hypercube model, which tracks only whether each server is busy or idle, is well suited to light traffic but inadequate for congested systems with server-specific queues and restricted service regions.
We model the system as a Markov chain on a nonnegative integer-valued state space and develop a sparse truncated hyperlattice approximation for efficient steady-state computation and performance evaluation. We further characterize the workloads attainable under stable dispatch policies, establish conditions for stabilizability, and identify dispatch policies that optimally balance workloads. We validate the model through simulation and apply it to the Atlanta Police Department, where rising workloads, staffing shortages, and boundary effects create significant operational challenges.
Using real 911 calls-for-service data, our analysis indicates that a police operations system with permitted overlapping patrols can significantly mitigate these problems, leading to more effective deployment of the police force.
Although the paper focuses on police districting applications, the generalized hypercube queueing model is applicable to other mobile server models in the general setup.
}


\maketitle

\section{Introduction}
\label{sec:introduction}

Service systems, such as Emergency Service Systems (ESS), are vital in providing emergency aid to communities, protecting public health, and ensuring safety. However, these systems often face resource limitations, including constrained budgets and personnel shortages \citep{peng2020probabilistic, yoon2021stochastic, zhu2022data}. As a result, the public and associated organizations need to evaluate and improve the efficiency of these systems to maintain their effectiveness.

Service systems are generally designed such that a specific emergency response unit is primarily responsible for a small geographical area while remaining available to aid neighboring regions if necessary, either through a structured dispatch hierarchy or on an ad-hoc basis. For instance, in policing service systems, personnel are typically assigned to a designated area known as a beat. Police departments often allocate patrol forces by dividing a city's geographical areas into multiple police zones. The design of these patrol zones has a significant impact on response times to 911 calls. Officers patrol their assigned beats and respond to emergency calls, with each spatial region having one queue.

A key component for analyzing service systems is to model the queueing dynamics of mobile servers, such as police cars, ambulances, and shared rides, depending on the context. Mobile servers refer to resources or servers that are not fixed in one location but are mobile, often to provide services in varying locations or under changing conditions. Some examples of mobile servers include ambulances, fire rescue trucks, on-demand delivery systems, and disaster response units.  In his seminal work, \cite{larson1974hypercube} introduced a general model for mobile servers, known as the {\it hypercube queueing models}. The model incorporated geographical components, or atoms, each with unique arrival rates, and mobile servers with distinct service rates.
This model can capture the light traffic regime of police operations, compute the steady-state distribution of binary states, and conveniently evaluate general performance metrics.

The motivation for this study stems from various factors, including modeling police operations with heavy workload (often linked to staff shortages), severe staff shortages, and the observed boundary effect of crime incidents.
These observations are supported by data, as illustrated in Figures~\ref{fig:response-time}, \ref{fig:stats}, and \ref{fig:bordereffect}. Each factor will be discussed in greater detail below.

\begin{figure}[!t]
\FIGURE
    {\includegraphics[width=0.8\textwidth]{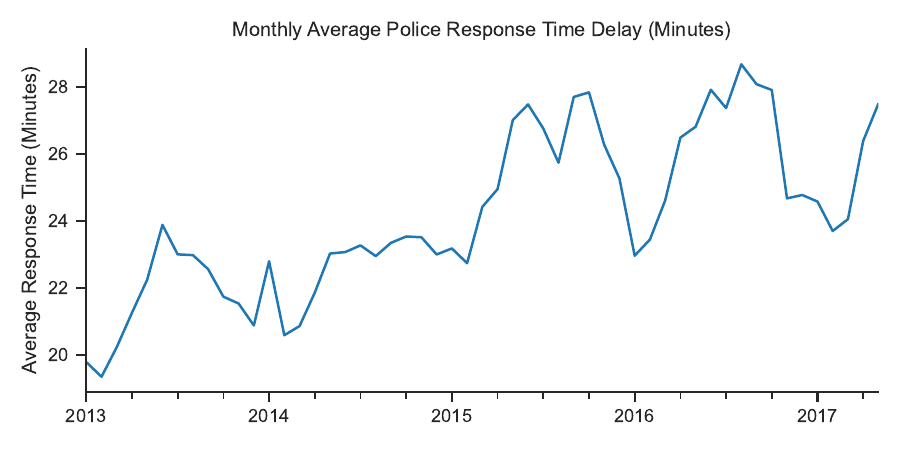}}
    {The increasing average police response time in Atlanta by month, from 2013 onwards. \label{fig:response-time}}
    {}
\end{figure}

\begin{figure}[!t]
\FIGURE
    {\includegraphics[width=0.9\textwidth]{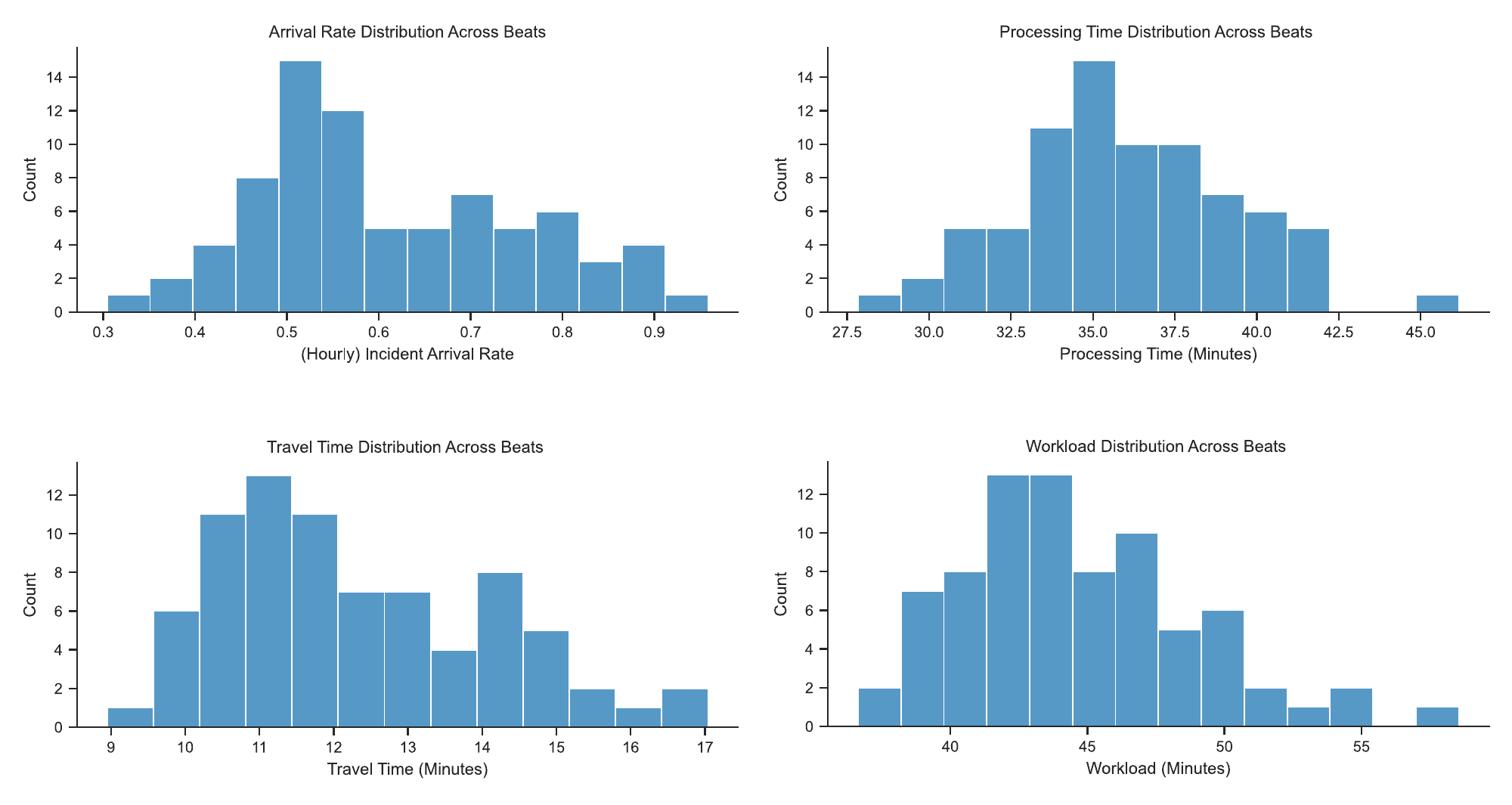}}
    {Operations statistics calculated using 911 data from the Atlanta Police Department, from 2013 to 2022, averaged for each of the 80 beats. \label{fig:stats}}
    {These figures show that calls are coming in faster than they can be handled by the police, as indicated by the workload in minutes and the hourly rate of the calls.
    }
\end{figure}

\begin{figure}[!t]
\FIGURE
    {\includegraphics[width=0.4\textwidth]{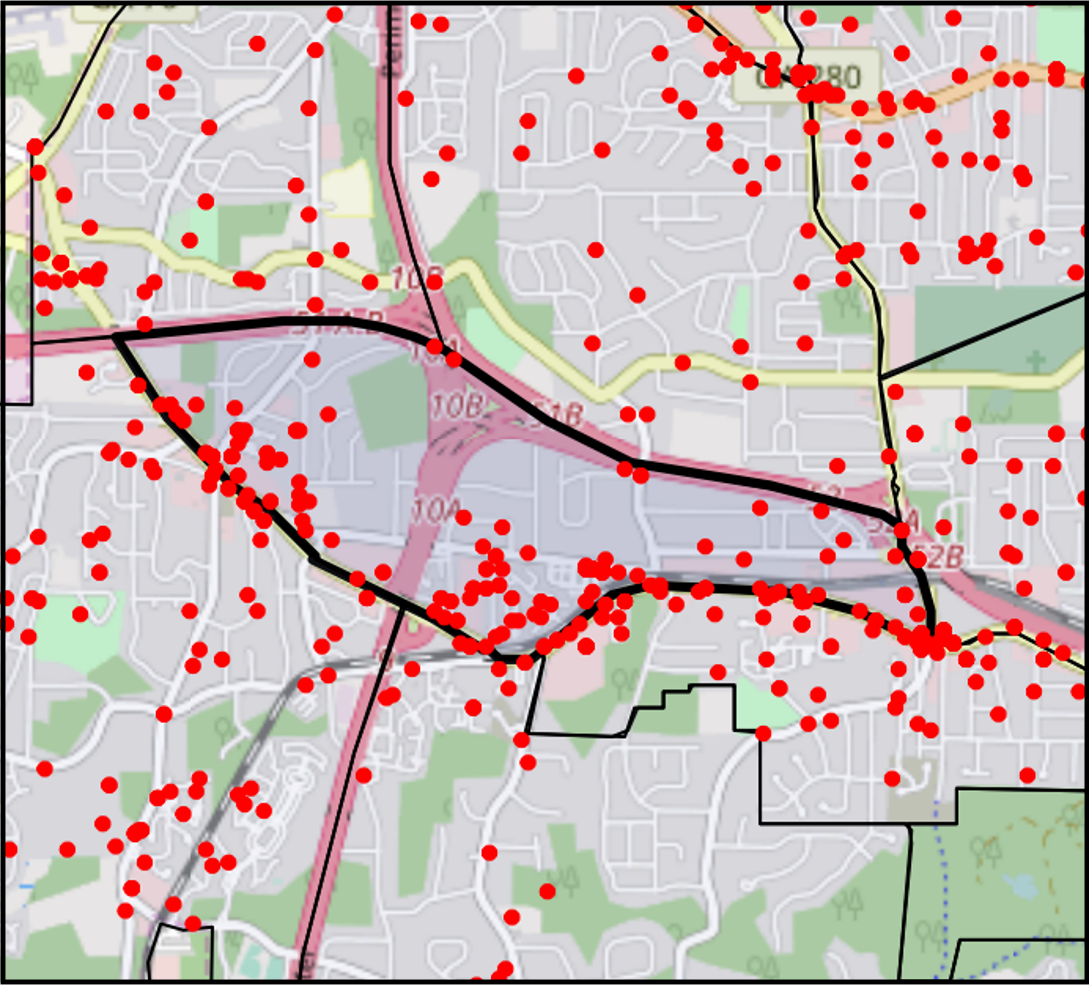}}
    {The border effect. \label{fig:bordereffect}}
    {Burglary incidents reported by the Atlanta police department in 2017 (represented by red dots) clustered near the border between beats (indicated by black lines).}
\end{figure}

First, modern police operations, often strained by high workload, necessitate a more flexible queueing model than what is provided by the traditional hypercube queueing model.
The hypercube model, while elegant in its simplicity, is limited to representing each server's status as simply ``busy'' or ``idle'' and operates under the assumption of a singular queue for the entire system.
However, in today's context, police operations systems frequently face overwhelming demands, leading to a scenario where all servers are constantly busy, which may not be accurately captured by the hypercube model.
Therefore, to effectively understand systems with heavier loads, where there could be multiple queues and each queue might exceed a length of one,
our objective is to model the distribution of queue lengths for each server. This approach goes beyond the basic ``busy'' and ``idle'' states (represented by binary vectors) used in Larson's original hypercube queueing models, offering a more detailed and realistic representation of contemporary police operations.

Second, the traditional ``one-officer-per-beat'' operational mode has become increasingly difficult to implement due to severe staff shortages, which currently pose a significant challenge for many police departments, including those in the Atlanta metro area \citep{Metropol35:online}. For example, as of January 2023, suburban Atlanta's Gwinnett County (GA) reportedly ``employed 690 officers out of an authorized strength of 939, leaving more than a quarter of sworn officer positions vacant'' \citep{AtlantaA89:online}. This leads to busier police operations and challenges in assigning at least one officer to each beat each day. To address officer shortages in practice, some municipalities have implemented flexible service regions within their districts, enabling officers to respond to incidents more effectively. This empirical strategy has proven successful. Moreover, by utilizing restricted service regions with some overlap, departments can strike a balance between full flexibility and complete isolation, ensuring that officers can assist neighboring areas when needed without compromising the efficiency of their primary service areas.

Third, implementing flexible service regions with overlapping beats can help address the so-called ``border effect.'' For example, in Atlanta, we have observed a notable boundary effect in the distribution of police incidents, with crime rates tending to be higher at the peripheries of police beats (refer to Figure~\ref{fig:bordereffect} for a burglary case example). Additionally, response times at beat boundaries are often longer due to increased travel distances. Police officers also report that more criminal activity occurs near the borders of two beats, resulting in a higher volume of 911 calls in these areas. The boundary effect may contribute to increased crime rates near the edges of designated police zones, underscoring the need for a more efficient allocation of law enforcement resources. This observation suggests that the proposed overlapping beat structure may offer the demonstrated theoretical advantages and potential additional benefits in combating the border effect.

Motivated by these operational realities, this paper introduces a new queueing framework that generalizes the classical hypercube model to accommodate overlapping service regions, individual queue lengths, and flexible dispatch policies. While overlapping patrol strategies have gained traction in practice, there is currently no mathematical framework that systematically evaluates their performance. We emphasize that this study is problem-driven and grounded in real-world data from the Atlanta Police Department, which has implemented overlapping patrols in response to the operational pressures described above.
Accordingly, each generalization our model makes relative to the classical hypercube model reflects a feature of the real system rather than modeling generality for its own sake, e.g., per-server queues capture operations under sustained saturation, partially overlapping service regions capture mutual support restricted to designated neighboring beats, and state-dependent dispatch reflects the practice of assigning each call to a specific unit upon arrival. Our proposed model addresses this gap and offers the following core contributions:
\begin{enumerate}
    \item \emph{Hypercube queueing models for spatial service systems}: We introduce a state-space representation that captures per-server queue lengths and structured service overlaps in spatial service systems, allowing analysis of system dynamics under light and heavy traffic.
    \item \emph{Efficient steady-state computation for large-scale systems}: We introduce a hyperlattice structure that, despite its infinite state space, enables tractable and efficient computation through structured sparsity and state truncation techniques.
    \item \emph{Structural results of dispatch-policy design}: We exactly characterize the workload vectors attainable under all stable Markovian dispatch policies, establish a necessary and sufficient stabilizability condition together with a lattice-structured decomposition of capacity bottlenecks, and derive a closed-form water-filling characterization of the balance-optimal dispatch policy. These results show that, for workload- and flow-based objectives, searching over state-dependent dispatch rules is provably unnecessary (Section~\ref{sec:attainable-workloads}).
    \item \emph{Application to real-world patrol systems}: We validate the model using data from Atlanta Police Department operations, demonstrating how overlapping patrol regions can alleviate operational strain and reduce the adverse effects of boundaries. Our framework offers actionable insights into how resource allocation and dispatch policies influence system performance. In particular, applying the balance-optimal dispatch probabilities to the districting plans of Atlanta’s central Downtown–Midtown corridor (Zone~5) reduces the workload imbalance by about 20\% at no additional staffing cost, nearly matching a plan that uses one more police unit under uniform dispatch.
\end{enumerate}

\begin{figure}[!t]
\FIGURE
    {\includegraphics[width=0.7\textwidth]{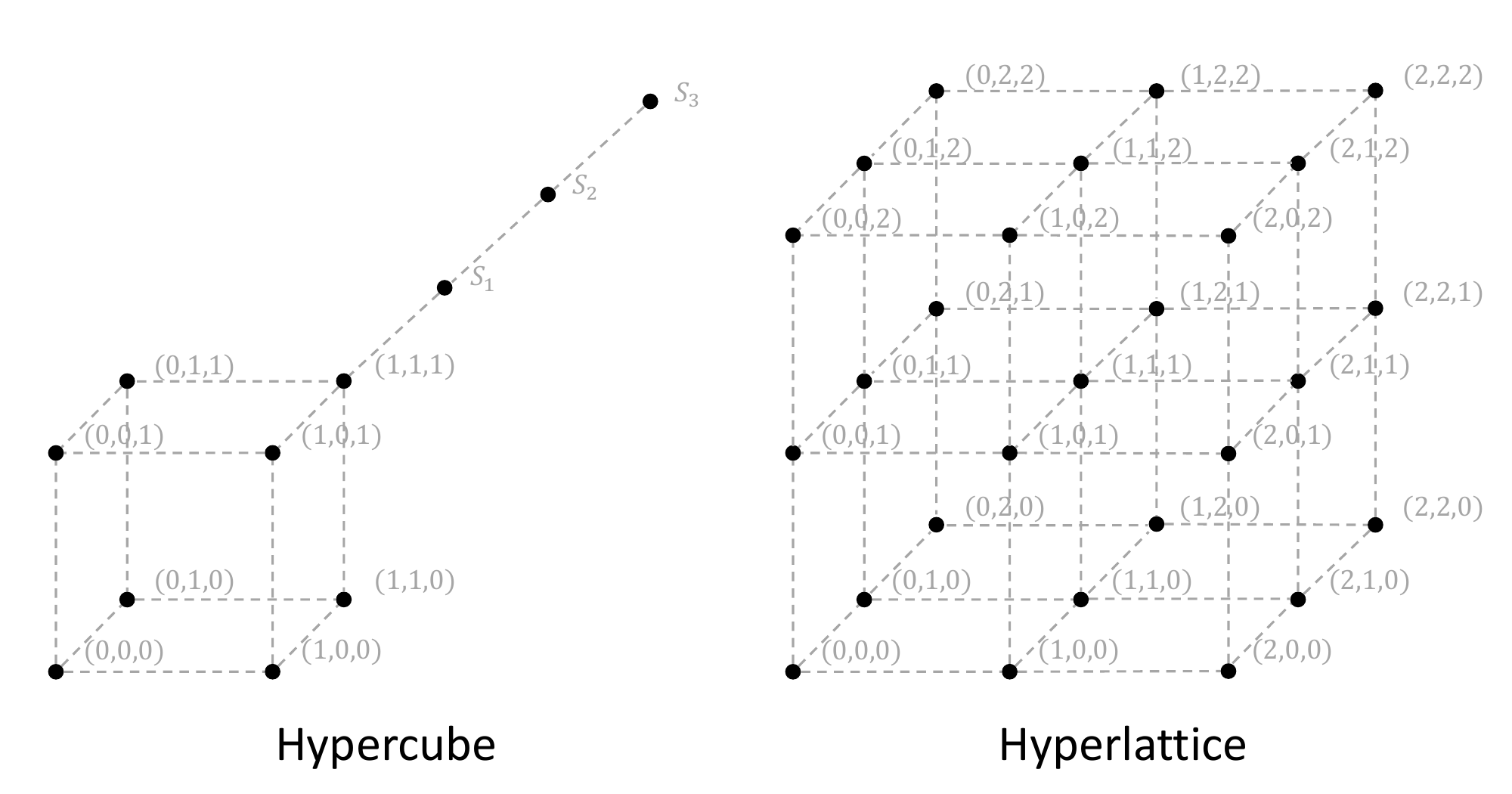}}
    {Illustration of the differences in state spaces. \label{fig:comparison}}
    {
    Our model represents the status of each server based on the number of calls assigned to its corresponding queue. In contrast, the hypercube model records only whether each server is idle or busy and reduces queueing beyond full saturation to a single tail state (i.e., a shared queue). As a result, our model's state space forms a {\em hyperlattice}, expanding the original finite-state hypercube structure into a countably infinite state space.
    }
    \label{fig:hyperlattice}
\end{figure}

To achieve our goal, we consider a Markov model with a potentially large state space represented by integer-valued vectors,
unlike the hypercube queueing model, which uses binary-valued state vectors as shown in Figure~\ref{fig:comparison}.
In this model, each server has its own queue, and the integer in the state represents the number of calls waiting in that queue.
Crucially, the policy that manages calls in each queue is determined by a hypergraph that reflects the overlapping nature of these regions.
The design of the state space results in a sparse transition matrix that allows for transitions between states whose vectors differ by 1 in the $\ell_1$ distance (in contrast, the hypercube queueing model considers Hamming distance one transition for binary-valued state vectors). Based on this, we develop an efficient state truncation technique to relate the generalized hypercube queueing model to the known properties of a general birth-and-death process with state-dependent rates. This reduces the detailed balancing equations to a much smaller linear system of equations, enabling efficient computation for steady-state distributions, which are key in evaluating general performance metrics as recognized in Larson's original work \citep{larson1974hypercube}.
We validate the proposed framework by testing it on both synthetic systems, with and without overlapping regions, and applying it to a real police system in collaboration with the Atlanta Police Department (APD).
With the adoption of the ``umbrella car'' strategy by the APD, where multiple officers collaborate to cover more areas due to understaffing and rising crime rates, our study aims to analyze this operational method and help the police evaluate its effectiveness using our model.
The results indicate that deploying more police units does not necessarily enhance policing effectiveness. Instead, overlapping police service regions provide a more flexible and dynamic response to varying call frequencies, significantly improving resource allocation efficiency compared to a rigid, non-overlapping structure.

This paper is organized as follows: We begin with a review of relevant literature on spatial queueing systems and Markov models with similar structures. In Section~\ref{sec:model}, we introduce a novel queueing framework for evaluating districting plans with overlapping regions. Section~\ref{sec:setup} discusses the problem setup and key assumptions. Section~\ref{sec:hyperlattice} presents a generalized hypercube queueing model -- the hyperlattice model -- for capturing system dynamics, while Section~\ref{sec:estimation} introduces an efficient method for estimating its steady-state distribution. Section~\ref{sec:performance-measure} introduces two key performance measures using the hyperlattice model. Section~\ref{sec:attainable-workloads} develops a flow-based structural theory of dispatch design, characterizing the attainable workloads, the stabilizability condition and its bottleneck decomposition, and the balance-optimal dispatch policy. Lastly, we demonstrate the effectiveness of our framework through experiments on synthetic systems in Section~\ref{sec:experiments} and conclude in Section~\ref{sec:discussions}, discussing potential extensions and generalizations.

\section{Related works}

{
Spatial queueing systems with overlapping service regions have received relatively limited attention in the literature. In this paper, we distinguish between \emph{completely overlapping} and \emph{partially overlapping} service regions. We refer to a system as completely overlapping when every call can be served by every server. By contrast, in a partially overlapping system, only certain subsets of servers (e.g., neighbors) jointly cover designated overlapping regions. For example, the overlapping zones studied in \cite{bammi1975allocation} can be viewed as completely overlapping service regions. Our focus is on partially overlapping service regions, which are more challenging to analyze but more closely reflect many spatial service systems, such as police patrol, where neighboring units may provide mutual support while maintaining primary geographic responsibilities.
}

Richard Larson's paper \citep{larson1974hypercube} and his co-authored book with Amedeo R. Odoni, {\em Urban Operations Research} \citep{larson1981urban}, introduced the hypercube queueing model, a pioneering work in this field. The hypercube queueing model provides a spatial queueing framework for analyzing and evaluating the performance of service systems. It permits servers to overlap into each other's territories and provide assistance when busy, making it particularly useful for emergency service systems, such as police and ambulance operations \citep{ansari2017maximum, de2015incorporating, takeda2007analysis, zhu2020data, zhu2022data, Morabito2008}.
However, it assumes identical service from all servers sharing a single queue, which may be oversimplified for real-world service systems and is unable to capture more complex systems. In our context, servers may have limited authorized movement, resulting in separate queues for each server. This is because other servers cannot travel to specific regions to offer assistance.
Several hypercube variants also consider systems in which multiple classes of servers may be eligible for the same call. For example, \cite{beojone2021efficient} proposed an extension of the hypercube model with dedicated and basic servers. In their setting, {\em server type} refers to heterogeneous operational roles: dedicated servers are assigned to the most severe calls, whereas basic servers handle less severe calls under the specified dispatch rules. A related distinction appears in \cite{de2015incorporating}, where Advanced Life Support (ALS) and Basic Life Support (BLS) ambulances have different responsibilities and call priorities. Thus, overlap in these models arises from heterogeneity in server type or call priority, rather than from neighboring same-type units sharing spatial responsibility. Our setting is different. We focus on spatial overlap among comparable patrol units: neighboring police units may jointly cover designated overlapping service regions, and eligible units have the same operational role for calls in those regions. Therefore, the key distinction in our model is not heterogeneity in server type, but partial overlap in spatial service responsibility. Moreover, \cite{takeda2007analysis} considered decentralized but dedicated service areas for ambulances, with each ambulance assigned to a specific station and region. In contrast, our hyperlattice model allows units to serve neighboring overlapping regions without requiring decentralization.

Among existing works extending the classic %is a paucity of research on generalizing the existing
hypercube queueing model for specific applications, a notable contribution is \cite{beojone2021efficient}, which introduces a variation of the hypercube model. Their extension is designed to represent a dispatch policy in which advanced-equipped servers serve solely life-threatening calls (called dedicated servers).
Our model, however, differs as it allows servers to operate within the territories of other servers, presenting a distinct problem framework.
Additionally, they grouped servers with similar characteristics in their state space representation, aiming to reduce computational demands. In our method, each server is still represented individually in the state space, managing to do so without significant computational expense. We also assume that each server has its own queue, diverging from most hypercube model variations, which typically use a single queue for all servers. A detailed comparison between various hypercube models and our hyperlattice model is presented in Table~\ref{tab:model_comparison}.

\setlength{\tabcolsep}{10pt}
\begin{table}[htbp]
\centering
\begingroup
\small
\renewcommand{\arraystretch}{1.25}
\caption{Comparison of Hyperlattice Model with Hypercube Models}
\label{tab:model_comparison}
\begin{tabular}{rlll}
\toprule
\textbf{Model Feature} & \textbf{\cite{larson1974hypercube}} & \textbf{\cite{beojone2021efficient}} & \textbf{Our Model} \\
\midrule
Overlap Structure & Completely overlapping & Completely overlapping & Partially overlapping \\
Call Priority  & Uniform           & Tiered (Urgent/Basic) & Uniform \\
Server Type & Homogeneous & Heterogeneous & Homogeneous \\
Dispatch Policy & Priority-based   & Priority-based         & Flexible (state-dependent) \\
Queueing Structure & Single loss queue & Single general queue & Multiple general queue \\
Tractability   & High              & Moderate               & High (via truncation) \\
\bottomrule
\end{tabular}
\endgroup
\end{table}

Recently, \cite{tsitsiklis2017flexible} introduced a flexible queue architecture for a multi-server model, comprising $n$ flexible servers and $n$ queues connected through a bipartite graph, with each server having its own associated queue. The paper primarily focuses on the model's theoretical aspects and highlights the key benefits of flexibility and resource pooling in such queueing networks. They specifically concentrate on the scaling regime where the system size $n$ tends to be infinite while the overall traffic intensity remains constant. However, their model cannot be directly extended to our scenario, as it does not account for the geographical relationships among servers.

The queueing literature extensively examines load-balancing models involving multiple dispatchers and flexibility constraints, often under the framework of flexible queues or load-balancing systems. A representative example is the {\em Join the Idle Queue} policy and related dispatching strategies designed to optimize system performance under varying workload conditions (see, e.g., \cite{cruise2020stability, weng2020optimal}). These models typically focus on analyzing optimal or near-optimal dispatching policies under simplified assumptions.
However, our contribution is primarily modeling-oriented:
We develop a generalized hypercube queueing framework that captures the spatial and structural complexity of real-world service systems. Unlike much of the load-balancing literature, which often assumes homogeneous servers, centralized coordination, or simplified system dynamics, our model explicitly incorporates server-level queues, spatially dependent service areas, and flexible, state-dependent dispatch policies.
Instead of focusing on specific policies, we develop a general and extensible framework that approximates system performance across a broad class of state-dependent dispatch policies, while maintaining computational tractability.

Several related studies have extended hypercube-type models to capture richer forms of spatial interaction among service units. \cite{atkinson2008hypercube} study a highway service system with customer-dependent service rates, where the service rate depends on both the demand location and the responding unit. Their model allows a call to be handled by a preferred or backup server, thereby capturing heterogeneous response capabilities within a loss-system framework. However, the state of each server records its availability and service status, rather than the number of calls assigned to that server; as a result, the model does not capture server-level waiting queues that may arise when service regions partially overlap. \cite{iannoni2009optimization} use a hypercube-based performance model \citep{larson1974hypercube} within an optimization framework for ambulance location and highway districting. In their setting, overlap enters through primary--backup response assignments and dispatch preference lists, and the hypercube model is used to evaluate candidate location and districting decisions. This differs from our focus, where overlapping service regions are part of the primitive queueing structure and directly shape the transition rates of the underlying CTMC. Similarly, \cite{boyaci2015approximation} propose approximation methods for large-scale spatial queueing systems by aggregating server states across graph-based partitions. Their framework is designed to make large-scale facility-location and spatial performance evaluation computationally tractable, and it captures cross-district responses through a finite-state loss model. However, it does not track queues at individual servers.

Our work is complementary to these studies. We model partially overlapping service regions through a {\em hyperlattice structure} (Figure~\ref{fig:hyperlattice}), in which each edge specifies a set of servers that can serve a shared demand region. The resulting CTMC has an integer-valued state variable for each server, allowing the model to capture both the job in service and any calls waiting in that server's queue. Moreover, dispatch decisions may depend on the full system state, so the transition rates can reflect flexible, state-dependent assignment rules.

\section{Proposed framework}
\label{sec:model}

This section develops a framework for assessing districting plans with overlapping patrol units for service systems.
To gain insight into the strategy of overlapping patrol, we consider a moving-server queueing system with multiple servers.
Typically, each server (\eg, a patrol service unit) is assigned to one particular service region, where the server has primary responsibility.
When a server is not busy serving any active calls, it traverses its home region to perform preventive patrol.
A key feature of our model is its consideration of a more flexible state-dependent policy, allowing servers to collaborate and jointly serve a third overlapping service region within a spatial queueing framework, as illustrated in Figure~\ref{fig:system-exps}.
For the reader's convenience, Table \ref{tab:notation} contains summary definitions of notations frequently used in the paper.

\subsection{Problem setup and key assumptions}
\label{sec:setup}

In this section, we describe our problem setup and several key assumptions.
Consider a service system patrolled by multiple servers, where each server is allocated to a specific area and has its own queue.
Unlike conventional service systems with separate service territories, our approach permits servers to have overlapping service areas.
The entire service system can be represented by an undirected \emph{hypergraph} \citep{bretto2013hypergraph} denoted by $(\mathcal{I}, \mathcal{E})$.
In this hypergraph, $\mathcal{I}=\{i = 1, \dots, I\}$ denotes the set of vertices, each representing an individual server, with $I$ being the total number of servers in the system.
The set $\mathcal{E}$ represents the hyperedges, where each hyperedge $e \in \mathcal{E}$ is a subset of $\mathcal{I}$, i.e., $\mathcal{E} \subseteq 2^\mathcal{I}$. Each hyperedge corresponds to a group of servers jointly overseeing an overlapping area. The cardinality of $\mathcal{E}$ is denoted by $E$.
Figure~\ref{fig:hypergraph-exps} presents the hypergraph of a three-server system as an example.

The service system encompasses a geographical space denoted by $\mathcal{S} \subset \mathbb{R}^2$. We denote the part of the space that server $i$ patrols as $\mathcal{S}_i \subseteq \mathcal{S}$.
We require that the service territories of all servers cover the entire space of the system, \ie, $\bigcup_{i \in \mathcal{I}} \mathcal{S}_i=\mathcal{S}$.
Regions jointly patrolled by a set of servers $e \in \mathcal{E}$, are termed as \emph{overlapping service regions}, and are denoted as $\mathcal{O}_{e}$.
Conversely, regions exclusively patrolled by server $i$ are termed \emph{primary service regions}, and are defined as $\mathcal{P}_i=\mathcal{S}_i \backslash \bigcup_{\{e: i \in e\}} \mathcal{O}_{e}$.
Note that it is important to guarantee $ e \not\subseteq e' $ for every $e, e' \in \mathcal{E}$ and $ e \neq e' $.
This ensures all primary and overlapping service regions are mutually exclusive and can be distinctly identified either by $i \in \mathcal{I}$ or by $e \in \mathcal{E}$.
Both primary and overlapping service regions can be arbitrarily small to avoid quantization errors, and they can have any geometric shape.
 Under this notation, a completely overlapping system is the special case in which every call is compatible with all servers, i.e., the relevant compatibility set is $\mathcal{I}$. The setting considered in this paper is more general and focuses on partial overlap, i.e., each hyperedge $e\in\mathcal{E}$ represents only a subset of servers that jointly cover a specific overlapping service region $\mathcal{O}_e$, while primary service regions $\mathcal{P}_i$ remain server-specific.

\begin{figure}[!t]
\FIGURE
    {\includegraphics[width=0.8\textwidth]{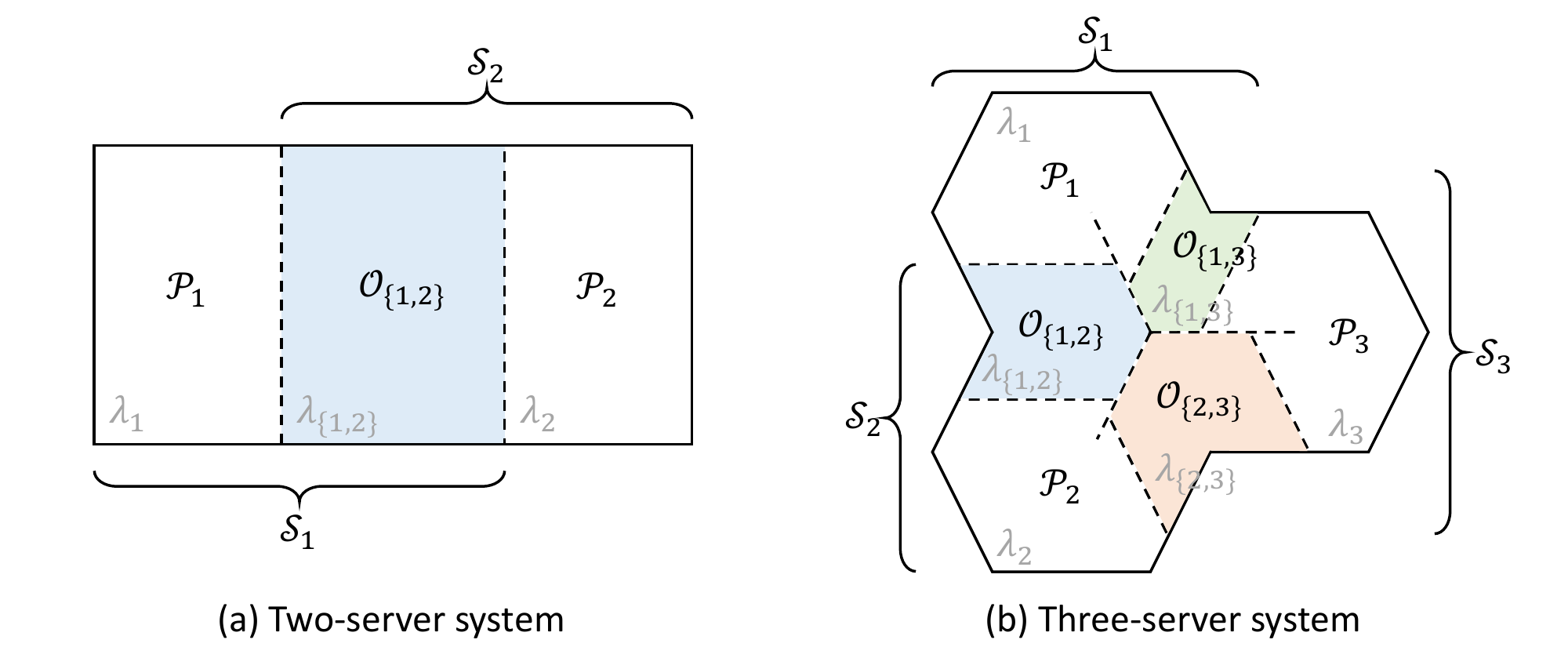}}
    {Examples of 2- and 3-server systems with overlapping patrols. \label{fig:system-exps}}
    {The white region $\mathcal{P}_i$ denotes the primary service area of server~$i$, which can only be served by that server. The shaded region $\mathcal{O}_{e}$ represents the overlapping service area $e = \{i,j\}$, which can be jointly served by servers~$i$ and~$j$.}
\end{figure}

\begin{figure}[!t]
\FIGURE
    {\includegraphics[width=0.6\textwidth]{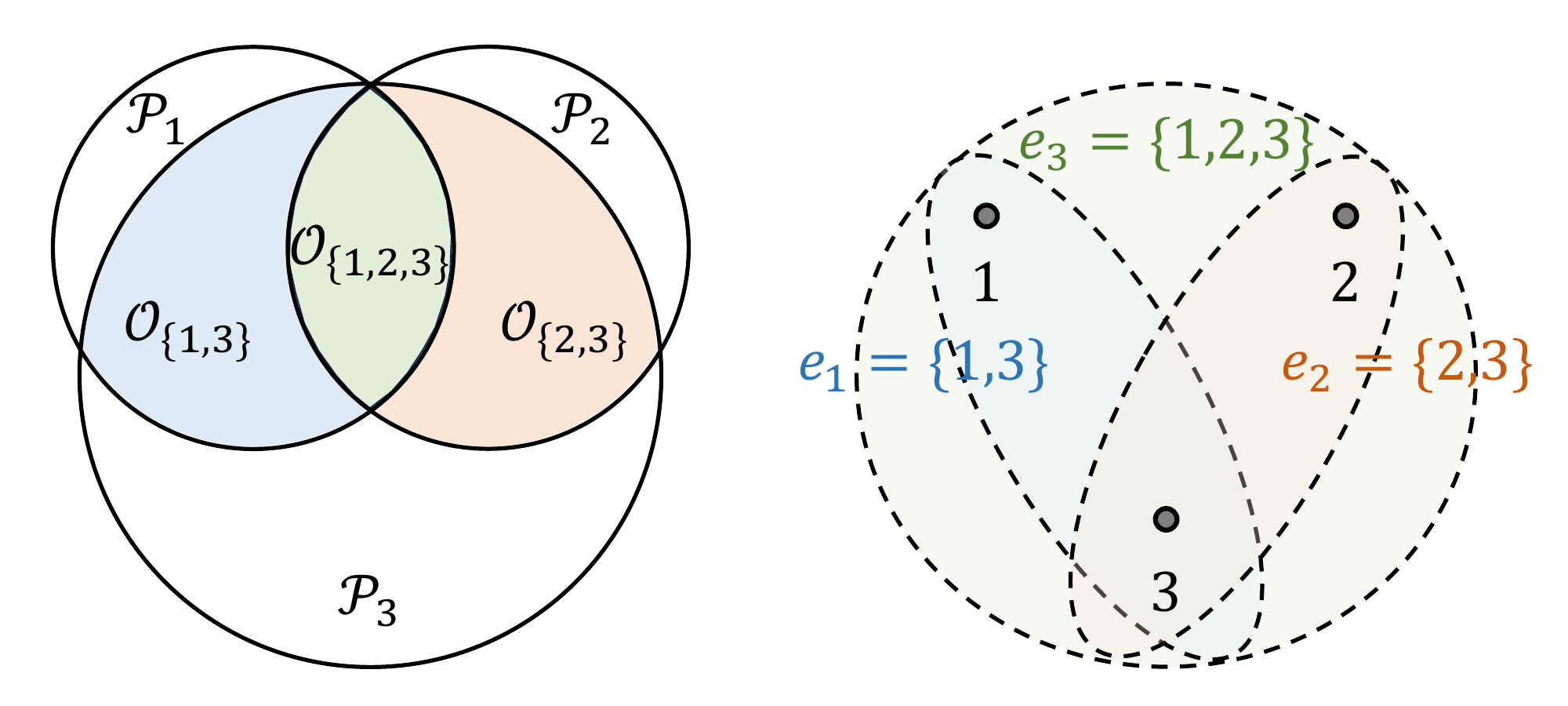}}
    {The hypergraph of a three-server system. \label{fig:hypergraph-exps}}
    {The black dots represent vertices and the shaded areas represent hyperedges.}
\end{figure}

Assume service calls arrive in the system according to a Poisson process with constant arrival rate. The arrivals of calls in primary service region $i$ and overlapping service region $e$ follow time-homogeneous Poisson processes with rate $\lambda_i$ and $\lambda_{e}$, respectively. Let $\lambda=\sum_{i \in \mathcal{I}} \lambda_i+\sum_{e \in \mathcal{E}} \lambda_{e}$ be the total call arrival rate in the entire region.

The service process is described as follows.
The server begins processing the call as soon as it arrives at the location where the call is reported.
The service time of each server is independent of the customer's location and the system's history.
The service time of server $i$ for any call for service has a negative exponential distribution with mean $1/\mu_i$.
Let $\mu = \sum_{i \in \mathcal{I}} \mu_i$ denote the total service rate in the entire region.
After completing the service, the servers return to or remain within their primary service territory, regardless of where the previous call originated.
For ergodicity, it is necessary that $\lambda < \mu$ should hold (a detailed stability condition is provided in Section~\ref{sec:CTMC_rep}).

\subsection{Hyperlattice queueing model for overlapping patrol}
\label{sec:hyperlattice}

We propose a generalized hypercube queueing model, called the \emph{hyperlattice} model, to capture the operational dynamics of systems with overlapping service regions.
The system state depends on the status of all the servers in this system, and the number of calls in each queue to be processed.

\subsubsection{State representation}
These states can be represented by a hyperlattice in dimension $I$. Each node of the hyperlattice corresponds to a state $B = (n_i)_{i\in\mathcal{I}}$ represented by a tuple of numbers, where a non-negative integer $n_i \in \mathbb{Z}_+$ indicates the status of server $i$.
Server $i$ is idle if $n_i = 0$ and busy if $n_i > 0$. The value of $n_i - 1$ represents the number of calls waiting in the queue of server $i$ when the server is busy.
Figure~\ref{fig:state-space-2d} (a) gives an example of the state space of a hyperlattice queueing model for a system with two servers.
It is worth noting that the state space in a hyperlattice queueing model can be divided into two parts, with one part consisting of states that possess at least one available server (referred to as \emph{non-saturated} states) and the other part comprising states in which all servers are busy (referred to as \emph{saturated} states), as shown in Figure~\ref{fig:two-types-states}.

\begin{figure}[!t]
\FIGURE
    {\includegraphics[width=\linewidth]{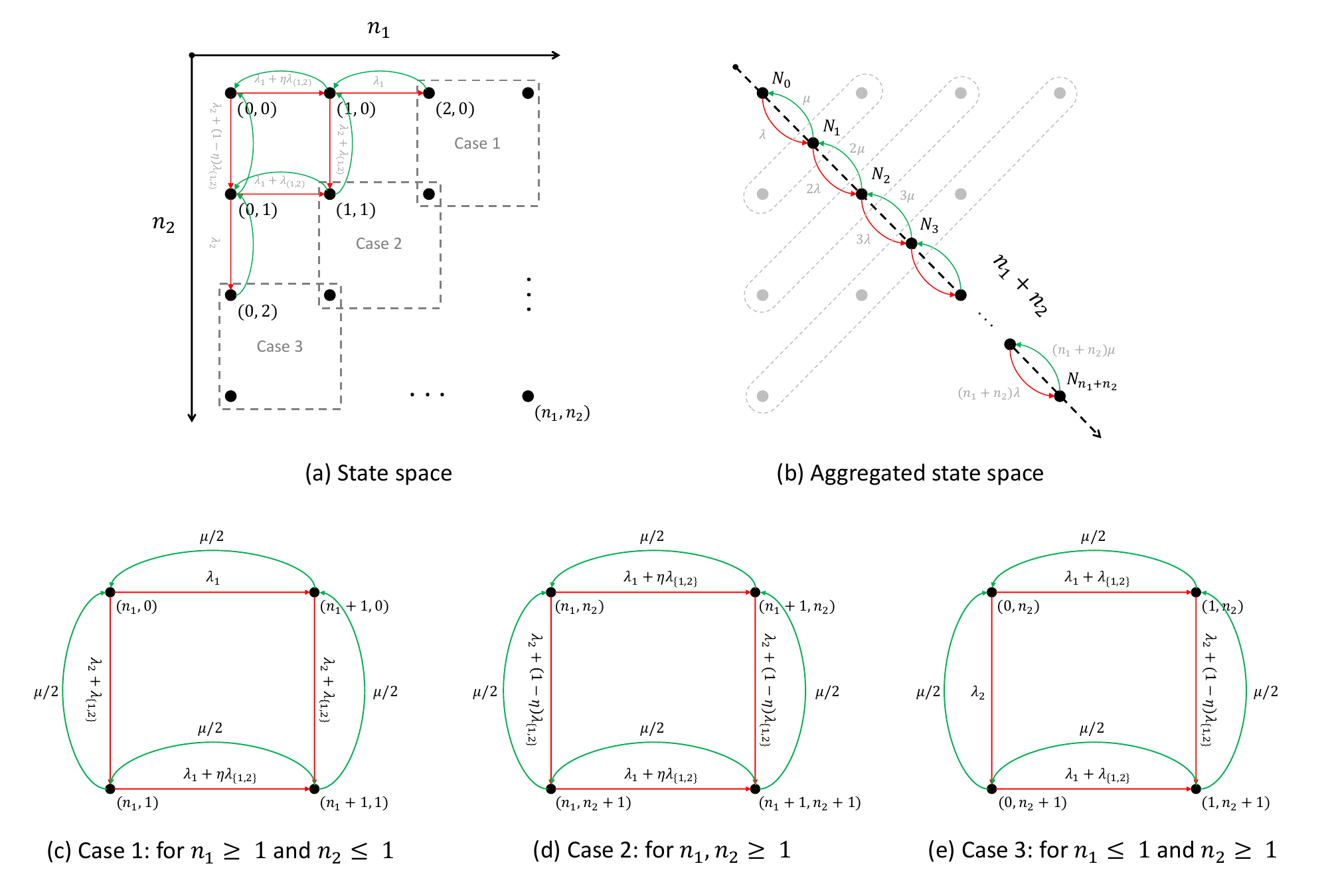}}
    {The state space in a hyperlattice queueing model for a system with two servers. \label{fig:state-space-2d}}
   {Panel (a) illustrates a two-server system with primary service regions $\mathcal{P}_1$ and $\mathcal{P}_2$ and an overlapping service region $\mathcal{O}_{\{1,2\}}$. Panel (b) shows the corresponding truncated hyperlattice state space, where each state is denoted by $(n_1,n_2)$; $n_i=0$ means server $i$ is idle, while $n_i>0$ means server $i$ is busy with one call in service and $n_i-1$ calls waiting in its queue. Green arrows indicate service-completion transitions and red arrows indicate arrival transitions. Only representative transitions are drawn for readability; the full transition-rate matrix and the corresponding upward and downward transition rates are defined in Section~\ref{sec:transition_matrix}. Panels (c)--(e) show representative local transitions under an idle-prioritizing dispatch policy. The labels $\mu/2$ apply only to the symmetric two-server illustration, where $\mu_1=\mu_2=\mu/2$, and do not imply that service capacity is split across regions (see Section~\ref{sec:transition_matrix}). The parameter $p\in[0,1]$ denotes the probability of assigning an overlapping-region arrival to server 1 when the random dispatch component applies; the corresponding probability for server 2 is $1-p$.}
\end{figure}

\begin{figure}[!t]
\FIGURE
    {\includegraphics[width=.8\linewidth]{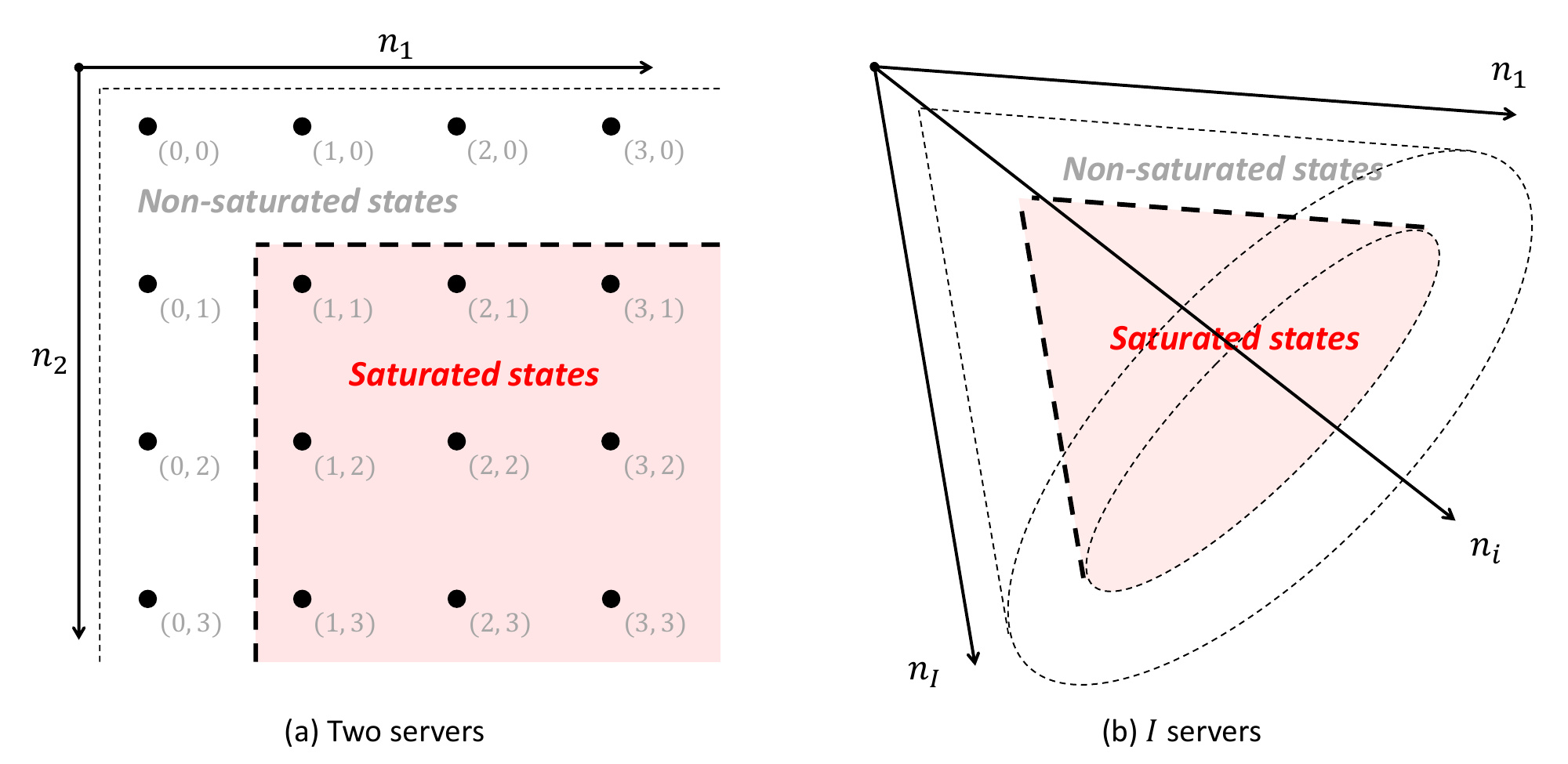}}
    {An illustration of two types of states in hyperlattice queueing models with different numbers of servers. \label{fig:two-types-states}}
    {The area colored red signifies the saturated states that have no idle servers, whereas the empty area represents the non-saturated states that have at least one server that is unoccupied.}
\end{figure}

\subsubsection{Touring algorithm}

To create the transition rate matrix, the states must be arranged in order. However, since the standard vector space lacks an inherent order, we employ a \emph{touring algorithm} to traverse the entire hyperlattice. This algorithm generates a sequence of $I$-digit non-negative integers $B_0, B_1, \dots$, which contains an infinite number of members and fully describes the order of the states.

\begin{figure}[!t]
\FIGURE
    {\includegraphics[width=\linewidth]{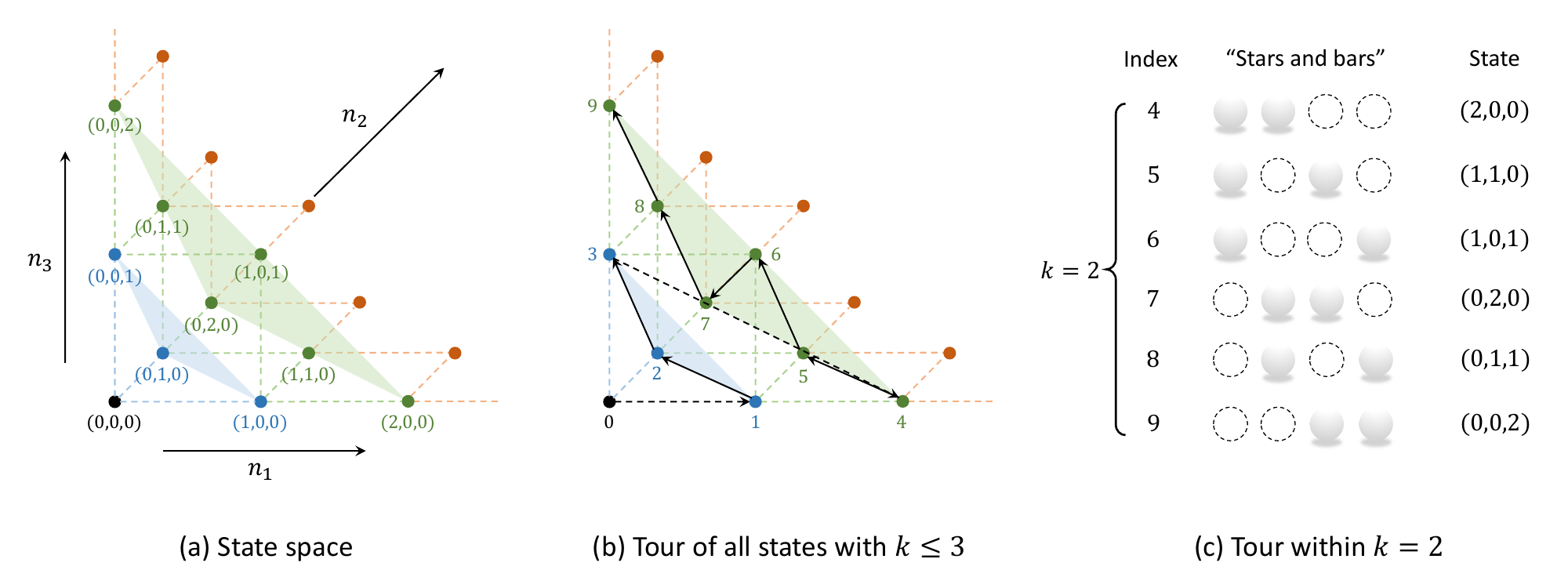}}
    {A hyperlattice queueing model for a system with three service regions. \label{fig:state-space-3d}}
    {(a) shows part of the state space of the hyperlattice for $k=1,2,3$; (b) presents the tour of all the states with $k \leq 3$, where solid arrows indicate transitions within states with the same number of units, while dashed arrows represent transitions between states with different numbers of units; (c) shows the sequence of all states within $k=2$, where the number of balls separated by dashed circles represents the status of server $i$ ($n_i$).}
\end{figure}

\begin{algorithm}[!t]
\SetAlgoLined
	{\bfseries Input:}
	Number of servers $I$;
	{\bfseries Output:}
        The sequence of ordered states $\mathcal{B}$;\\
	{\bfseries Initialization:}
        $\mathcal{B} \leftarrow \emptyset$\;
	\For{$k = 0, 1, \dots$}{
            $\mathcal{B}_k \leftarrow \{(n_i) \in \mathbb{Z}_+^I \mid \sum_{i\in\mathcal{I}} n_i = k\}$, where {$|\mathcal{B}_k|=\binom{k+I-1}{I-1}$};\\
            Sort $\mathcal{B}_k$ using an arbitrary order\;
            $\mathcal{B} \leftarrow (\mathcal{B}, \mathcal{B}_k)$\;
	}
\caption{Tour algorithm for a hyperlattice}
\label{algo:tour}
\end{algorithm}

We simplify the tour problem by decomposing it into $K$ sub-problems.
Each sub-problem $k \le K$ first identifies the states in which exactly $k$ calls are staying in the system (either waiting in the queue or being served), and then proceeds to solve the enumerative combinatorics problem by searching for $\binom{k+I-1}{I-1}$ possible combinations in a pre-determined sequence.
It is important to note that the sequence generated by each sub-problem is just one of numerous possible tours of the cutting plane in the hyperlattice.
Figure~\ref{fig:state-space-3d} presents a hyperlattice with $I=3$, along with its corresponding tour of all the states.
The combinatorics can be translated into the ``stars and bars'' \citep{flajolet2009analytic}, \ie, finding all possible positions of $I-1$ dashed circles that separate $k$ balls into $I$ groups, as illustrated in Figure~\ref{fig:state-space-3d} (c).
The number of balls being separated in group $i$ can be regarded as $n_i$.
Algorithm~\ref{algo:tour} summarizes the tour algorithm that generates the sequence $B_0, B_1, \dots$.
For convenience, we index the states based on the order produced by the tour algorithm. The state index is represented by $u \in \mathcal{U} = \{0, 1, 2, \dots\}$.

\subsubsection{Markovian dispatch policy}
\label{sec:dispatch_policy}
We define a class of dispatch policies within a general state-dependent framework.
Let $\mathcal{H}$ denote the set of all Markovian dispatch policies.
Our state-dependent dispatch policy, $\eta \in \mathcal{H}$, is defined as follows:
When a call arrives at location $s \in \mathcal{S}$, we first determine whether $s$ belongs to a primary service region or an overlapping service region. If $s \in \mathcal{P}_i$ for some server $i \in \mathcal{I}$, the call is routed directly to server $i$.
If $s$ belongs to an overlapping region, i.e., $s \in \mathcal{O}_e$ for some server set $e \in \mathcal{E}$, then the call can be served by any server in $e$. In this case, the dispatch policy is characterized by a discrete distribution $\eta_{e,u} = (\eta_{e,u}(i))_{i \in e}$, where $\eta_{e,u}(i)$ denotes the probability of assigning the call to server $i \in e$.
The corresponding dispatch policy is denoted as $\eta = \left\{\eta_{e, u} \mid e \in \mathcal{E}, u \in \mathcal{U}\right\}$. This broad class of state-dependent policies includes canonical examples such as ``Join-the-Shortest-Idle-Queue'' in \cite{weng2020optimal}.

\begin{figure}[!t]
\FIGURE
    {\includegraphics[width=\linewidth]{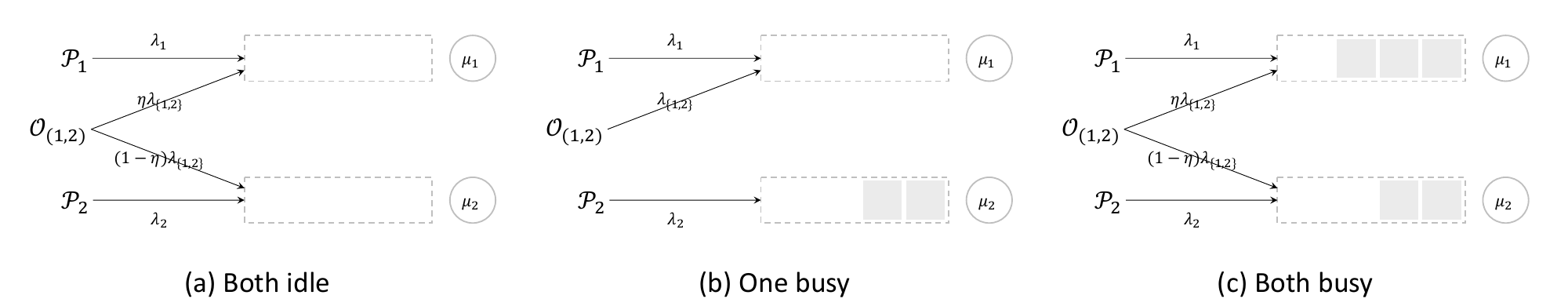}}
    {Queue networks with two servers and two queues. \label{fig:queue-networks}}
    {(a) both servers are idle, (b) one of the servers is busy, and (c) both servers are busy, where $p$ is the probability of assigning the call to server 1.}
\end{figure}

A simple example of such a dispatch policy for a two-server system with a single overlapping service region is illustrated in Figure~\ref{fig:queue-networks}, where the distribution $\eta_{e,u}$ is designed to prioritize the idle server. Specifically, in state $u$, if only one server is idle, the dispatch policy assigns the call exclusively to that server, i.e., $\eta_{e,u}(i) = 1$ for the idle server $i$. Otherwise, if both servers are either available or busy, the call is assigned randomly, with $\eta_{e,u}(1) = p$ and $\eta_{e,u}(2) = 1 - p$ for some fixed probability $p \in [0,1]$.

\subsubsection{Transition rate matrix}
\label{sec:transition_matrix}

Now we define the transition rate matrix for the hyperlattice queue $Q = (q_{uv})_{u,v\in\mathcal{U}}$, where $q_{uv}$ ($u \neq v$) denotes the transition rate departing from state $B_u$ to state $B_v$.
Let $u$ represent the index of state $(n_1, n_2, \dots, n_i, \dots, n_I)$ and let $d_{uv} = \norm{B_u - B_v}_1$ be the $\ell_1$ distance (or Manhattan distance) between two vertices $B_u$ and $B_v$ on the hyperlattice.
We define $u^{i+}$ as the index of state $(n_1, n_2, \dots, n_i + 1, \dots, n_I)$; note that thus we have $d_{uu^{i+}} = 1$ and $B_{u^{i+}} \succ B_u$ ($B_u$ is elementwise greater than $B_v$).
Diagonal entries $q_{uu}$ are defined such that $q_{uu} = - \sum_{v:v \neq u} q_{uv}$ and therefore the rows of the matrix sum to zero.
There are two classes of transitions on the hyperlattice: upward transitions that change a server’s status from idle to busy or add a new call to its queue; and downward transitions that do the reverse.
The downward transition rate from state $B_{u^{i+}}$ to its adjacent state $B_u$ is always $q_{u^{i+}u} = \mu_i$.
The upward transition rate, however, will depend on how regions overlap with each other and the dispatch rule when a call is received, which can be formally defined by the following proposition.

\begin{proposition} [Upward transition rate]
\label{prop:upward-transition}
For an arbitrary state index $u$, the upward transition rate from $B_u$ to $B_{u^{i+}}$ is:
\begin{equation}
q_{u u^{i+}}= \lambda_i+\sum_{e \in \mathcal{E}}
\eta_{e,u}(i) \lambda_{e},
\label{eq:upward-transition-rate}
\end{equation}
where
$\eta_{e,u}(\cdot)$ is a probability mass function that gives the probability of the new call in the overlapping service region $\mathcal{O}_{e}$ being received by server $i \in e$ at state $B_u$. The probability mass function $\eta_{e,u}$ has the following properties:
\begin{align}
    \eta_{e, u} (i) & = 0, &~ \forall  i \notin e, e \in \mathcal{E}, u\in\mathcal{U},\label{prob property i}\\
    \sum_{i \in \mathcal{I}} \eta_{e,u}(i) & = 1, &~ \forall e \in \mathcal{E}, u\in\mathcal{U}.\label{prob property ii}
\end{align}
\end{proposition}

The following argument justifies the proposition.
The first term of \eqref{eq:upward-transition-rate} suggests that, for an arbitrary server $i$, it must respond to a call received in its primary service region $i$.
The second term of \eqref{eq:upward-transition-rate} means that the server $i$ needs to respond to a call received in the overlapping service region
$e$ according to the dispatch policy $\eta_{e,u}$.
It is essential to note that the dispatch policy $\eta_{e,u}$ is state-dependent and can be defined based on the user's specific needs, providing a general framework that is suitable for elucidating both random and deterministic dispatch strategies.
For example,
Figure~\ref{fig:queue-networks} represents a specific case of the dispatch policy, in which
the assignment of server $i$ to the call received in overlapping service region $e$ can be discussed in two scenarios:
(a) When all servers $i \in e$ are busy ($u: n_i > 0, \forall~n_i \in B_u$), this call will join the queue of server $i$ based on the probability $\eta_{e,u}(i)$.
(b) Conversely, when there is at least one server idle ($u: n_i = 0, n_i \in B_u$), the server $i$ will be assigned to this call based on the probability $\eta_{e,u}(i)$.

\begin{remark}
The exogenous arrival processes associated with primary service regions and overlapping service regions are assumed to be independent Poisson processes. However, the arrival process assigned to an individual server is Poisson only {\em under a state-independent dispatch policy}. Specifically, if $\eta_{e,u}(i)=\eta_e(i)$ for all states $u\in\mathcal{U}$, then independent thinning and superposition imply that server $i$ receives a Poisson arrival stream with rate
$$
    \lambda_i+\sum_{e\in\mathcal{E}}\eta_e(i)\lambda_e .
$$
In contrast, under a state-dependent dispatch policy $\eta_{e,u}$, the assignment decision depends on the current CTMC state $B_u$. Therefore, the arrival stream assigned to server $i$ is generally not a Poisson process with a fixed rate. Rather, it is modulated by the current system state. The transition rate in \eqref{eq:upward-transition-rate} should therefore be understood as the instantaneous rate at which a new call is assigned to server $i$, conditional on the current state $B_u$, instead of as an assumption that each server receives an independent exogenous Poisson arrival stream.
\end{remark}

\subsubsection{ Markov chain representation}
\label{sec:CTMC_rep}
The hyperlattice model introduced above can be formulated as a continuous-time Markov chain (CTMC) on a discrete state space. Specifically, it is characterized by the state space $B$ and the transition rate matrix $Q$. The original state space is defined as $B  \coloneqq  \mathbb{Z}_{+}^{|\mathcal{I}|}$.
For convenience, we introduce a touring algorithm that establishes a one-to-one mapping from the high-dimensional state space $B$ to a one-dimensional index state space $\mathcal{U}  \coloneqq  \mathbb{Z}_{+}$. Additionally, let $e_i$ denote a vector with all elements zero except for the $i$-th entry, which is set to 1. The transition rate matrix $Q = (q_{uv})_{u,v\in\mathcal{U}}$ is then given by:
\begin{equation}\label{eq:trans_matrix}
q_{u v}=
\begin{cases}
\lambda_i+\sum_{e \in \mathcal{E}} \eta_{e, u}(i) \lambda_e & \text{if } B_v=B_u+e_i, \quad \forall i \in I, \\
\mu_i & \text{if } B_v=B_u-e_i, \left(B_u\right)_i>0, \quad \forall i \in I, \\
-\lambda-\sum_{i \in I:\left(B_u\right)_i>0} \mu_i & \text{if } B_v=B_u, \\
0 & \text{otherwise.}
\end{cases}
\end{equation}
For the hyperlattice queueing model to admit a steady-state distribution, the underlying CTMC must be positive recurrent. In Section~\ref{sec:attainable-workloads}, we characterize stability under the general Markovian dispatch policies introduced in Section~\ref{sec:dispatch_policy}.

\subsubsection{Balance equations}

To obtain the steady-state probabilities of the system, we can solve the balance equations given the transition rate matrix $Q$.
Let $\mathbb{P}\{B\}$ denote the probability that the system is occupying state $B$ under steady-state conditions.
The balance equations can be written as:
\begin{equation}
    \left ( \sum_{v: {B_v \succ B_u \atop d_{uv}=1}} q_{uv} + \sum_{v: {B_u \succ B_v \atop d_{uv}=1}} q_{uv} \right ) \mathbb{P}\{B_u\} =  \sum_{v: {B_u \succ B_v \atop d_{vu}=1}} q_{vu}  \mathbb{P}\{B_v\} + \sum_{v: {B_v \succ B_u \atop d_{vu} = 1}} q_{vu} \mathbb{P}\{B_v\},\quad  u \in \mathcal{U}.
    \label{eq:balance-equations}
\end{equation}
We also require that the probabilities sum to one, namely,
\[\sum_{u \in \mathcal{U}} \mathbb{P}\{B_u\} = 1.\]
However, one can observe that the above balance equations are hard to solve as
the number of states on the hyperlattice is infinite and grows exponentially with the number of servers, even when the queue capacity is limited.
Additionally, from Equation~\eqref{eq:trans_matrix}, the transition rates $q_{u,v}$ explicitly depend on the state-dependent dispatch policy $\eta_{e,u}$ for all $e \in \mathcal{E}$ and $u,v \in \mathcal{U}$. This state-dependent structure further complicates solving Equations~\eqref{eq:balance-equations}, rendering them computationally intractable.
To overcome this difficulty, we introduce an efficient approximation through a truncated hyperlattice model in the following section.

\subsection{A truncated hyperlattice model}
\label{sec:estimation}

We approximate the solution to Equations~\eqref{eq:balance-equations} by solving the balance equations for a truncated hyperlattice model, in which the total capacity across all servers is bounded by an upper bound $K$.
Specifically, instead of modeling each server as an infinite-capacity queue, we model the entire system analogously to an aggregated Erlang loss queue, i.e., the arrival rate in any region is set to zero whenever the total number of units in the system reaches $K$. In a light-traffic regime, setting the upper bound $K$ sufficiently large allows the truncated hyperlattice model to approximate the stationary dynamics of the original infinite-capacity queues closely. Figure~\ref{fig:state-space-2d}(b) illustrates an example of a 2-server overlapping queueing system truncated at $K=2$, with the original system shown in Figure~\ref{fig:state-space-2d}(a).

\subsubsection{ Approximate stationary distribution under truncation}
\label{sec:steady-state-computation}
Rather than attempting to calculate the stationary distribution across an infinite number of states in the original state space, the problem can be simplified by focusing only on a finite subset of states.
These states are selectively considered if their $k$ values are at most a user-defined hyperparameter $K$, i.e., $\norm{B}_1 \le K$.
In the finite truncated model, the stationary probabilities over these retained states are normalized to one.
It is worth noting that the stationary probability of the ``tail'' of the hyperlattice (states with large $k$ value) is usually low and can be negligible since their chance of happening decreases with the length of the queue. Specifically, in a light-traffic regime (e.g., when the arrival rate is significantly smaller than the service rate), the steady-state probability in the tail of the hyperlattice decreases fast as $K$ increases. As a result, the truncated queueing system provides a close approximation of the original system when $K$ is sufficiently large, where the stationary tail probability is negligible in the original system.
In contrast, in a high-traffic regime, the truncated hyperlattice model may diverge significantly from the original model with infinite capacity. However, for certain performance metrics, e.g., workload variance across servers, the truncated model can still serve as a reasonable approximation of the original model. The accuracy of this approximation in a high-traffic regime depends primarily on how sensitively the metric of interest responds to the system state. In the following experiments, we demonstrate that when focusing on server workload, the truncated model remains a good approximation of the simulation results.

We represent this set of truncated states by
$\mathcal{U}_K := \{u\in\mathcal U:\norm{B_u}_1\le K\}$, and let $U_K:=|\mathcal U_K|$ denote its cardinality. Equivalently, after reindexing the truncated states, one may write $\mathcal U_K=\{0,\dots,U_K-1\}$.
In the subsequent discussion, Lemma~\ref{lemma:set-size} presents the equation for $U_K$ in relation to $K$ and $I$.
This demonstrates that the size of the set can expand combinatorially relative to the number of servers $I$, and it can become quite substantial even for moderate values of $K$, particularly in systems with multiple servers.

\begin{lemma}
    The size of the set $\mathcal{U}_K$ is $U_K = \binom{K+I}{I}$.
    \label{lemma:set-size}
\end{lemma}
\proof{Proof of Lemma~\ref{lemma:set-size}}
    The size of the set $\mathcal{U}_K$ is equal to the summation $U_K = \sum_{k=0}^K \binom{k+I-1}{I-1}$, which can be simplified using the Hockey Stick identity \cite{jones1996generalized}:
    $$
        \sum_{k=n}^{m}\binom{k}{n}=\binom{m+1}{n+1}.
    $$
    Applying the identity with $n=I-1$ and $m=K+I-1$, we get:
    $$
        \sum_{k=0}^K \binom{k+I-1}{I-1} = \sum_{k=I-1}^{K+I-1} \binom{k}{I-1} = \binom{K+I}{I}.
    $$
    Therefore, the result of the sum is $\binom{K+I}{I}$.
\endproof

We next define the transition-rate matrix for the truncated model. Importantly, this matrix is not the principal submatrix of the infinite generator. Instead, the truncated system is treated as a finite-capacity loss system: when $\norm{B_u}_1=K$, new arrivals are blocked and do not create upward transitions.

Let $Q_K=(q^{(K)}_{uv})_{u,v\in\mathcal U_K}$ be the generator of this finite truncated CTMC, defined by
\begin{equation}
q^{(K)}_{uv}
=
\begin{cases}
\lambda_i+\sum_{e\in\mathcal E}\eta_{e,u}(i)\lambda_e,
    & \text{if } B_v=B_u+e_i \text{ and } \norm{B_u}_1<K, \\[1mm]
\mu_i,
    & \text{if } B_v=B_u-e_i \text{ and } (B_u)_i>0, \\[1mm]
-\displaystyle\sum_{\substack{w\in\mathcal U_K\\w\ne u}}q^{(K)}_{uw},
    & \text{if } v=u, \\[3mm]
0,
    & \text{otherwise.}
\end{cases}
\label{eq:truncated-generator}
\end{equation}
Thus, at boundary states $\norm{B_u}_1=K$, all upward transition rates are set to zero and the diagonal entry is recomputed from the remaining off-diagonal transition rates. Hence each row of $Q_K$ sums to zero, so $Q_K$ is a valid finite-state generator.

Denote the steady-state probabilities of the truncated loss model for the states in set $\mathcal{U}_K$ as $\pi_K = \{\mathbb{P}\{B_u\}\}_{u\in \mathcal{U}_K}$.
This distribution should be distinguished from the restriction of the original infinite-capacity stationary distribution to $\mathcal U_K$; $\pi_K$ is the normalized stationary distribution of the finite truncated CTMC.
We can express the balance equations in \eqref{eq:balance-equations} as a compact form:
\begin{equation}
    Q_K^\top \pi_K = 0,
    \label{eq:compact-balance-equations}
\end{equation}
where the right-hand side is a column vector of zeros.
The sum of probabilities in set $\mathcal{U}_K$ is equal to 1, i.e.,
\begin{equation}
\sum_{u\in \mathcal{U}_K} \mathbb{P}\{B_u\} = 1.
\label{eq:sum-trunc-steady-state-dist}
\end{equation}

Therefore, we can compute $\pi_K$ by solving \eqref{eq:compact-balance-equations} and \eqref{eq:sum-trunc-steady-state-dist}.
However, the $U_K$ balance equations are linearly dependent, and together with the normalization equation they form $U_K+1$ equations for $U_K$ unknown probabilities. To eliminate this redundancy, we replace one row of $Q_K^\top$ with a row of ones, usually the last row, and denote the modified coefficient matrix as $\tilde Q_K$. Similarly, we modify the ``solution'' vector, which was all zeros, to be a column vector with one in the last row and zeros elsewhere, and denote it as $e_K$. In practice, we solve the resulting equation:
$$
    \tilde Q_K \pi_K = e_K.
$$

\begin{figure}[!t]
\FIGURE
    {\includegraphics[width=\linewidth]{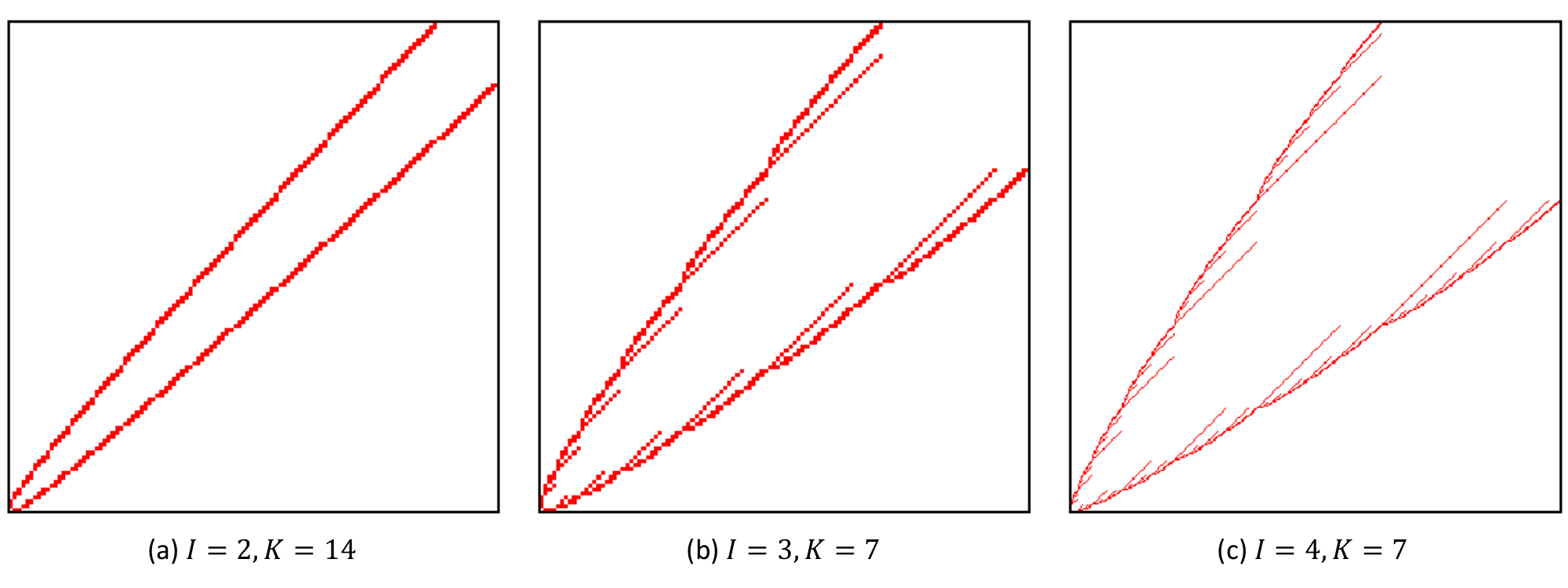}}
    {Support of the transition rate matrix $Q_K$ of hyperlattice queueing models (without diagonal entries). \label{fig:transition-rate-matrix}}
    {The shaded entries correspond to the matrices' non-zero entries, and the white spaces are zero. }
\end{figure}

We note that the solution to the above balance equations requires solving a linear system involving the modified matrix $\tilde Q_K$, which can become computationally challenging for large matrix sizes.
By Lemma~\ref{lemma:set-size}, the size of $Q_K$ can reach a staggering $U_K \approx (K+I)^I / I!$ (using Stirling's approximation \citep{romik2000stirling}) even for moderate values of $K$ when the system has $I>2$ servers, presenting a considerable computational hurdle.
Fortunately, the sparsity of transition rate matrices (as illustrated in Figure~\ref{fig:transition-rate-matrix}) can be leveraged to iteratively find the steady-state distribution using the power method \citep{nesterov2015finding}, by constructing a probability transition matrix of a discrete-time Markov chain through uniformization of the finite truncated continuous-time Markov chain.
Let $\mathbb{I}$ represent the identity matrix and define
$
    \nu_K := \max_{u\in\mathcal U_K}\left(-q^{(K)}_{uu}\right).
$
Choose any uniformization rate $\gamma>\nu_K$ and define
$
    P_K := \mathbb{I}+Q_K/\gamma.
$
Then $P_K$ is a stochastic matrix with strictly positive diagonal entries.
By initializing with a vector $\pi_K^{(0)}$, we can compute
$$
    \pi_K^{(t)} = P_K^\top \pi_K^{(t-1)}
$$
at each iteration $t=1,2,\dots$, until the distance between $\pi_K^{(t)}$ and $\pi_K^{(t-1)}$ falls below a specified threshold.
If the truncated CTMC is irreducible on $\mathcal U_K$, the positive diagonal entries make the uniformized chain aperiodic, and the power iteration converges to the unique stationary distribution $\pi_K$. If the truncated chain is not irreducible on the full truncated state space, the same computation should be applied to the relevant closed communicating class.

\subsubsection{Analysis of two canonical policies}
Although the stationary distribution of a general state-dependent policy is often intractable to derive, two canonical policies admit explicit characterizations: the \emph{Static Random-Routing policy} and the \emph{Join-the-Idle-Queue policy}. We analyze these policies under the truncated hyperlattice model and show that the truncated approximation provides a valid method in the light-traffic regime.
\paragraph{Static Random-Routing policy.}
Consider a state-independent static random-routing policy $\eta$. Specifically, recall that $\eta_e$ is a vector denoting the probability distribution over servers that determines how an arrival from region $e$ is routed. Then, the system is equivalent to $|\mathcal{I}|$ independent $M/M/1$ queues, where the arrival rate to server $i$ is
$\lambda_i + \sum_{e \in \mathcal{E}} \eta_e(i)\lambda_e$ and the service rate is $\mu_i$.
Assume the stability condition holds, i.e.,
$\lambda_i + \sum_{e \in \mathcal{E}} \eta_e(i)\lambda_e < \mu_i - \delta_i$ for each $i \in \mathcal{I}$ and some $\delta_i > 0$. Then, for server $i$, the stationary distribution of the number-in-system is the geometric distribution
$
\pi_i(n) = (1-\rho_i) \rho_i^n
$, where $\rho_i = \frac{\lambda_i + \sum_{e \in \mathcal{E}} \eta_e(i)\lambda_e}{\mu_i}$ is the traffic intensity at server $i$. Then, let $N_i$ be a random variable with distribution $\pi_i$, and define $S  \coloneqq  \sum_{i \in \mathcal{I}} N_i$ as the total number-in-system across all servers. The stationary probability of the truncated state space $\mathcal{U}_K$ is
$\mathbb{P}\!\left\{S \leq K \right\}$. The probability generating function of $N_i$ is
$$
G_i(z)=\mathbb{E}\left[z^{N_i}\right]=\frac{1-\rho_i}{1-\rho_i z}, \quad 0 \leq z<\frac{1}{\rho_i} .
$$
Since each server is independent, the probability generating function of $S$ is $\prod_{i\in\mathcal{I}} G_i(z)$. Then, for any $z \in\left(1,1 / \rho_{\max }\right)$ with $\rho_{\max } \coloneqq \max_{i\in\mathcal I}\rho_i < 1$, Markov’s inequality on $z^S$ yields
$$
\mathbb{P}\{S \geq K\}=\mathbb{P}\left\{z^S \geq z^K\right\} \leq z^{-K} \prod_{i \in \mathcal{I}} \frac{1-\rho_i}{1-\rho_i z} \leq C(z) e^{-K \ln z},
$$
where $C(z)  \coloneqq  \prod_{i \in \mathcal{I}} \frac{1-\rho_i}{1-\rho_i z}$ is a finite constant given $z$. Then, for any $z \in\left(1,1 / \rho_{\max }\right)$,
\begin{equation}\label{eq:bounded_tail}
    \sum_{u \in \mathcal{U}_K} \mathbb{P}\left\{B_u\right\} = \mathbb{P}\{S \leq K\}=1-\mathbb{P}\{S \geq K+1\} \geq 1-C(z) e^{-(K+1) \ln z} .
\end{equation}
Eq.~\ref{eq:bounded_tail} establishes that the stationary probability of the truncated state space converges to 1 exponentially fast in $K$.
As an illustrative example of a light-traffic two-server system with $\rho_1 = \rho_2 = 1/2$ and $z = 3/2$, we find that $C(z) = 4$, and
hence
$
\sum_{u \in \mathcal{U}_K} \mathbb{P}\{B_u\} \geq 1 - 4 e^{-(K+1)\ln(3/2)} = 1 - 4\left(\tfrac{2}{3}\right)^{K+1}.
$
In particular, when taking $K=10$, the stationary probability of the truncated state space $\mathcal{U}_K$ exceeds $0.95$. This example demonstrates that the truncated hyperlattice model introduced in Section~\ref{sec:steady-state-computation} provides a valid and accurate approximation in the light-traffic regime.

\paragraph{Join-the-Idle-Queue policy.}
Consider an idle-prioritizing variant of a static random-routing policy $\eta$. For an arrival from an overlapping service region $e\in\mathcal{E}$, let
$$
    \mathcal I_e(B_u) \coloneqq \{i\in e:(B_u)_i=0\}
$$
denote the set of eligible idle servers in state $B_u$. Primary-region arrivals are still routed to their primary server. For overlapping-region arrivals, the Join-the-Idle-Queue (JIQ) policy assigns the call to an eligible idle server whenever such a server exists; when no eligible server is idle, the policy uses the static random-routing distribution $\eta_e$. For example, one may use the state-dependent rule
$$
\eta^{\mathrm{JIQ}}_{e,u}(i)=
\begin{cases}
|\mathcal I_e(B_u)|^{-1},
    & i\in \mathcal I_e(B_u),\quad
      \varnothing\neq \mathcal I_e(B_u)\neq e,\\[2mm]
\eta_e(i),
    & i\in e,\quad
      \mathcal I_e(B_u)=e\ \text{or}\ \mathcal I_e(B_u)=\varnothing,\\[1mm]
0,  & \text{otherwise.}
\end{cases}
$$
Thus, the policy coincides with the static random-routing rule when all eligible servers in $e$ are either idle or busy, while prioritizing idle servers in the intermediate states.

Because the JIQ rule is state-dependent, the assignment process into each server is generally state-modulated rather than a Poisson process with a fixed rate. Moreover, in the heterogeneous or partially overlapping hyperlattice model, the aggregate process
$$
    S^{\mathrm{JIQ}}(t) \coloneqq \sum_{i\in\mathcal I} (B(t))_i
$$
is generally not itself a birth--death process, since its total departure rate sdepends on which servers are busy, not only on the total number of calls in the system. Therefore, a pathwise domination of $S^{\mathrm{JIQ}}(t)$ by the corresponding static random-routing process $S^{\eta}(t)$ should not be invoked without additional structural assumptions.

In settings where a valid monotone coupling can be established, e.g., in homogeneous fully flexible systems, or more generally under the hypotheses of an applicable stochastic-comparison theorem, the natural comparison is
$
    S^{\mathrm{JIQ}}\preceq_{\mathrm{st}}S^{\eta},
$
where the inequality is in the standard {\em stochastic order} (see, e.g., \cite{van2012stochastic}). In stationary form, this implies
$$
    \mathbb P\{S^{\mathrm{JIQ}}\ge K\}
    \le
    \mathbb P\{S^{\eta}\ge K\} \quad  \Longleftrightarrow \quad
    \mathbb P\{S^{\mathrm{JIQ}}\le K\}
    \ge
    \mathbb P\{S^{\eta}\le K\}.
$$
Consequently, whenever this stochastic comparison is justified, the exponential tail bound derived for the static random-routing policy also controls the JIQ tail, i.e,
$$
    \mathbb P\{S^{\mathrm{JIQ}}\le K\}
    \ge
    \mathbb P\{S^{\eta}\le K\}
    \ge
    1-C(z)e^{-(K+1)\ln z},
    \qquad
    z\in(1,1/\rho_{\max}).
$$
Under the stability condition, $\rho_{\max}=\max_{i\in\mathcal I}\rho_i<1$.
Thus, under the structural conditions that support the stochastic comparison, the truncated hyperlattice model remains accurate in the same light-traffic regime as the static random-routing benchmark. For the general heterogeneous or partially overlapping model, the accuracy of a finite truncation should be assessed directly from the stationary mass assigned to the boundary and tail states of the corresponding CTMC.

\subsubsection{Generalization of the hypercube model}

The hyperlattice queueing model serves as a natural generalization of Larson's hypercube model, introduced initially in \cite{larson1974hypercube}.
Specifically, when the entire region is treated as a single overlapping service region for all servers, i.e., $\mathcal{E}=\{\mathcal{I}\}$, the compatibility structure becomes fully flexible, as in the classical hypercube model.
To recover the original hypercube CTMC exactly, one must further specialize the state representation to record only binary busy/idle server statuses, together with the corresponding loss or shared-queue treatment used in the hypercube formulation.
In the hypercube model, the dispatch policy $\eta$ follows a ``fixed preferences'' rule, where the dispatcher always selects the highest-ranked available unit based on a predefined preference list, such as the closest-distance dispatching policy.
In contrast, the hyperlattice model keeps an integer-valued queue length for each server, so full compatibility alone does not collapse the model to the binary hypercube state space.
Thus, the hypercube model can be viewed as a special case obtained under additional state-space and queueing restrictions, while the hyperlattice model extends this framework to server-level queues and partially overlapping service regions.

Beyond these similarities, our hyperlattice queueing model extends the hypercube model in two fundamental ways.
First, it differentiates between overlapping and primary service regions, whereas the hypercube model assumes a single, fully overlapping region.
Second, it enables a more flexible, state-dependent dispatch policy that accounts for overlapping services.
In the experiments in \ref{sec:compare_hypercube}, we demonstrate that the hyperlattice queueing model provides a more accurate approximation than the hypercube model under complex queueing dynamics with overlapping services.

\subsection{Measures of performance}
\label{sec:performance-measure}

In this section, we explicate two key performance measures for a service system with overlapping service regions, using the proposed hyperlattice queueing model as a means of analysis.

\subsubsection{Individual workloads}\label{sec:workload}
The workload of server $i$, denoted by $\rho_i$, is simply equal to the fraction of time that server $i$ is busy serving calls. Thus $\rho_i$ is equal to the sum of the steady-state probabilities on part of the hyperlattice corresponding to $n_i > 0$, \ie,
\begin{equation}
    \rho_i = \sum_{u \in \mathcal{U}} \mathbbm{1}\{u: n_i > 0, n_i \in B_u\} \cdot \mathbb{P}\{B_u\},
    \label{eq:workload}
\end{equation}
where $\mathbbm{1}$ denotes indicator function and $n_i$ represents the status of server $i$ in state $B_u$.
In practice, because the efficient model estimation derived in Section~\ref{sec:estimation} only computes the steady-state probabilities for the states in $\mathcal{U}_K$, the workload of server $i$ is numerically approximated by
$$
\rho_i \approx \sum_{u \in \mathcal{U}_K} \mathbbm{1}\{u: n_i > 0, n_i \in B_u\} \cdot \mathbb{P}\{B_u\}.
$$
The individual workloads can be further used to calculate various types of system-wide workload imbalance defined in \cite{larson1974hypercube}.

\subsubsection{Fraction of dispatches} For the remainder of the system performance characteristics, it is necessary to compute the \textit{fraction of dispatches that send a server to each sub-region under its responsibility}. We use $\rho_{i,e}$ to represent the fraction of dispatches from sending server $i$ to one of its primary overlapping service regions $\mathcal{O}_{e}$. We have \begin{equation}\label{eq:fraction-dispatch}
\rho_{i,e} =
\begin{cases}\sum_{u \in \mathcal{U}} \eta_{e,u}(i) \lambda_{e} \mathbb{P}\left\{B_u\right\} / \lambda, & \text { (a) } i \in e, \\ 0, & \text { (b) } i \notin e,
\end{cases}
\end{equation} In addition, we use $\rho_{i,i}$ to represent the fraction of dispatches that send the server $i$ to its primary service region $i$, which is consistently equal to $ \lambda_i /\lambda$ since this region can exclusively be handled by server $i$. We note that $ \rho_{i,e} = 0$ if $i \notin e$ and $\sum_{i, e} \rho_{i,e} + \sum_i \rho_{i,i} = 1$. Additionally, in practice, we consider $u \in \mathcal{U}_{K}$ rather than $\mathcal{U}$ in \eqref{eq:fraction-dispatch} as we only calculate the steady-state probabilities for the first $U_K$ states according to Section~\ref{sec:estimation}.

\subsubsection{Travel time estimation}

The dispatch fractions can be used to derive travel-time metrics in the same spirit as the hypercube model of \cite{larson1974hypercube}.
For any server $ i \in \mathcal{I} $ and one of its primary overlapping service regions with server set $ e \in \mathcal{E} $, let $ t_{i,e} $ denote the mean travel time for sending server $ i $ to the service region $\mathcal{O}_{e}$, and let $ t_{i,i} $ denote the mean travel time for sending server $ i $ to its primary service region.
Using the infinite-capacity dispatch fractions in \eqref{eq:fraction-dispatch}, the unconditional mean travel time (MTT) is
\begin{equation}
\mathrm{MTT} = \sum_{i \in \mathcal{I}} \left( \sum_{e \in \mathcal{E}} \rho_{i,e} t_{i,e} + \rho_{i,i} t_{i,i} \right).
\end{equation}
Similarly, the mean travel time for each individual server $i\in \mathcal{I}$, denoted as $\tau_{i}$, can be estimated by
\begin{equation}
{\tau}_i=\frac{\sum_{e \in \mathcal{E}} \rho_{i,e} t_{i,e} + \rho_{i,i} t_{i,i}}{\sum_{e \in \mathcal{E}} \rho_{i,e} + \rho_{i,i}},
\end{equation}
and the mean travel time for each overlapping service region $e\in \mathcal{E}$, denoted as $\tau_{e}$, is given by
\begin{equation}
{\tau}_{e}=\frac{\sum_{i \in \mathcal{I}} \rho_{i,e} t_{i,e}}{\sum_{i \in \mathcal{I}} \rho_{i,e}}.
\end{equation}
In the truncated loss model, the same formulas are used after replacing $\rho_{i,e}$ and $\rho_{i,i}$ by the admitted-arrival-normalized quantities $\rho^{(K)}_{i,e}$ and $\rho^{(K)}_{i,i}$.
This post-hoc travel-time calculation treats $t_{i,e}$ and $t_{i,i}$ as exogenous inputs and does not embed travel dynamics directly into the CTMC.
This follows the standard hypercube modeling approach, where travel-time measures are evaluated after the stationary dispatch fractions have been computed.

\begin{remark}
\label{rem:flow-based-performance-metrics}
The three performance measures above, i.e., individual workloads, unconditional dispatch fractions, and the post-hoc travel-time metrics, are all determined by {\em long-run average flow quantities}. In particular, the dispatch fractions are normalized assignment rates, $\rho_{i,e}=x_{ei}/\lambda$, and the travel-time metrics are weighted averages of these fractions when the travel times $t_{i,e}$ and $t_{i,i}$ are treated as exogenous inputs. As we show in Section~\ref{sec:attainable-workloads}, stationary workloads are characterized by the same long-run assignment rates through rate conservation.
Therefore, objectives based only on these three types of metrics can be studied through the {\em static flow polytope}.
\end{remark}

\subsection{Flow-based characterization}
\label{sec:attainable-workloads}

We characterize the server workloads attainable in the hyperlattice model under the Markovian dispatch policies introduced in Section~\ref{sec:dispatch_policy} and the stability conditions. The result separates the first-order workload allocation problem from the higher-order queueing dynamics captured by the hyperlattice. In particular, we show that state-dependent dispatching can affect the joint queue-length distribution, but it does not enlarge the set of attainable stationary workload vectors.

\subsubsection{Attainable stationary workload region}
For a vector $z=(z_i)_{i\in\mathcal I}$ and a server set
$S\subseteq\mathcal I$, write
$
    z(S) \coloneqq \sum_{i\in S}z_i.
$
Define the demand that is \emph{trapped within} $S$ and the service capacity of
$S$, respectively, by
\begin{align}
    D(S)
    & \coloneqq \sum_{i\in S}\lambda_i
      +\sum_{e\in\mathcal E:\,e\subseteq S}\lambda_e,
      \label{eq:trapped-demand}\\
    M(S)
    & \coloneqq \sum_{i\in S}\mu_i.
      \label{eq:subset-capacity}
\end{align}
The first term in $D(S)$ consists of the calls from primary service regions in $S$. The second term consists of calls from overlapping service regions whose entire compatibility set lies in $S$. We note that none of these calls can be assigned to a server outside $S$.

\paragraph{Long-run flow conservation under stable policies.}
Let $\mathcal H_{\mathrm{st}}\subseteq\mathcal H$ denote the set of Markovian dispatch policies under which the hyperlattice CTMC \eqref{eq:trans_matrix} is positive recurrent. For $\eta\in\mathcal H_{\mathrm{st}}$, let $\pi^\eta=(\pi_u^\eta)_{u\in\mathcal U}$ be its stationary distribution and let
$
    \rho_i^\eta
     \coloneqq \sum_{u\in\mathcal U}
      \mathbbm{1}\{(B_u)_i>0\}\pi_u^\eta
$
be the stationary workload \eqref{eq:workload} of server $i$. For every overlapping service region $e\in\mathcal E$ and eligible server $i\in e$, define the stationary assignment rate
$$
    x_{ei}^\eta
     \coloneqq \lambda_e\sum_{u\in\mathcal U}
       \pi_u^\eta\eta_{e,u}(i).
$$
In particular, $x_{ei}^\eta$ is the long-run rate at which calls arriving in $\mathcal O_e$ are assigned to server $i$.
Stationarity then implies the following flow-conservation property.

\begin{lemma}
\label{lemma:stationary-rate-conservation}
For every dispatch policy $\eta\in\mathcal H_{\mathrm{st}}$,
\begin{align}
    \sum_{i\in e}x_{ei}^\eta&=\lambda_e,
      && e\in\mathcal E,
      \label{eq:edge-flow-conservation}\\
    \mu_i\rho_i^\eta
      &=\lambda_i+\sum_{e\in\mathcal E:\,i\in e}x_{ei}^\eta,
      && i\in\mathcal I.
      \label{eq:server-rate-conservation}
\end{align}
\end{lemma}
Lemma~\ref{lemma:stationary-rate-conservation} establishes that, under any dispatch policy $\eta\in\mathcal H_{\mathrm{st}}$ that renders the system stationary, the long-run arrival and service rates satisfy the corresponding flow-conservation equations.

\paragraph{Hypergraph-flow characterization.}
Define the hyperedge-flow polytope
$$
    \mathcal X
     \coloneqq \left\{
       x=(x_{ei})_{e\in\mathcal E,\,i\in e}:
       x_{ei}\ge0,\quad
       \sum_{i\in e}x_{ei}=\lambda_e
       \ \forall e\in\mathcal E
       \right\}.
$$
For $x\in\mathcal X$, define the induced long-run average arrival rate at server $i$ by
$$
    a_i(x)
     \coloneqq \lambda_i+\sum_{e\in\mathcal E:\,i\in e}x_{ei}.
$$
The polytope $\mathcal X$ describes all feasible long-run allocations of overlapping-region demand to eligible servers. For each hyperedge $e$, the variable $x_{ei}$ represents the rate at which calls from $\mathcal O_e$ are assigned to server $i\in e$, and the constraint $\sum_{i\in e}x_{ei}=\lambda_e$ ensures that all demand from that overlapping region is allocated. Once such an allocation is fixed, $a_i(x)$ is the total arrival rate induced at server $i$, consisting of its primary demand and the portion of overlapping demand assigned to it. The next lemma characterizes exactly which server-level arrival-rate vectors can arise from such feasible hyperedge flows. The characterization is expressed through {\em cut constraints}: every server set $S$ must receive at least the demand trapped inside $S$, while equivalently it cannot receive demand from regions that have no compatible server in $S$.

\begin{lemma}
\label{lemma:hypergraph-flow-characterization}
A vector $a=(a_i)_{i\in\mathcal I}\in\mathbb R_+^I$ can be represented as $a_i=a_i(x)$ for some $x\in\mathcal X$ iff
\begin{align}
    a(\mathcal I)&=\lambda,
      \label{eq:attainable-total-rate}\\
    a(S)&\ge D(S),
      \quad S\subseteq\mathcal I.
      \label{eq:attainable-lower-cuts}
\end{align}
In particular, define
\begin{equation}
    F(S)
     \coloneqq \sum_{i\in S}\lambda_i
      +\sum_{e\in\mathcal E:\,e\cap S\ne\varnothing}\lambda_e.
    \label{eq:reachable-demand-function}
\end{equation}
Then the same set of vectors is
\begin{equation}
    \left\{
      a\in\mathbb R_+^I:
      a(\mathcal I)=\lambda,
      a(S)\le F(S)\ \forall S\subseteq\mathcal I
    \right\}.
    \label{eq:attainable-upper-cuts}
\end{equation}
\end{lemma}

Lemma~\ref{lemma:hypergraph-flow-characterization} characterizes feasible server-level arrival rates purely in terms of hypergraph flow constraints. Together with Lemma~\ref{lemma:stationary-rate-conservation}, it implies that every stable Markovian dispatch policy induces a feasible hyperedge flow $x^\eta\in\mathcal X$ such that
$
\mu_i\rho_i^\eta = a_i(x^\eta)
$
for all $i\in\mathcal I$.
Thus, any dynamically attainable workload vector must arise from a feasible arrival-rate vector after the change of variables $a_i=\mu_i\rho_i$. Conversely, any feasible arrival-rate vector with $a_i<\mu_i$ can be implemented by a state-independent randomized dispatch policy, yielding independent stable $M/M/1$ queues with workloads $\rho_i=a_i/\mu_i$. The following theorem translates this flow characterization into an exact description of the attainable workload region.

\begin{theorem}
\label{thm:attainable-workload}
The set of stationary workload vectors attainable under stable Markovian dispatch policies is
\begin{align}
    \mathcal R_{\mathrm{st}}
    & \coloneqq \left\{
       (\rho_i^\eta)_{i\in\mathcal I}:
       \eta\in\mathcal H_{\mathrm{st}}
       \right\}=\left\{
       \rho\in[0,1)^I:
       \sum_{i\in\mathcal I}\mu_i\rho_i=\lambda,\quad
       \sum_{i\in S}\mu_i\rho_i\ge D(S)
       \ \forall S\subseteq\mathcal I
       \right\}.
    \label{eq:attainable-workload-region}
\end{align}
Moreover, for every state-dependent policy $\eta\in\mathcal H_{\mathrm{st}}$, there exists a stable state-independent randomized policy $\bar\eta$ that has the same workload vector and the same long-run assignment rates $x_{ei}^\eta$.
\end{theorem}

\begin{corollary}
\label{cor:static-sufficiency-workload}
Let $\Phi$ be any objective that depends only on the stationary workloads and the long-run overlapping-region assignment rates. Then
\begin{align}
    \inf_{\eta\in\mathcal H_{\mathrm{st}}}
      \Phi\!\left(\rho^\eta,x^\eta\right)
    =
    \inf_{\substack{x\in\mathcal X:\\a_i(x)<\mu_i\ \forall i}}
      \Phi\!\left(
        \left(\frac{a_i(x)}{\mu_i}\right)_{i\in\mathcal I},x
      \right).
    \label{eq:static-sufficiency-objective}
\end{align}
In particular, state dependence cannot improve an objective based only on the individual workloads, their variance, or the unconditional dispatch fractions.
\end{corollary}
Corollary~\ref{cor:static-sufficiency-workload} gives a concrete policy-structure implication. If the objective only uses average workload information, then state dependence adds no value, i.e., the optimal value can be attained by a state-independent randomized policy, or equivalently, by solving the corresponding static flow problem over $\mathcal X$. State-dependent policies remain important only for performance metrics that depend on the full stationary queue-length distribution rather than on its average flow rates alone.

\subsubsection{Stabilizability and bottleneck server sets}
We now use the attainable workload characterization to identify when the system can be stabilized and which server sets form the capacity bottlenecks. The key quantity is the smallest possible peak workload across servers. The next theorem shows that this peak workload is exactly determined by the worst trapped-demand-to-capacity ratio over all server subsets. As a result, stability is equivalent to requiring every nonempty server set to have strictly more service capacity than the demand trapped inside it.
\begin{theorem}
\label{thm:bottleneck-stability}
Define the minimum achievable peak workload by
\begin{equation}
    \theta^*
     \coloneqq \min_{x\in\mathcal X}
      \max_{i\in\mathcal I}\frac{a_i(x)}{\mu_i}.
    \label{eq:min-peak-workload}
\end{equation}
Then
\begin{equation}
    \theta^*
    =\max_{\varnothing\ne S\subseteq\mathcal I}
      \frac{D(S)}{M(S)}.
    \label{eq:bottleneck-ratio}
\end{equation}
Consequently, the following statements are equivalent:
\begin{enumerate}
    \item there exists a stable Markovian dispatch policy;
    \item there exists a stable state-independent randomized dispatch policy;
    \item $\theta^*<1$;
    \item $D(S)<M(S)$ for every nonempty $S\subseteq\mathcal I$.
\end{enumerate}
\end{theorem}

The ratio $D(S)/M(S)$ provides a natural measure of how constrained a server set $S$ is: it compares the demand that must be served within $S$ against the total service capacity available in $S$. Server sets that attain the largest such ratio determine the peak workload $\theta^*$ and identify the parts of the system that constrain stability or balanced workload allocation. We formalize these sets below.

\begin{definition}[Bottleneck server sets] \label{def:bottleneck-server-set}
Let
\begin{equation}
\mathcal B^*  \coloneqq  \left\{ S\subseteq\mathcal I: D(S)=\theta^*M(S) \right\}, \label{eq:bottleneck-set-family}
\end{equation}
where the empty set is included by convention. A nonempty set $S^*\in\mathcal B^*$ is called a \emph{bottleneck server set}.
Equivalently, a bottleneck server set is a nonempty server subset whose trapped demand-to-capacity ratio attains the optimal peak workload, i.e., $D(S^*) = \theta^* M(S^*)$.
\end{definition}

\begin{corollary}\label{cor:bottleneck-structure-lattice}
Let $S^*$ be a bottleneck server set and let $x^*$ be any optimal flow in \eqref{eq:min-peak-workload}. Then
\begin{align}
    a_i(x^*)&=\theta^*\mu_i,
      \quad i\in S^*,
      \label{eq:bottleneck-server-saturation}\\
    x_{ei}^*&=0,
      \qquad e\not\subseteq S^*,\quad i\in e\cap S^*.
      \label{eq:no-crossing-flow-into-bottleneck}
\end{align}
Moreover, the set $\mathcal B^*$ is closed under union and intersection, i.e., if $S,T\in\mathcal B^*$, then
$$
    S\cup T\in\mathcal B^*,
    \qquad
    S\cap T\in\mathcal B^* .
$$
Thus, the collection of bottleneck server sets has a unique maximal member, i.e., the union of all bottleneck server sets.
\end{corollary}

The corollary has several direct operational implications. First, every server inside a bottleneck set must be loaded exactly up to the optimal peak workload $\theta^*$; otherwise, the trapped demand inside the set could not be fully served without exceeding the peak workload elsewhere in the set. Second, calls from an overlapping region that has at least one eligible server outside the bottleneck are not routed into the bottleneck under any peak-workload optimal allocation. Thus, the limited capacity inside a bottleneck set is reserved for calls that cannot be served outside that set. Finally, the lattice property shows that bottlenecks have a coherent spatial structure: overlapping bottleneck sets can be merged without losing bottleneck status, and their common intersection is also bottlenecked whenever it is nonempty. Operationally, the maximal bottleneck set identifies the largest part of the system whose internal demand is capacity-limiting. This set is a natural target for interventions such as adding service capacity, redesigning overlap boundaries, or allowing additional servers outside the set to help with calls currently trapped inside the bottleneck.

\subsubsection{Balance-optimal state-independent policies}
\label{sec:balance-optimal-policy}

Corollary~\ref{cor:static-sufficiency-workload} shows that when the objective depends only on the stationary workloads, the search for an optimal dispatch policy can be restricted to state-independent randomized policies, or equivalently to the flow polytope $\mathcal X$. We now characterize the optimal policy explicitly for workload-balance objectives such as the workload standard deviation used in Section~\ref{sec:experiments}. For clarity, we state the result for homogeneous service rates, $\mu_i=\mu$ for all $i\in\mathcal I$.

\paragraph{Recursive bottleneck decomposition.}
The characterization applies the bottleneck decomposition of Theorem~\ref{thm:bottleneck-stability} recursively. Let $U_1\coloneqq\mathcal I$ and $D_1\coloneqq D$. For $k=1,2,\dots$, define
\begin{equation}
    \theta_k
    \coloneqq \max_{\varnothing\ne S\subseteq U_k}
      \frac{D_k(S)}{M(S)},
    \qquad
    S_k \coloneqq \text{the maximal bottleneck set of the system } (U_k, D_k),
    \label{eq:water-filling-recursion}
\end{equation}
and contract the system by setting $U_{k+1}\coloneqq U_k\setminus S_k$ and $D_{k+1}(S)\coloneqq D_k(S\cup S_k)-D_k(S_k)$ for $S\subseteq U_{k+1}$. The residual function $D_{k+1}$ is again a trapped-demand function, i.e., it consists of the primary demand in $S$ together with all overlapping-region demand whose remaining eligible servers all lie in $S$ (in the proof, we show that no balance-optimal allocation routes such demand into the earlier, more heavily loaded blocks). The maximal bottleneck set at each level is well defined by the lattice property of Corollary~\ref{cor:bottleneck-structure-lattice} applied to the residual system. The recursion terminates after $m\le I$ steps with a partition $\mathcal I=S_1\cup\dots\cup S_m$. We refer to \eqref{eq:water-filling-recursion} as the \emph{water-filling decomposition} and to $\rho^\ast$ as the \emph{water-filling workload vector}, where $\rho^\ast_i\coloneqq\theta_k$ for $i\in S_k$.

\begin{proposition}[Balance-optimal policy]
\label{cor:water-filling}
Suppose $\mu_i=\mu$ for all $i\in\mathcal I$ and $\theta^*<1$. For $x\in\mathcal X$, write $\rho_i(x)\coloneqq a_i(x)/\mu$. Then:
\begin{enumerate}
    \item The water-filling levels satisfy $1>\theta_1>\theta_2>\dots>\theta_m$, where $\theta_1=\theta^*$ and $S_1$ is the maximal bottleneck set of Definition~\ref{def:bottleneck-server-set}, and $\rho^\ast$ is attainable by a stable state-independent randomized policy.
    \item $\rho^\ast$ is the unique minimizer of the workload variance, equivalently the workload standard deviation, over the attainable region $\mathcal R_{\mathrm{st}}$. In particular,
    $
        \min_{\eta\in\mathcal H_{\mathrm{st}}}
        \mathrm{sd}\left(\rho^\eta\right)
        =\mathrm{sd}\left(\rho^\ast\right),
    $
    and the minimum is attained by a state-independent randomized policy.
    \item A flow $x\in\mathcal X$ attains $\rho^\ast$ if and only if every overlapping region routes only to its least-loaded eligible servers, i.e.,
    \begin{equation}
        x_{ei}>0
        \;\Longrightarrow\;
        \rho_i(x)=\min_{j\in e}\rho_j(x),
        \qquad
        \forall e\in\mathcal E,\ i\in e.
        \label{eq:balance-equilibrium}
    \end{equation}
    Any such flow induces a balance-optimal policy $\bar\eta_e(i)=x_{ei}/\lambda_e$.
\end{enumerate}
\end{proposition}

The proposition answers the dispatch-design question for workload-balance objectives in closed structural form. The equilibrium condition \eqref{eq:balance-equilibrium} is directly interpretable for dispatchers, i.e., \emph{in the long-run average, no overlapping region should send demand to an eligible server that is busier than another of its eligible servers}. The water-filling structure shows what this equilibrium produces: workloads are equalized within nested blocks whose load levels are exactly the successive bottleneck ratios, so the most constrained part of the system operates at $\theta^*$, the next most constrained at $\theta_2$, and so on. In addition, the optimal workload vector is unique, although the optimal flow and the optimal policy need not be. Computationally, the decomposition requires at most $I$ maximum-flow computations, one per level, reusing the network of Lemma~\ref{lemma:hypergraph-flow-characterization}. We illustrate this water-filling structure with an explicit example below.

\begin{figure}[!t]
\FIGURE
    {\includegraphics[width=.9\linewidth]{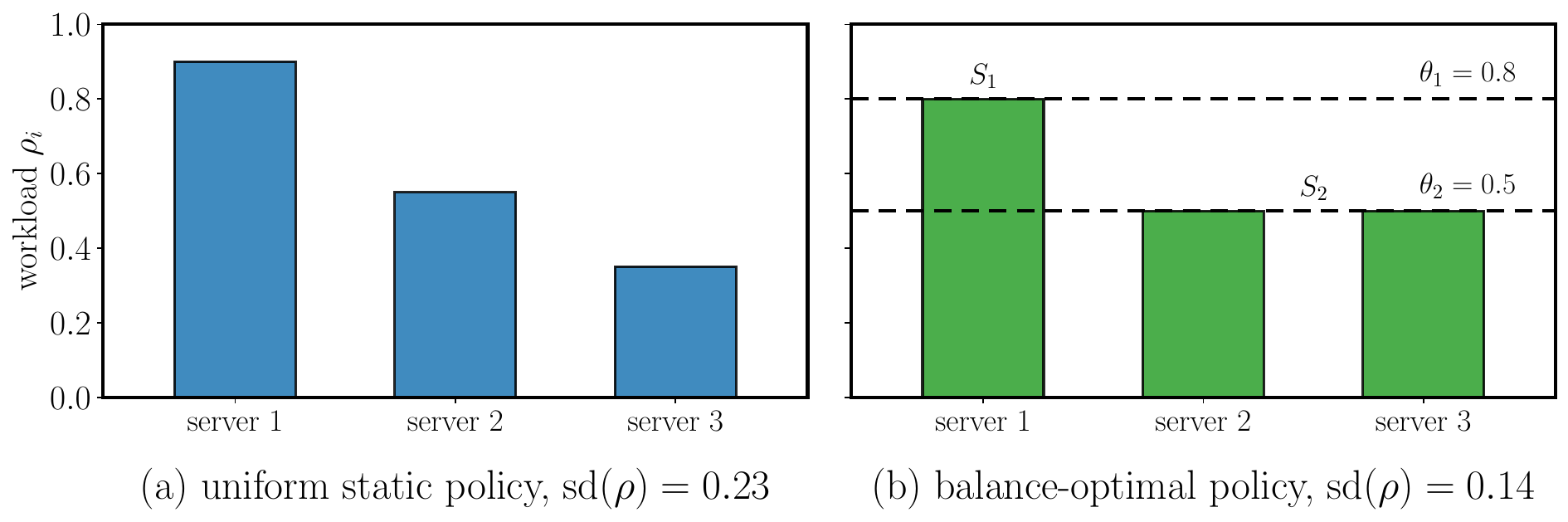}}
    {Water-filling illustration of the balance-optimal policy in a three-server example. \label{fig:water-filling}}
    {Three servers with $\mu_i=1$, primary arrival rates $(\lambda_1,\lambda_2,\lambda_3)=(0.8,0.3,0.2)$, and overlapping regions $\mathcal{O}_{\{1,2\}}$ and $\mathcal{O}_{\{2,3\}}$ with rates $0.2$ and $0.3$. (a) Workloads under the uniform static policy $\bar\eta_e=(1/2,1/2)$. (b) Workloads under the balance-optimal policy of Corollary~\ref{cor:water-filling}; dashed lines mark the water-filling levels $\theta_1=0.8$ (block $S_1=\{1\}$) and $\theta_2=0.5$ (block $S_2=\{2,3\}$).}
\end{figure}

\begin{example}
Figure~\ref{fig:water-filling} illustrates the corollary on a three-server example with $\mu_i=1$, primary arrival rates $(\lambda_1,\lambda_2,\lambda_3)=(0.8,0.3,0.2)$, and two overlapping regions $\mathcal{O}_{\{1,2\}}$ and $\mathcal{O}_{\{2,3\}}$ with arrival rates $\lambda_{\{1,2\}}=0.2$ and $\lambda_{\{2,3\}}=0.3$. The first-level bottleneck is $S_1=\{1\}$ with $\theta_1=D(\{1\})/M(\{1\})=0.8$: server 1's primary demand alone pins its workload, and no dispatch policy can relieve it. Contracting, the residual system has $D_2(\{2\})=\lambda_2+\lambda_{\{1,2\}}=0.5$, $D_2(\{3\})=0.2$, and $D_2(\{2,3\})=1.0$, so $\theta_2=0.5$ with maximal bottleneck $S_2=\{2,3\}$. The water-filling workloads are therefore $\rho^\ast=(0.8,0.5,0.5)$, attained by routing all of $\lambda_{\{1,2\}}$ to server 2 and all of $\lambda_{\{2,3\}}$ to server 3, i.e., $\bar\eta_{\{1,2\}}=(0,1)$ and $\bar\eta_{\{2,3\}}=(0,1)$. This flow satisfies the equilibrium condition \eqref{eq:balance-equilibrium}: neither overlapping region sends demand to a busier eligible server. For comparison, the uniform static policy $\bar\eta_e=(1/2,1/2)$ yields $\rho=(0.9,0.55,0.35)$ with $\mathrm{sd}(\rho)=0.23$, versus the optimal $\mathrm{sd}(\rho^\ast)=0.14$; a numerical search over $\mathcal X$ confirms that no feasible flow attains a smaller value. Moreover, solving the truncated hyperlattice model under the induced balance-optimal static policy reproduces the workloads $(0.80,0.50,0.50)$ up to truncation error, consistent with Theorem~\ref{thm:attainable-workload}.
\end{example}

\section{Experiments}
\label{sec:experiments}
In this section, we evaluate the performance of the hyperlattice model using the proposed metrics in Section~\ref{sec:performance-measure}. Our assessment includes a simulated environment in Section~\ref{sec:synthetic} and a real-world case study on police redistricting in Atlanta, Georgia, conducted in collaboration with the Atlanta Police Department in Section~\ref{sec:case_study}.

\subsection{Synthetic results}\label{sec:synthetic}

We first evaluate the effectiveness of the proposed hyperlattice queueing model through synthetic experiments. These experiments employ diverse parameter settings within synthetic service systems to investigate the accuracy of the performance metrics estimated by our models.
To embark on this, we first use the hyperlattice queueing model to estimate the steady-state probabilities of the systems by solving the balance equations \eqref{eq:balance-equations}. Based on the solutions, we then calculate summary statistics, including individual workloads and the fraction of dispatches, as shown in \eqref{eq:workload} and \eqref{eq:fraction-dispatch}. To validate the accuracy of our model estimations, we also developed a simulation program using \texttt{simpy} \citep{simpy}, a process-based discrete-event simulation framework based on Python. The primary purpose of this simulation program is to provide a ground truth for the queueing dynamics of the system, allowing us to compare the estimated performance measures from our model with the simulation results.

The simulation setup involves a $100 \times 100$ grid, where the width and height are both set to $100$ units. We explore different levels of overlapping service areas by varying the overlap width from $0$ to $99$ in $100$ increments. For each overlap width, we calculate the subregion radius and set the positions of the two servers accordingly. The servers are positioned symmetrically along the $x$-axis at coordinates determined by the radius of the subregion, and each server is responsible for a specific subregion polygon.
When an event occurs, its location is randomly determined within the region. The simulation then identifies the police units responsible for the event based on their assigned regions. If multiple police units are available, a dispatch strategy is employed to select which unit will respond. The selected police unit travels to the event location, and the travel time is computed based on the Euclidean distance between the unit's current position and the event location.
Upon arrival, the police unit handles the event, which involves a service time drawn from an exponential distribution based on the unit's service rate.
The police units have a service rate of $1$ and move at a constant speed of 10 units, ensuring that queuing delay remains the predominant factor influencing the server's response time for each incident.
The event arrival rate is set to $1$, with each simulation run lasting $1000$ events.
Using the simulation results, we can approximate the true individual workload of each server $\rho_{i}$ by calculating the percentage of time that the server is busy, and the true fraction of dispatches $\rho_{i,e}$ (or $\rho_{i,i}$) by calculating the percentage of calls received in region $e$ (or $i$) that have been assigned to server $i$.

 The experiments were conducted on a personal laptop running macOS, equipped with an M3 Pro chip featuring a $12$-core CPU with a clock speed of up to $4.06$ GHz and $18$ GB of LPDDR5 RAM with a memory speed of up to $153.6$ GB/s.

\subsubsection{Overlapping patrolling}

We also test our models on the synthetic systems that allow overlapping patrolling. Adapting the truncated hyperlattice approximation from Section~\ref{sec:estimation} with an upper bound state of $K=10$, we create a two-server system where the service region is a 100 by 100 square with a single overlapping service region in the middle ($I=2$ and $E = 1$), as depicted in Figure~\ref{fig:system-exps} (a). We assume that calls are uniformly distributed with an arrival rate of $\lambda=1$ and that the service rate of each server is fixed at $\mu_i=1$.
Changing the hyperparameters of the system, such as the size of the overlapping region and the dispatch policy, can significantly shift the queueing dynamics.
Our numerical experiments suggest that the hyperlattice queuing models can accurately capture intricate queuing dynamics under different hyperparameter settings. In Figures \ref{fig:numerical_rho} (a) and (b), we observe that the approximation error of both the individual workloads and the fraction of dispatch by the hyperlattice model is less than $0.1$ across all $\lambda/\mu$ ratios. Additionally, in Figure \ref{fig:numerical_rho} (c), we show that the mean response time of the system is also closely approximated by the hyperlattice model.

\begin{figure}[!t]
\FIGURE
    {\includegraphics[width=\linewidth]{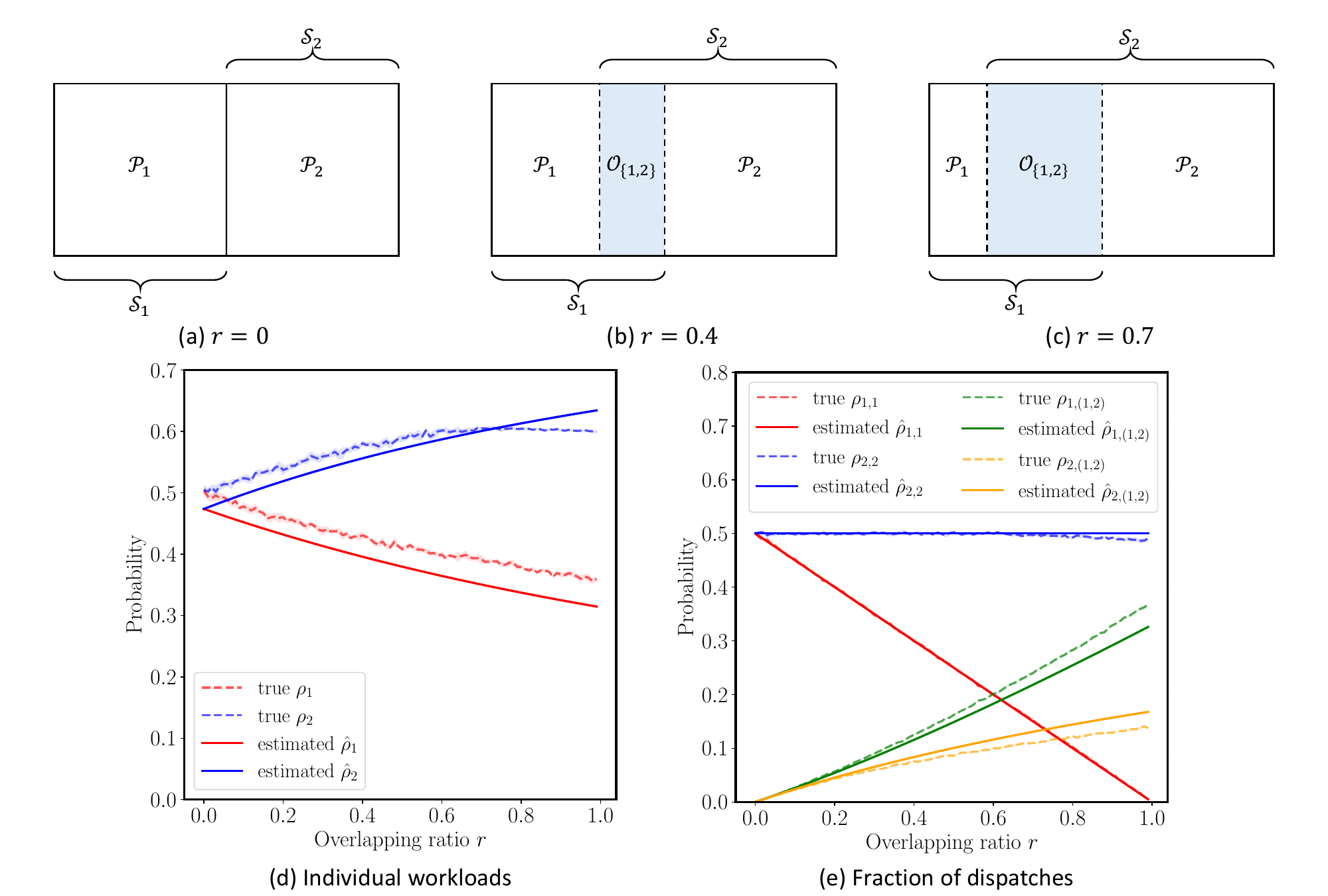}}
    {Synthetic results for two-server systems with varying overlapping ratio $r$. \label{fig:exp-setting-ratio}}
    {In (a), (b), and (c), the shaded area represents the overlapping region of the system. (d) and (e) show individual workloads $\rho_i$ and fraction of dispatches $\rho_{i,(i,j)}$ with varying overlapping ratio $r$. The solid lines represent the estimated performance measures using hyperlattice queueing models, whereas the dashed lines represent the simulation means over 10 independent trajectories, with shaded bands denoting 95\% confidence intervals. }
\end{figure}

To examine the impact of the overlapping ratio on system dynamics, we first create 100 systems by varying the overlapping ratio $r$. This ratio represents the percentage of the area of the overlapping region with respect to the area of the entire service region of server 1, i.e., $r = |\mathcal{O}_{\{1,2\}}| / (|\mathcal{O}_{\{1,2\}}| + |\mathcal{P}_{1}|)$, as shown in Figure~\ref{fig:exp-setting-ratio} (a-c). We then compare the estimated values of $\rho_i$ and $\rho_{i,e}$ to their true values (indicated by dashed lines) in Figure~\ref{fig:exp-setting-ratio} (d) and (e), and find that the estimated values are a good approximation of their true counterparts.
Clearly, the workloads of the two servers differ more as the overlapping ratio increases, thereby enlarging the area that server 1 has to cover.
Moreover, as the ratio increases, the probability of dispatching either server to the overlap region grows (indicated by green and orange lines). In contrast, the probability of sending server 2 to its primary area remains the same (blue line), and the chance of dispatching server 1 decreases (red line). This observation arises from our simplified synthetic setup, where we use a uniform dispatch policy and the dispatch probability to a region directly correlates with its size.

\begin{figure}[!t]
\FIGURE
    {\includegraphics[width=\linewidth]{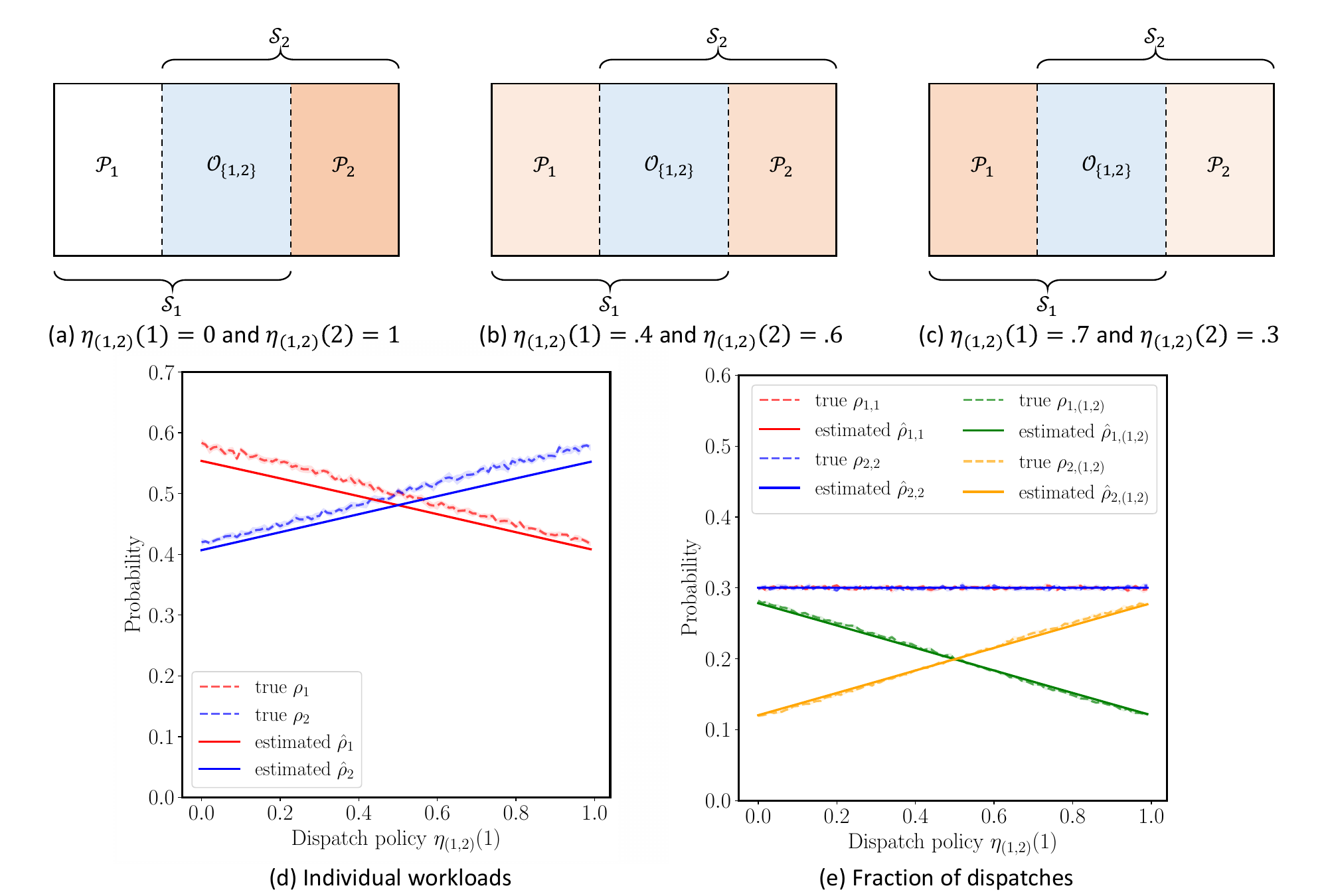}}
    {Synthetic results for two-server systems with varying dispatch policy $\eta_{(1,2)}(1)$ and $\eta_{(1,2)}(2)$. \label{fig:exp-setting-eta}}
    {In (a), (b), and (c), the shaded area represents the overlapping region of the system. (d) and (e) show individual workloads $\rho_i$ and fraction of dispatches $\rho_{i,(i,j)}$ with varying overlapping ratio $r$. The solid lines represent the estimated performance measures using hyperlattice queueing models, whereas the dashed lines represent the simulation means over 10 independent trajectories, with shaded bands denoting 95\% confidence intervals. }
\end{figure}

Furthermore, we also investigate the impact of dispatch policy by creating another 100 systems with different values of $\eta_{(1,2)}(1)$ and $\eta_{(1,2)}(2)$ (where $\eta_{(1,2)}(1) = 1 - \eta_{(1,2)}(2)$), while keeping the service regions the same, as shown in Figure~\ref{fig:exp-setting-eta} (a-c).
For simplicity, we here omit the subscript state $u$ in the notation of dispatch policy $\eta_{e, u}(i),~\forall i \in e$ as we assume it depends on the state $u$ only only when prioritizing idle servers and remains independent of $u$ when all servers are busy, as discussed in Section~\ref{sec:dispatch_policy}.
Figure~\ref{fig:exp-setting-eta} (d) and (e) show that the estimated $\rho_i$ and $\rho_{i,e}$ (indicated by solid lines) can closely approximate their true values suggested by the simulation.
As we observe, when the dispatch probability for a call from the overlapping region to either server is 0.5, the workload distribution between the two servers is equal, \ie, $\eta_{1,2}(1) = \eta_{1,2}(2) = 0.5$.

\subsection{Case study: Police Redistricting in Atlanta, GA}\label{sec:case_study}

\begin{figure}[!t]
\FIGURE
    {\includegraphics[width=.8\linewidth]{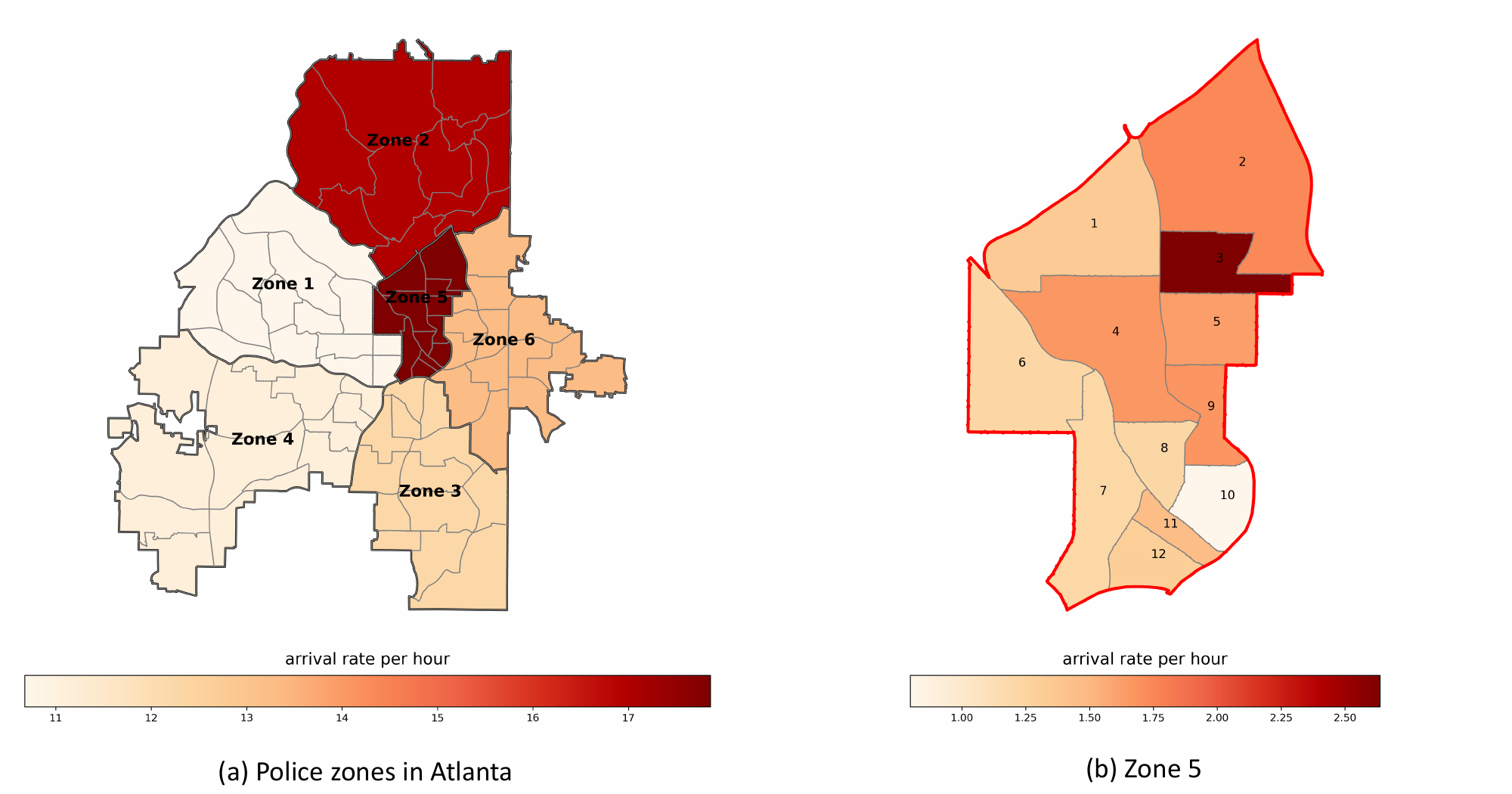}}
    {Police workload in Atlanta, Georgia. \label{fig:districting-map}}
    {(a) presents a comprehensive map delineating the various police zones in Atlanta, Georgia, while (b) specifically focuses on Zone 5. The intensity of color shading directly corresponds to the frequency of police call arrivals, providing a visual representation of call distribution across the zones.}
\end{figure}

\begin{figure}[!t]
\FIGURE
    {\includegraphics[width=.86\linewidth]{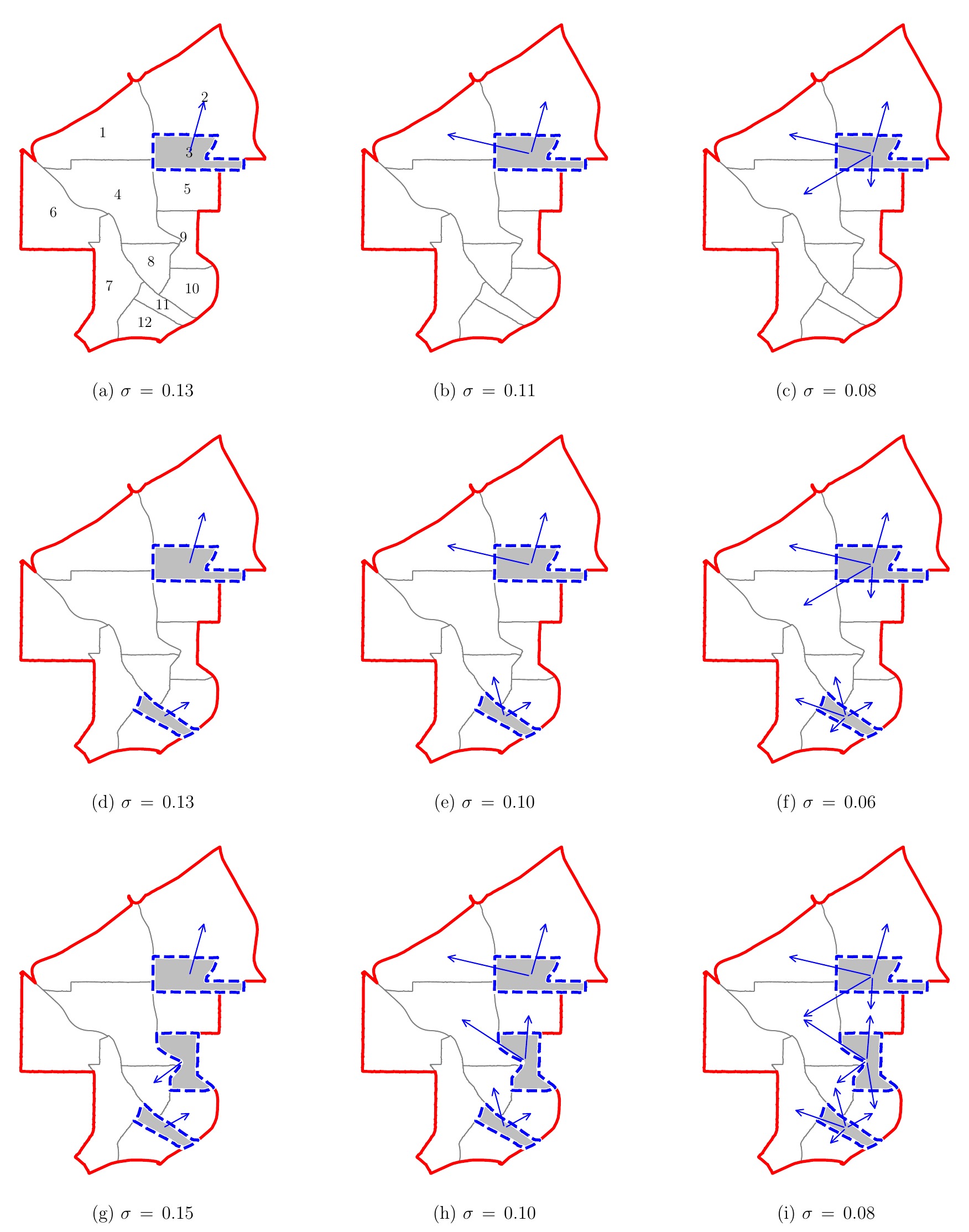}}
    {An illustrative example showing different districting plans for Zone 5 in Atlanta. \label{fig:districting-exp}}
    {Gray lines represent the basic geographical units patrolled by the police, and red lines outline the districting plans. In the map, regions shaded in grey with dashed blue lines represent beats without a dedicated police unit. The presence of blue arrows originating from these beats signifies that their policing needs are addressed by units from adjacent beats. Panel (a) annotates each beat with its region index (beat number minus 500); the same layout applies to all panels. Panels (g)--(i) in the bottom row show the plans with three beats lacking dedicated units (regions 3, 9, and 11); the reported $\sigma$ values correspond to the uniform dispatch probabilities.}
\end{figure}

In this section, we present a case study focusing on the police redistricting problem within Atlanta, Georgia. Figure \ref{fig:districting-map} illustrates the various police zones of the city, with a particular emphasis on the intricate beats of Zone 5. Ideally, each beat would be served by a dedicated police unit, ensuring a focused response within its boundaries. However, this is often not feasible due to a limited number of available police units, primarily attributed to constraints in officer and resource allocation. In instances where a specific beat lacks a dedicated police unit, units from adjacent beats are required to assist. While such support is permissible among units within the same zone, assistance across zones is explicitly prohibited. The depth of color shading in Figure \ref{fig:districting-map} is a direct representation of the frequency of police call arrivals, quantified by the arrival rate $\lambda$. These rates are estimated based on an analysis of historical 911 calls-for-service data, specifically collected during the years 2021 and 2022, as detailed in the study of \cite{ zhu2022data}. The color indicates that Zone 5 experiences the highest frequency of police calls, averaging approximately 18 calls per hour. As a result, our analysis concentrates on examining potential overlapping districting plans specific to the beats in Zone 5. A key performance metric under consideration is the standard deviation of the individual workloads, as previously discussed in Section \ref{sec:workload}.

In Figure \ref{fig:districting-exp}, we measure the standard deviation of the individual workloads for different overlapping districting plans using the truncated hyperlattice model in Section~\ref{sec:estimation} with an upper bound state of $K=10$, denoted as $\sigma$. Panels (a), (b), and (c) display a scenario where a single beat lacks a dedicated police unit. Specifically, in panel (a), the districting plan is non-overlapping since the policing requirements of this unassigned beat are exclusively managed by a single adjacent beat. In panels (b) and (c), we display different overlapping districting plans by varying the number of overlapping beats. In these configurations, the standard deviations of workload are notably reduced, measured at $0.11$ and $0.08$ respectively, demonstrating the impact of overlapping districts on workload distribution uniformity. Panels (d), (e), and (f) depict scenarios wherein two beats operate without dedicated police units. Notably, panel (f) demonstrates a further reduction in workload standard deviation compared to panel (c), measured at $0.06$, attributable to an additional beat being serviced by the overlapping police units from its neighbors.
Panels (g), (h), and (i) depict scenarios in which three beats (regions 3, 9, and 11) operate without dedicated police units, so that nine units cover the twelve beats. The degree of overlap increases from a single neighboring unit per shared beat in panel (g) to all adjacent units in panel (i). Under uniform dispatch probabilities, the imbalance of the nine-unit plans remains above that of the corresponding ten-unit plans. For example, it is $0.08$ in panel (i) compared with $0.06$ in panel (f). As shown next, balance-optimal dispatch probabilities recover most of this gap.

\input{tables/workload}

In Table \ref{table:workload_table}, we record the workload allocation among police units across beats under different districting plans. Plans (a) -- (f) align with corresponding panels in Figure \ref{fig:districting-exp}. Plan ($*$) depicts an independent scenario where each beat is serviced by a dedicated police unit with non-overlapping districts. It is noteworthy that the total number of police units varies by districting plan, with plans (a) -- (c) utilizing 11 units, plans (d) -- (f) utilizing 10 units, plans (g) -- (i) utilizing 9 units, and plan ($*$) utilizing all 12 units. We observe that plan (f) exhibits a lower standard deviation in workload compared to plan ($*$), despite the latter having two additional police units.

\paragraph{Dispatch Optimization.}
The dispatch probabilities used in Table~\ref{table:workload_table} allocate the demand of each beat without a dedicated unit uniformly across its overlapping units.
The policy structure developed in Section~\ref{sec:attainable-workloads} allows us to address two further questions for each districting plan: where the underlying capacity bottlenecks lie, and how much workload imbalance can be reduced by adjusting dispatch probabilities alone, without changing district boundaries.
Before proceeding, we note that Zone 5 is overloaded, i.e., the total call rate exceeds the total service capacity even when all 12 units are deployed, so $\theta^*>1$ for every plan and no dispatch policy can stabilize the system (Theorem~\ref{thm:bottleneck-stability}), which is consistent with the sustained congestion that motivates this study.
The flow characterization in Lemma~\ref{lemma:hypergraph-flow-characterization} and the water-filling decomposition in \eqref{eq:water-filling-recursion} remain well-defined and characterize the most balanced attainable allocation of offered load across units. When the system is stable, this allocation coincides with the workload-optimal policy in Corollary~\ref{cor:water-filling}. We therefore use the water-filling dispatch probabilities $\bar\eta^\ast$ as a theory-guided dispatch design and evaluate their performance using the same truncated hyperlattice pipeline as in Table~\ref{table:workload_table}.

\input{tables/optimal_workload}

Table~\ref{table:optimal_workload} reports the bottleneck ratio $\theta^*$, the maximal bottleneck set (Definition~\ref{def:bottleneck-server-set}), and the workload standard deviation under uniform dispatch and $\bar\eta^\ast$. The bottleneck sets are computed exactly by enumeration. We omit plans (a), (d), and (g) because each shared region has only one eligible unit. Thus, the balance-optimal dispatch coincides with uniform dispatch. For these plans, $\theta^=4.06$ and unit 2 is the bottleneck. It must serve both high-demand regions 2 and 3, making these plans worse than dedicating a unit to region 3 as in plan ($*$).

Three observations stand out in our case study.
First, the bottleneck ratio quantifies how each districting plan redistributes structural load. Dedicating a unit to the busiest region 3 (plan ($*$)) yields $\theta^*=2.69$, pinned by region 3 alone; widening the overlap to four neighbors (plans (c) and (f)) spreads the shared demand and lowers the bottleneck to $\theta^*=2.19$, below even the fully dedicated plan, despite using fewer units.
Second, the maximal bottleneck sets localize where the capacity shortfall resides and how it spreads: the four units surrounding region 3 in plans (c)/(f), and, in plan (i), the six units $\{1,2,4,5,8,10\}$, showing that wide overlap delocalizes the bottleneck across a larger pool at a slightly higher ratio ($\theta^*=2.27$), i.e., these sets identify the natural targets for added capacity or further overlap.
Third, replacing the uniform dispatch probabilities with $\bar\eta^\ast$ weakly reduces the workload imbalance in every plan, and the gains grow with the richness of the overlap structure. For plans with limited routing freedom the gains are marginal, for plan (f) the imbalance decreases from $0.061$ to $0.055$ ($\sim 10\%$), and for plan (i) it decreases from $0.080$ to $0.064$ ($\sim 20\%$).
The optimal weights help explain this difference. Uniform dispatch sends too many calls to units shared by multiple regions, while $\bar\eta^\ast$ better accounts for the chain structure. For instance, it sends most calls from region 9 to units 8 and 10, none to unit 4, and assigns calls from region 11 only to units 7 and 12. Under $\bar\eta^\ast$, the nine-unit plan (i) achieves $\sigma=0.064$, which is close to the $\sigma=0.061$ achieved by the ten-unit plan (f) under uniform dispatch. In this sense, theory-guided dispatch recovers most of the workload balance lost from removing one unit. More broadly, the analysis separates the imbalance that can be addressed through better dispatch from the imbalance created by the districting plan itself, at negligible computational cost.

\input{tables/mtt_table}

{In Table \ref{table:travel_time_table}, we present the mean travel times of different police units under various districting plans, as well as the overall MTT. It is observed that enabling overlapping service regions results in increased mean travel times for police units serving neighboring areas. Although plan (f) exhibits a lower standard deviation in workload compared to plan ($*$) with fewer police units, it increases the mean travel time for neighboring units by more than one minute and the overall mean travel time by about half a minute.}

{
\paragraph{Capacity Planning.}
The stability characterization of Theorem~\ref{thm:bottleneck-stability} can also be converted into a capacity-planning tool, e.g., for a given districting plan, what is the minimum number of police cars, and where should they be allocated, so that the zone becomes stabilizable (i.e., $\theta^*<1$)?
An additional car allocated at a beat is modeled as an additional server eligible for exactly the regions covered by that beat's unit, so stability is still governed by the trapped-demand cuts, i.e., a staffing vector $(k_j)_j$, where $k_j$ denotes the number of cars allocated with unit $j$, stabilizes the system if and only if $D(S)<\mu\sum_{j\in S}k_j$ for every nonempty unit set $S$. The minimum fleet size $N^*$ is therefore the optimal value of a small integer program over the $2^I-1$ cut constraints, which we solve exactly.

Table~\ref{table:staffing} shows the results. Since all demand must be served, no plan can stabilize Zone~5 with fewer than $\lfloor \lambda/\mu \rfloor+1=19$ cars, establishing a full-pooling lower bound. We highlight three insights. First, the fully dedicated plan ($*$) requires $N^*=24$ cars, five more than the pooling bound, because every beat must independently over-provision against its own demand ($k_j=\lceil\lambda_j/\mu\rceil$).
Second, the staffing savings grow with the richness of the overlap, i.e., within each group of plans, widening the overlap reduces the minimum fleet monotonically, whereas the single-neighbour plans (a), (d), and (g) yield no savings at all ($N^*=24$). Third, plan (i) stabilizes the zone with $N^*=20$ cars, four fewer than the dedicated plan and within one car of the full-pooling bound, indicating that partial overlap, while preserving primary geographic responsibility, recovers nearly all of the staffing benefit of complete pooling. Together with Table~\ref{table:optimal_workload}, the districting plan thus determines both how many cars are needed for stability and how evenly the offered load can be balanced among them.
}

\input{tables/staffing}

The case study also provides valuable managerial insights for addressing the police redistricting problem. It shows that deploying a larger number of police units does not necessarily equate to more effective policing. For instance, plan (f) with 10 units had a lower standard deviation of workload compared to plan ($*$) with 12 units. This suggests that utilizing overlapping police beats, as opposed to a rigid, non-overlapping structure, can significantly enhance the efficiency of resource allocation.
 However, it should be noted that utilizing overlapping police beats may come at the cost of increased mean travel times. Practitioners must carefully balance the trade-off between response time and resource allocation efficiency in their actual implementations. In general, the overlapping districting approach is preferred in high-demand areas, such as Zone 5 in Atlanta or other cities facing similar challenges.

\section{Discussions}
\label{sec:discussions}

This paper presents a novel generalized hypercube queueing model that captures the detailed queueing dynamics of each server in a system.
We demonstrate the model's versatility by applying it to analyze overlapping patrols. Our Markov model employs a hyperlattice to represent an extensive state space, where each vertex corresponds to an integer-valued vector. The transition rate matrix is highly sparse, permitting transitions between vectors with an $\ell_1$ distance difference of 1. To enhance computational tractability, we introduce an efficient state truncation technique that connects our model to a general birth-and-death process with state-dependent rates. This simplifies the detailed balancing equations into a significantly smaller linear system of equations, facilitating the efficient calculation of steady-state distributions. These distributions are essential for evaluating general performance metrics, as acknowledged in Larson's seminal work.

We emphasize that, akin to the rationale in Larson's original paper \citep{larson1974hypercube}, no exact analytical expression exists for the hyperlattice model regarding travel time measures for an infinite-line capacity system. This challenge arises from the server's position not being chosen based on the received region probability distribution of the call during system saturation periods when the server is dispatched consecutively. Consequently, in practice, we must assume that queueing delay is the primary factor influencing the server's response time for each call.

Additionally, we highlight that our framework concentrates on a random policy controlled by a set of state-dependent matrices, and it can be extended to more intricate decision-making situations. This generalization would empower us to develop and optimize police dispatching policies that take into account queueing dynamics for more realistic systems. This consideration is particularly relevant since conventional First-Come-First-Served (FCFS) approaches may not be optimally efficient, leaving room for enhancement and optimization. Lastly, although our paper focuses on police districting, the generalized hypercube queueing model can also be applied to model mobile servers in a broader range of applications.

An important consideration emphasized in the literature is the explicit modeling of travel times in queueing systems, as region-dependent service rates and spatial dynamics critically influence queue behavior and dispatch effectiveness (e.g., \cite{kanoria2021dynamic, besbes2022spatial}). In our hyperlattice queueing model, travel times are treated as exogenous — consistent with Larson’s hypercube formulation — to maintain analytical and computational tractability. Nevertheless, incorporating travel-time dynamics directly into the Markov chain would improve accuracy for spatially sensitive dispatch policies. As an intermediate step, our supplementary experiments in Section \ref{sec:compare_hypercube} incorporate heterogeneous service rates that reflect server-specific travel times, thereby validating the robustness of the truncated hyperlattice approximation. Extending the model to endogenize travel-time dynamics remains a promising direction for future research.

\bibliographystyle{informs2014}
\bibliography{ss_refs}

\clearpage
\setcounter{page}{1}
\setcounter{section}{0}
\setcounter{figure}{0}
\setcounter{table}{0}
\setcounter{equation}{0}

\renewcommand{\thepage}{EC\arabic{page}}
\renewcommand{\thesection}{EC\arabic{section}}
\renewcommand{\thefigure}{EC\arabic{figure}}
\renewcommand{\thetable}{EC\arabic{table}}
\renewcommand{\theequation}{EC\arabic{equation}}

\begin{center}
{\Large\bfseries E-Companion for
``Generalized Hypercube Queueing Models with Overlapping Service Regions''}\\[1em]
\end{center}

\section{Notation table}
\begin{table}[!h]
\caption{Summary of key notations}
\label{tab:notation}
\resizebox{\textwidth}{!}{%
\begin{tabular}{lll}
\midrule[0.3pt]
\bf Section & \bf Notation & \bf Description \\ \hline
\ref{sec:setup}
& $I$ & Total number of servers. \\
& $\mathcal{I} = \{i = 1, \dots, I\}$ & Set of indices of servers as well as their primary service regions.\\
& $\mathcal{E} = \{e \subseteq \mathcal{I}\}$ & Set of server subsets defining overlapping service regions.\\
& $\mathcal{S}_i$, $\mathcal{S}$ & The service region patrolled by server~$i$, and the entire geographical space.\\
& $\mathcal{P}_i$ & Primary service region patrolled by server~$i$.\\
& $\mathcal{O}_e$ & Overlapping service region patrolled by server set~$e$.\\
& $\lambda_{i}$, $\lambda_{e}$, $\lambda$ & Arrival rates of primary region~$\mathcal{P}_i$, overlapping region~$\mathcal{O}_e$, and the entire region.\\
& $\mu_i$, $\mu$ & Service rate of server~$i$ and the total service rate of all servers.\\
\hline
\ref{sec:hyperlattice}
& $n_i \in \mathbb{Z}_+$ & Status of server $i$. Server $i$ is idle if $n_i = 0$ and busy if $n_i > 0$. \\
& $\eta_{e,u}(i) \in [0, 1]$ & Probability of choosing server $i$ in overlapping region $e$ at state $u$. \\
& $B = (n_i)_{i\in\mathcal{I}}$ & A state (node) in the hyperlattice.\\
& $\mathcal{U} = \{u = 1, 2, \dots\}$ & Set of state indices ordered by the tour algorithm. \\
& $Q = (q_{uv})_{u,v\in\mathcal{U}}$ & Transition rate matrix for the hyperlattice queue.\\
\hline
\ref{sec:estimation}
& $U_K$ & Number of states where the total number of calls in the system is less than $K$.\\
& $\mathcal{U}_K = \{u = 1, \dots, U_K\}$ & Set of states for the truncated hyperlattice queue. \\
& $Q_K = (q_{uv})_{u,v\in\mathcal{U}_K}$ & Transition rate matrix for the truncated hyperlattice queue.\\
\hline
\ref{sec:performance-measure}
& $\rho_i$ & Fraction of time that server $i$ is busy serving calls. \\
& $\rho_{i,e} / \rho_{i,i}$ & Fraction of dispatches that send server $i$ to each sub-region $e$ (or $i$). \\
& $\tau_i / \tau_e$ & Mean travel time for server $i$ or region $e$. \\
\hline
{\ref{sec:attainable-workloads}}
& $\mathcal{H}$, $\mathcal{H}_{\mathrm{st}}$ & Set of Markovian dispatch policies, and its subset under which the hyperlattice CTMC is stable. \\
& $D(S)$, $M(S)$ & Demand trapped within server set $S$, and the total service capacity of $S$. \\
& $F(S)$ & Total demand reachable by the servers in $S$. \\
& $x^{\eta}_{ei}$ & Long-run rate at which calls from overlapping region $e$ are assigned to server $i$ under policy $\eta$. \\
& $\mathcal{X}$ & Hyperedge-flow polytope of feasible allocations of overlapping-region demand. \\
& $a_i(x)$ & Arrival rate induced at server $i$ by the flow $x\in\mathcal{X}$. \\
& $\mathcal{R}_{\mathrm{st}}$ & Attainable workload region under stable Markovian dispatch policies. \\
& $\theta^*$ & Minimum achievable peak workload (bottleneck ratio). \\
& $\mathcal{B}^*$, $S^*$ & Family of bottleneck server sets, and a bottleneck server set. \\
& $\theta_k$, $S_k$, $U_k$, $D_k$ & Water-filling levels and blocks, and the residual server set and trapped-demand function at level $k$. \\
& $\rho^\ast$, $\bar\eta^\ast$ & Water-filling workload vector, and the balance-optimal state-independent dispatch policy. \\
\midrule[0.3pt]
\end{tabular}%
}
\end{table}

\section{Technical proofs}\label{sec:proofs}

\subsection{Proof of Lemma~\ref{lemma:stationary-rate-conservation}}
\proof{Proof of Lemma~\ref{lemma:stationary-rate-conservation}.}
Equation~\eqref{eq:edge-flow-conservation} follows from $\sum_{i\in e}\eta_{e,u}(i)=1$ for every state $u$.

We prove \eqref{eq:server-rate-conservation} without assuming that the stationary mean queue length is finite. Fix $i\in\mathcal I$ and define the instantaneous rate at which calls are assigned to server $i$ in state $B_u$ by
$$
    r_i^\eta(u)
     \coloneqq \lambda_i+\sum_{e\in\mathcal E:\,i\in e}
      \lambda_e\eta_{e,u}(i).
$$
For $m\in\mathbb Z_+$, consider the bounded function $f_m(B_u) \coloneqq \min\{(B_u)_i,m\}$. The generator of the CTMC satisfies
$$
    \mathcal L^\eta f_m(B_u)
    =r_i^\eta(u)\mathbbm{1}\{(B_u)_i<m\}
     -\mu_i\mathbbm{1}\{1\le (B_u)_i\le m\}.
$$
Because $f_m$ is bounded, stationarity implies
$$
    0=\sum_{u\in\mathcal U}\pi_u^\eta\mathcal L^\eta f_m(B_u).
$$
Letting $m\to\infty$ and applying dominated convergence to the first term and monotone convergence to the second term gives
\begin{align*}
    \mu_i\rho_i^\eta
    &=\sum_{u\in\mathcal U}\pi_u^\eta r_i^\eta(u)\\
    &=\lambda_i+\sum_{e\in\mathcal E:\,i\in e}
      \lambda_e\sum_{u\in\mathcal U}
      \pi_u^\eta\eta_{e,u}(i)\\
    &=\lambda_i+\sum_{e\in\mathcal E:\,i\in e}x_{ei}^\eta.
\end{align*}
This proves \eqref{eq:server-rate-conservation}.
\hfill\Halmos
\endproof

\subsection{Proof of Lemma~\ref{lemma:hypergraph-flow-characterization}}
\proof{Proof of Lemma~\ref{lemma:hypergraph-flow-characterization}.}
Suppose first that $a=a(x)$ for some $x\in\mathcal X$. Summing over all servers
gives
$$
    a(\mathcal I)
    =\sum_{i\in\mathcal I}\lambda_i
     +\sum_{e\in\mathcal E}\sum_{i\in e}x_{ei}
    =\sum_{i\in\mathcal I}\lambda_i
     +\sum_{e\in\mathcal E}\lambda_e
    =\lambda.
$$
For any $S\subseteq\mathcal I$, all primary calls indexed by $i\in S$ must be assigned inside $S$, and all calls from a hyperedge $e\subseteq S$ must also be assigned inside $S$. Therefore,
\begin{align*}
    a(S)
    &=\sum_{i\in S}\lambda_i
      +\sum_{e\in\mathcal E}
        \sum_{i\in e\cap S}x_{ei}\\
    &\ge\sum_{i\in S}\lambda_i
      +\sum_{e\in\mathcal E:\,e\subseteq S}
        \sum_{i\in e}x_{ei}
     =D(S).
\end{align*}
This proves necessity.

For sufficiency, let $a$ satisfy
\eqref{eq:attainable-total-rate}--\eqref{eq:attainable-lower-cuts}, and set
\begin{equation*}
    y_i \coloneqq a_i-\lambda_i,
    \qquad
    L \coloneqq \sum_{e\in\mathcal E}\lambda_e.
\end{equation*}
The singleton constraints in \eqref{eq:attainable-lower-cuts} imply $y_i\ge0$, and \eqref{eq:attainable-total-rate} implies $y(\mathcal I)=L$. Construct a directed network with a source node $s$, one demand node for every $e\in\mathcal E$, one server node for every $i\in\mathcal I$, and a sink node $t$. Give arc $(s,e)$ capacity $\lambda_e$, arc $(e,i)$ infinite capacity for every $i\in e$, and arc $(i,t)$ capacity $y_i$.

Consider an $s$--$t$ cut and let $S\subseteq\mathcal I$ be the set of server nodes on the source side. A finite cut can place demand node $e$ on the source side only if $e\subseteq S$. For a fixed $S$, the minimum-capacity cut places on the source side every demand node satisfying $e\subseteq S$. Its capacity is
\begin{align*}
    \sum_{e\in\mathcal E:\,e\not\subseteq S}\lambda_e
    +\sum_{i\in S}y_i
    &=L-\sum_{e\in\mathcal E:\,e\subseteq S}\lambda_e+y(S)\\
    &=L+a(S)-D(S)\\
    &\ge L.
\end{align*}
The cut with only $s$ on the source side has capacity $L$. Hence, by the max-flow/min-cut theorem, the maximum flow has value $L$. All arcs leaving $s$ are therefore saturated. Since $y(\mathcal I)=L$, all arcs entering $t$ are also saturated. Let $x_{ei}$ be the flow on arc $(e,i)$. Flow conservation then gives
\begin{equation*}
    \sum_{i\in e}x_{ei}=\lambda_e,
    \qquad
    \sum_{e\in\mathcal E:\,i\in e}x_{ei}=y_i,
\end{equation*}
so $x\in\mathcal X$ and $a_i=a_i(x)$.

Finally, because $a(\mathcal I)=\lambda$, the lower inequality for
$\mathcal I\setminus S$ is equivalent to
\begin{align*}
    a(S)
    &=\lambda-a(\mathcal I\setminus S)\\
    &\le \lambda-D(\mathcal I\setminus S)\\
    &=\sum_{i\in S}\lambda_i
      +\sum_{e\in\mathcal E:\,e\cap S\ne\varnothing}\lambda_e
     =F(S).
\end{align*}
This proves the equivalent upper-cut representation.
\hfill\Halmos
\endproof

\subsection{Proof of Theorem~\ref{thm:attainable-workload}}
\proof{Proof of Theorem~\ref{thm:attainable-workload}.}
Fix $\eta\in\mathcal H_{\mathrm{st}}$. By Lemma~\ref{lemma:stationary-rate-conservation}, $x^\eta=(x_{ei}^\eta)\in\mathcal X$ and
$$
    \mu_i\rho_i^\eta=a_i(x^\eta),
    \qquad i\in\mathcal I.
$$
The all-zero state is accessible from every hyperlattice state: with positive probability, no new call arrives until the finitely many calls currently in the system have completed service. Hence the all-zero state belongs to the positive recurrent class and has strictly positive stationary probability. It follows that $\rho_i^\eta<1$ for every server $i$. Applying Lemma~\ref{lemma:hypergraph-flow-characterization} to $a_i=\mu_i\rho_i^\eta$ establishes that every dynamically attainable workload vector belongs to the right-hand side of \eqref{eq:attainable-workload-region}.

We next construct a static policy with the same first-order allocation. For every $e\in\mathcal E$ with $\lambda_e>0$, define
$$
    \bar\eta_e(i) \coloneqq \frac{x_{ei}^\eta}{\lambda_e},
    \qquad i\in e.
$$
If $\lambda_e=0$, choose any probability distribution on $e$. Under $\bar\eta$, independent Poisson thinning implies that server $i$ receives a Poisson arrival stream with rate
\begin{equation*}
    \lambda_i+\sum_{e\in\mathcal E:\,i\in e}
      \lambda_e\bar\eta_e(i)
    =a_i(x^\eta)
    =\mu_i\rho_i^\eta
    <\mu_i.
\end{equation*}
The system therefore consists of independent stable $M/M/1$ queues. Its stationary distribution is
\begin{equation*}
    \bar\pi(B_u)
    =\prod_{i\in\mathcal I}
      (1-\rho_i^\eta)(\rho_i^\eta)^{(B_u)_i},
\end{equation*}
and its server workloads are exactly $(\rho_i^\eta)_{i\in\mathcal I}$. By construction, it also preserves every assignment rate $x_{ei}^\eta$.

Conversely, let $\rho$ belong to the right-hand side of \eqref{eq:attainable-workload-region}, and set $a_i \coloneqq \mu_i\rho_i$. Lemma~\ref{lemma:hypergraph-flow-characterization} yields some $x\in\mathcal X$ satisfying $a_i=a_i(x)$. Define a state-independent policy by $\bar\eta_e(i)=x_{ei}/\lambda_e$ when $\lambda_e>0$. Since $a_i=\mu_i\rho_i<\mu_i$, all induced $M/M/1$ queues are stable and have workloads $\rho_i$. Hence every vector on the right-hand side is attainable by a stable state-independent policy. \hfill \Halmos
\endproof

\subsection{Proof of Theorem~\ref{thm:bottleneck-stability}}
\proof{Proof of Theorem~\ref{thm:bottleneck-stability}.}
We first establish a lower bound. For any $x\in\mathcal X$ and any nonempty $S\subseteq\mathcal I$, Lemma~\ref{lemma:hypergraph-flow-characterization} gives $a(x)(S)\ge D(S)$. Therefore,
\begin{align*}
    D(S)
    &\le\sum_{i\in S}a_i(x)\\
    &=\sum_{i\in S}\mu_i\frac{a_i(x)}{\mu_i}\\
    &\le M(S)\max_{i\in\mathcal I}\frac{a_i(x)}{\mu_i}.
\end{align*}
It follows that
\begin{equation*}
    \max_{i\in\mathcal I}\frac{a_i(x)}{\mu_i}
    \ge \frac{D(S)}{M(S)}.
\end{equation*}
Taking the maximum over nonempty $S$ and then the minimum over
$x\in\mathcal X$ yields
\begin{equation}
    \theta^*
    \ge\max_{\varnothing\ne S\subseteq\mathcal I}
      \frac{D(S)}{M(S)}.
    \label{eq:bottleneck-lower-bound}
\end{equation}

For the reverse inequality, set
\begin{equation*}
    \theta^\dagger
     \coloneqq \max_{\varnothing\ne S\subseteq\mathcal I}
      \frac{D(S)}{M(S)}
\end{equation*}
and define the residual capacity available for overlapping-region calls at server $i$ by
\begin{equation*}
    c_i \coloneqq \theta^\dagger\mu_i-\lambda_i.
\end{equation*}
The singleton cut $S=\{i\}$ implies $c_i\ge0$. Moreover, for every $S\subseteq\mathcal I$,
\begin{align}
    \sum_{i\in S}c_i
    &=\theta^\dagger M(S)-\sum_{i\in S}\lambda_i\notag\\
    &\ge\sum_{e\in\mathcal E:\,e\subseteq S}\lambda_e.
    \label{eq:residual-capacity-cuts}
\end{align}
Build the same source--hyperedge--server--sink network used in the proof of Lemma~\ref{lemma:hypergraph-flow-characterization}, but give arc $(i,t)$ capacity $c_i$. Condition \eqref{eq:residual-capacity-cuts} states exactly that every source--sink cut has capacity at least $L=\sum_{e\in\mathcal E}\lambda_e$. The max-flow/min-cut theorem therefore yields a flow $x^*\in\mathcal X$ such that
\begin{equation*}
    \sum_{e\in\mathcal E:\,i\in e}x_{ei}^*\le c_i,
    \qquad i\in\mathcal I.
\end{equation*}
Consequently,
\begin{equation*}
    a_i(x^*)
    =\lambda_i+\sum_{e\in\mathcal E:\,i\in e}x_{ei}^*
    \le\theta^\dagger\mu_i,
\end{equation*}
so $\theta^*\le\theta^\dagger$. Combining this with
\eqref{eq:bottleneck-lower-bound} proves
\eqref{eq:bottleneck-ratio}.

Suppose now that a stable policy exists. By Theorem~\ref{thm:attainable-workload}, its workload vector satisfies
$\rho_i<1$ and
\begin{equation*}
    D(S)
    \le\sum_{i\in S}\mu_i\rho_i
    <\sum_{i\in S}\mu_i
    =M(S)
\end{equation*}
for every nonempty $S$. Hence $\theta^*<1$. Conversely, if $\theta^*<1$, the flow $x^*$ constructed above satisfies $a_i(x^*)\le\theta^*\mu_i<\mu_i$ for all $i$. The state-independent policy $\eta_e(i)=x_{ei}^*/\lambda_e$ therefore produces independent stable $M/M/1$ queues. This proves the equivalence of the four stability statements.
\hfill
\Halmos
\endproof

\subsection{Proof of Corollary~\ref{cor:static-sufficiency-workload}}
\proof{Proof of Corollary~\ref{cor:static-sufficiency-workload}.}
Fix any stable Markovian dispatch policy $\eta\in\mathcal H_{\mathrm{st}}$. By Theorem~\ref{thm:attainable-workload}, its stationary assignment rates $x^\eta$ belong to $\mathcal X$ and satisfy
$$
    \mu_i\rho_i^\eta=a_i(x^\eta)<\mu_i,
    \qquad i\in\mathcal I .
$$
Then, every pair $(\rho^\eta,x^\eta)$ attainable by a stable state-dependent policy is represented on the right-hand side of \eqref{eq:static-sufficiency-objective}.

Conversely, take any $x\in\mathcal X$ with $a_i(x)<\mu_i$ for all $i\in\mathcal I$. Theorem~\ref{thm:attainable-workload} shows that this flow can be implemented by a stable state-independent randomized dispatch policy whose workload vector is
$$
    \rho_i=\frac{a_i(x)}{\mu_i},
    \qquad i\in\mathcal I,
$$
and whose long-run assignment rates are exactly $x$. Therefore the feasible pairs $(\rho,x)$ on both sides of \eqref{eq:static-sufficiency-objective} coincide, which proves the equality of the two infima. \Halmos
\hfill
\Halmos
\endproof

\subsection{Proof of Corollary~\ref{cor:bottleneck-structure-lattice}}
\proof{Proof of Corollary~\ref{cor:bottleneck-structure-lattice}.}
We first prove \eqref{eq:bottleneck-server-saturation} and \eqref{eq:no-crossing-flow-into-bottleneck}. Let $S^*$ be a bottleneck server set and let $x^*$ be any optimal solution of \eqref{eq:min-peak-workload}.
Since $x^*$ is optimal, we have
$$
    a_i(x^*)\le \theta^*\mu_i,
    \quad i\in\mathcal I .
$$
Moreover, by the definition of $D(S^*)$, all demand counted in $D(S^*)$ must be served by servers in $S^*$. Therefore,
$$
    D(S^*)\le a(x^*)(S^*),
$$
where $a(x^*)(S^*)=\sum_{i\in S^*}a_i(x^*)$. Combining these inequalities gives
\begin{equation}
    D(S^*)
    \le a(x^*)(S^*)
    \le \theta^*M(S^*)
    =D(S^*),
    \label{eq:bottleneck-equality-chain}
\end{equation}
where the last equality follows from $S^*\in\mathcal B^*$. Hence all inequalities in \eqref{eq:bottleneck-equality-chain} hold as equalities.

Since $a_i(x^*)\le \theta^*\mu_i$ for each $i\in S^*$ and
$$
    \sum_{i\in S^*}a_i(x^*)
    =
    \sum_{i\in S^*}\theta^*\mu_i,
$$
we must have $a_i(x^*)=\theta^*\mu_i$ for every $i\in S^*$. This proves \eqref{eq:bottleneck-server-saturation}.

It remains to prove \eqref{eq:no-crossing-flow-into-bottleneck}. By the definition of $a_i(x^*)$,
\begin{align*}
    a(x^*)(S^*)
    &=
    \sum_{i\in S^*}\lambda_i
    +
    \sum_{e\in\mathcal E}
    \sum_{i\in e\cap S^*}x_{ei}^* .
\end{align*}
For every hyperedge $e\subseteq S^*$, feasibility of $x^*$ implies
$$
    \sum_{i\in e}x_{ei}^*=\lambda_e .
$$
Therefore,
\begin{align*}
    a(x^*)(S^*)-D(S^*)
    &=
    \sum_{e\in\mathcal E:\, e\not\subseteq S^*}
    \sum_{i\in e\cap S^*}x_{ei}^* .
\end{align*}
The left-hand side is zero by \eqref{eq:bottleneck-equality-chain}. Since every term on the right-hand side is nonnegative, each term must be zero. Hence
$$
    x_{ei}^*=0,
    \qquad e\not\subseteq S^*,\quad i\in e\cap S^*,
$$
which proves \eqref{eq:no-crossing-flow-into-bottleneck}.

We next prove the lattice property. The function $D$ is supermodular. In particular, its primary-demand term $\sum_{i\in S}\lambda_i$ is modular, and for each hyperedge $e\in\mathcal E$, the set function $S\mapsto \mathbbm{1}\{e\subseteq S\}$ is supermodular. Since $M$ is modular, the function
$$
    G(S) \coloneqq D(S)-\theta^*M(S)
$$
is supermodular. By the definition of $\theta^*$, we have $G(S)\le 0$ for every $S\subseteq\mathcal I$. If $S,T\in\mathcal B^*$, then $G(S)=G(T)=0$. Supermodularity gives
$$
    0
    =
    G(S)+G(T)
    \le
    G(S\cup T)+G(S\cap T)
    \le
    0.
$$
Therefore $G(S\cup T)=G(S\cap T)=0$, which implies $S\cup T\in\mathcal B^*$ and $S\cap T\in\mathcal B^*$.

Because $\mathcal I$ is finite, the union of all bottleneck server sets is also an element of $\mathcal B^*$ by repeated application of the union-closure property. This union contains every bottleneck server set and is therefore the unique maximal bottleneck server set.
\hfill \Halmos
\endproof

\subsection{Proof of Proposition~\ref{cor:water-filling}}
\proof{Proof of Proposition~\ref{cor:water-filling}.}
Throughout, $\mu_i=\mu$ for all $i$, so $\rho_i(x)=a_i(x)/\mu$. By Theorem~\ref{thm:attainable-workload} and Lemma~\ref{lemma:hypergraph-flow-characterization}, the attainable workload vectors correspond, via $a_i=\mu\rho_i$, to the points of the polytope
$$
    \mathcal A
    \coloneqq
    \left\{a(x):x\in\mathcal X\right\}
    =
    \left\{a\in\mathbb R_+^I:
    a(\mathcal I)=\lambda,\
    a(S)\ge D(S)\ \forall S\subseteq\mathcal I\right\}.
$$
Since $a(\mathcal I)=\lambda$ is constant on $\mathcal A$, the workload variance
$$
    \mathrm{Var}(\rho)
    =\frac{1}{I}\sum_{i\in\mathcal I}\rho_i^2
     -\left(\frac{\lambda}{I\mu}\right)^2
$$
is a strictly increasing affine transformation of $\varphi(a)\coloneqq\sum_{i\in\mathcal I}a_i^2$. Hence minimizing the workload variance, or equivalently the standard deviation, over $\mathcal R_{\mathrm{st}}$ is equivalent to minimizing the strictly convex function $\varphi$ over $\mathcal A$.

\emph{Step 1.}
We first verify by induction that each residual function has the trapped-demand form
\begin{equation}
    D_k(S)
    =\sum_{i\in S}\lambda_i
     +\sum_{e\in\mathcal E:\,\varnothing\ne e\cap U_k\subseteq S}\lambda_e,
    \qquad S\subseteq U_k.
    \label{eq:residual-trapped-demand}
\end{equation}
The base case $k=1$ is the definition of $D$. For the induction step, the contraction $D_{k+1}(S)=D_k(S\cup S_k)-D_k(S_k)$ together with \eqref{eq:residual-trapped-demand} gives
$$
    D_{k+1}(S)
    =\sum_{i\in S}\lambda_i
     +\sum_{\substack{e:\;\varnothing\ne e\cap U_k\subseteq S\cup S_k\\
                      e\cap U_k\not\subseteq S_k}}\lambda_e .
$$
An edge $e$ in the last sum satisfies $e\cap U_{k+1}=(e\cap U_k)\setminus S_k\ne\varnothing$ and $e\cap U_{k+1}\subseteq S$; conversely, every edge with $\varnothing\ne e\cap U_{k+1}\subseteq S$ appears in the sum. This proves \eqref{eq:residual-trapped-demand} for $k+1$. In particular, each $(U_k,D_k)$ is a system of the same form as $(\mathcal I,D)$, with residual hyperedges $e\cap U_k$, so Theorem~\ref{thm:bottleneck-stability} and Corollary~\ref{cor:bottleneck-structure-lattice} apply verbatim at every level; the maximal bottleneck set $S_k$ is therefore well defined and unique, and the recursion terminates after $m\le I$ steps.

Next, suppose $\theta_{k+1}\ge\theta_k$ for some level $k$, and let $S\subseteq U_{k+1}$ attain $\theta_{k+1}$. Then
$$
    D_k(S\cup S_k)
    =D_{k+1}(S)+D_k(S_k)
    \ge\theta_k M(S)+\theta_k M(S_k)
    =\theta_k M(S\cup S_k).
$$
Since $\theta_k$ is the maximum ratio at level $k$, the inequality must hold with equality, so $S\cup S_k$ is a level-$k$ bottleneck set strictly containing $S_k$, contradicting the maximality of $S_k$. Hence $\theta_1>\theta_2>\dots>\theta_m$, and $\theta_1=\theta^*<1$ by assumption, with $S_1$ the maximal bottleneck set of Definition~\ref{def:bottleneck-server-set}.

\emph{Step 2.}
Say that edge $e\in\mathcal E$ is \emph{resolved} at level $k(e)\coloneqq\min\{k:e\cap U_k\subseteq S_k\}$; this level exists because $U_m=S_m$. For each level $k$, let $x^k$ be an optimal flow of the level-$k$ peak-workload problem \eqref{eq:min-peak-workload} on the system $(U_k,D_k)$. By Corollary~\ref{cor:bottleneck-structure-lattice} applied at level $k$, $x^k$ loads every server in $S_k$ exactly to $\theta_k\mu$ and assigns no flow to $S_k$ from residual edges not contained in $S_k$. Define $x^\ast$ by giving every edge $e$ its allocation from level $k(e)$. Then $x^\ast\in\mathcal X$: each edge is fully allocated within $e\cap S_{k(e)}\subseteq e$. Moreover, for $i\in S_k$, edges resolved at levels $l<k$ satisfy $e\cap U_k=\varnothing$ and contribute nothing to server $i$, while edges resolved at levels $l>k$ are allocated within $S_l\subseteq U_{k+1}$, disjoint from $S_k$. Hence $a_i(x^\ast)=a_i(x^k)=\theta_k\mu$ for every $i\in S_k$, i.e., $\rho(x^\ast)=\rho^\ast$. Since $\theta_1<1$, we have $a_i(x^\ast)<\mu_i$ for all $i$, so by Theorem~\ref{thm:attainable-workload} the state-independent policy $\bar\eta_e(i)=x^\ast_{ei}/\lambda_e$ is stable and attains $\rho^\ast$. This completes the proof of part 1.

\emph{Step 3.}
The polytope $\mathcal X$ is a Cartesian product over $e\in\mathcal E$ of the scaled simplices $\{x_e\ge0:\sum_{i\in e}x_{ei}=\lambda_e\}$. At any $x\in\mathcal X$, the cone of feasible directions is generated by the swap directions $d_{e,i\to j}$, which decrease $x_{ei}$ and increase $x_{ej}$ at unit rate for $i,j\in e$, and which are feasible whenever $x_{ei}>0$. Since $a(\cdot)$ is affine, the directional derivative of $\varphi$ along $d_{e,i\to j}$ equals $2\left(a_j(x)-a_i(x)\right)$. Because $\varphi$ is convex and $\mathcal X$ is a polytope, $x$ minimizes $\varphi$ over $\mathcal X$ if and only if every feasible directional derivative is nonnegative, i.e., if and only if
\begin{equation}
    x_{ei}>0
    \;\Longrightarrow\;
    a_i(x)\le a_j(x)\quad\forall j\in e,
    \label{eq:equilibrium-in-a}
\end{equation}
which is the equilibrium condition \eqref{eq:balance-equilibrium} after dividing by $\mu$.

\emph{Step 4.}
We claim $x^\ast$ satisfies \eqref{eq:equilibrium-in-a}. Let $e$ be resolved at level $k$ and let $i\in e$ with $x^\ast_{ei}>0$; then $i\in S_k$ and $\rho_i(x^\ast)=\theta_k$. Every other $j\in e$ lies in some block $S_l$ with $l\le k$: members of $e$ inside $U_k$ belong to $S_k$ by the definition of $k(e)$, and members outside $U_k$ belong to earlier blocks. Hence $\rho_j(x^\ast)=\theta_l\ge\theta_k$ by Step 1, so $\rho_i(x^\ast)=\min_{j\in e}\rho_j(x^\ast)$ and \eqref{eq:equilibrium-in-a} holds. By Step 3, $x^\ast$ minimizes $\varphi$ over $\mathcal X$.

Since $\varphi$ is strictly convex and $\mathcal A$ is convex and compact, the minimizing arrival-rate vector over $\mathcal A$ is unique, and by the previous paragraph it equals $a(x^\ast)=\mu\rho^\ast$. Combining this with the reduction at the beginning of the proof and with Corollary~\ref{cor:static-sufficiency-workload}, $\rho^\ast$ is the unique minimizer of the workload standard deviation over $\mathcal R_{\mathrm{st}}$, and the minimum is attained by the stable state-independent policy $\bar\eta$ of Step 2. This proves part 2.

Finally, for part 3, let $x\in\mathcal X$. If $x$ satisfies \eqref{eq:balance-equilibrium}, then by Step 3 it minimizes $\varphi$, and by uniqueness $a(x)=\mu\rho^\ast$, so $x$ attains $\rho^\ast$. Conversely, if $x$ attains $\rho^\ast$, then $x$ minimizes $\varphi$ over $\mathcal X$, and the necessity direction of Step 3 yields \eqref{eq:balance-equilibrium}.
\hfill \Halmos
\endproof

\section{Comparison with hypercube model}\label{sec:compare_hypercube}

We selected the original hypercube model \citep{larson1974hypercube} as our primary baseline because it is a foundational benchmark in spatial queueing systems and has been widely used to model emergency services under non-overlapping districting. Unlike more specialized models that incorporate application-specific complexities, the hypercube model provides a clean and interpretable point of comparison, allowing us to isolate the performance gains of our proposed hyperlattice framework, which extends the hypercube by incorporating server-level queues and overlapping service regions.
To enable a meaningful comparison, we construct a synthetic system subjected to traditional non-overlapping districting, defined by a $2 \times 2$ grid encompassing four identical sub-regions, each served by its respective server.
For a meaningful comparison with the hypercube queueing model, where servers can attend to calls from any region if other servers are unavailable, we introduce an overlapping service zone comprising all four sub-regions. This design ensures no primary service area is designated, granting every server the flexibility to address calls from any sub-region, i.e., $\mathcal{E} = \{\mathcal{I}\}$.
Furthermore, we operate under the assumption that the call distributions are even across these sub-regions. The simulation dispatch policy is assumed to be $\eta_{\mathcal{I}} = \left( \frac{1}{2}, \frac{1}{12}, \frac{1}{4}, \frac{1}{6} \right)$ across all states, with a preference for allocating idle servers if available. This results in a call arrival rate of $\lambda/4$  = 1 for each.
All servers in this model exhibit consistent performance, with a uniform service rate denoted as $\mu$ = 1.
Lastly, we adopt a standardized dispatch policy that is uniformly applied across all servers.

\begin{figure}[!t]
\FIGURE
    {\includegraphics[width=.8\linewidth]{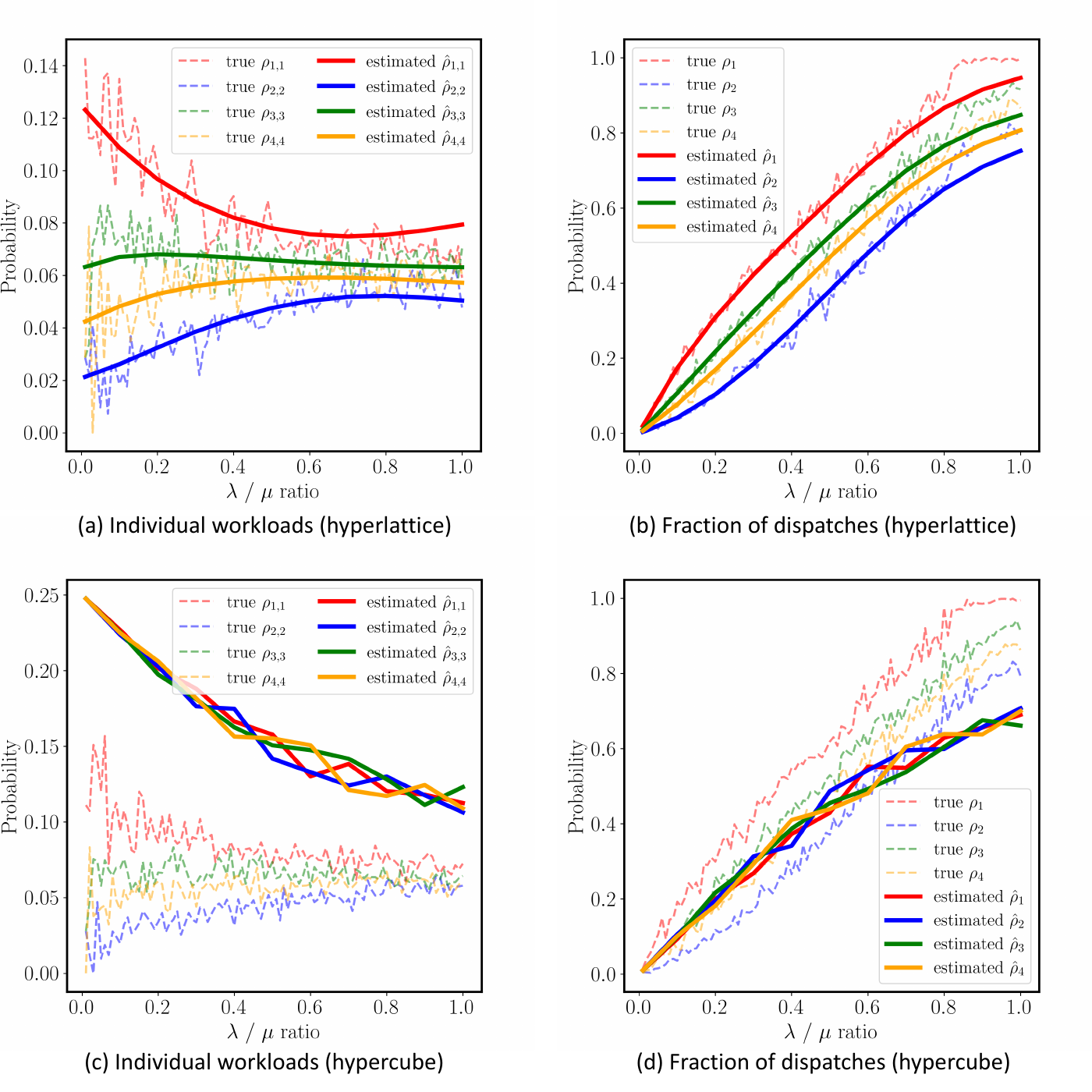}}
    { Synthetic results for four-server systems with varying $\lambda/\mu$ ratio. \label{fig:exp-setting-eta-alt}}
    { (a) individual workloads $\rho_i$ and (b) fraction of dispatches to its own region $\rho_{i,i}$ with varying $\lambda/\mu$ ratio. The solid lines represent the estimated performance measures using hyperlattice queueing models, whereas the dashed lines represent the true performance measures obtained from simulation runs. (c) and (d) present the corresponding results for Larson's hypercube queueing model.}
\end{figure}

Figures~\ref{fig:exp-setting-eta-alt} (a) and (b) show individual workloads $\rho_i$ and the fraction of dispatches to their own regions $\rho_{i,i}$ across different $\lambda/\mu$ ratios for the hyperlattice queueing model, while Figures~\ref{fig:exp-setting-eta-alt} (c) and (d) present similar results for the hypercube queueing model. We observe that the hyperlattice model closely approximates the trend of the simulation results for varying $\lambda/\mu$ ratios. In contrast, the hypercube model fails to provide a reasonable estimation of the simulation results.

\begin{figure}[!t]
\FIGURE
    {\includegraphics[width=\linewidth]{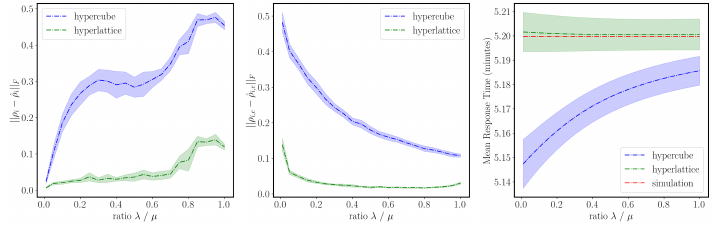}}
    { Performance metrics from queueing models and simulations with varying $\lambda / \mu$ ratios. \label{fig:numerical_rho} }
    { (a) the Frobenius norm of the difference between individual workloads $\rho_i$ and (b) the fraction of dispatches $\rho_{i,e}$ derived from the hypercube/hyperlattice queueing models and simulation results (denoted as $\hat{\rho}_{i}$ and $\hat{\rho}_{i,e}$) for systems with varying $\lambda / \mu$ ratios. (c) displays the estimated mean response time from the hypercube/hyperlattice queueing models for systems with varying $\lambda / \mu$ ratios. The dashed lines and the color regions represent the mean and the 95\% confidence interval of the norm obtained from 10 simulation runs. }
\end{figure}

Figures~\ref{fig:numerical_rho} (a) and (b) demonstrate the discrepancies in the estimation of individual workloads and dispatch fractions between the hypercube and hyperlattice models.
In Figure~\ref{fig:numerical_rho} (a), it is evident that the hypercube model's effectiveness diminishes substantially as the ratio escalates, whereas our method sustains commendable performance even at higher ratios. This advantage is primarily attributed to the hyperlattice's capability to account for more detailed dynamics between queues, a feature not present in the hypercube model.
Figure~\ref{fig:numerical_rho} (b) reinforces the superior predictive accuracy of the hyperlattice model in estimating the fraction of dispatches, regardless of the $\lambda/\mu$ ratio.
The success is attributed to the harmonized policy adherence between the simulation and that implemented by the hyperlattice model.
Furthermore, this dispatch policy demonstrates robustness against fluctuations in the $\lambda/\mu$ ratio.
Conversely, the hypercube model, which assumes a shared queue among all servers, fails to capture high-workload dynamics due to its simplified representation of the service process.
 In Figure \ref{fig:numerical_rho} (c), we further compare the mean response time between the hypercube model, the hyperlattice model, and the simulation results. Our findings indicate that the hyperlattice model consistently provides a close approximation of the mean response time across various $\lambda/\mu$ ratios. In contrast, the hypercube model offers a more accurate approximation when the system is heavily loaded, but performs poorly when the system is nearly empty.
Moreover, the hypercube model is constrained to a myopic policy (usually only dependent on travel distance), adopting a fixed preference for dispatch that becomes noticeably inadequate at lower $\lambda/\mu$ ratios, leading to a pronounced divergence from simulation results. However, an increase in the ratio results in a shift in the hypercube model towards a homogeneous policy, akin to the one employed in our simulation, consequently reducing the gap in estimation accuracy.
Such a shift highlights the limitations of the hypercube model, underscoring its restricted adaptability to various scenarios.

 \begin{figure}[!htbp]
\FIGURE
    {\includegraphics[width=0.9\linewidth]{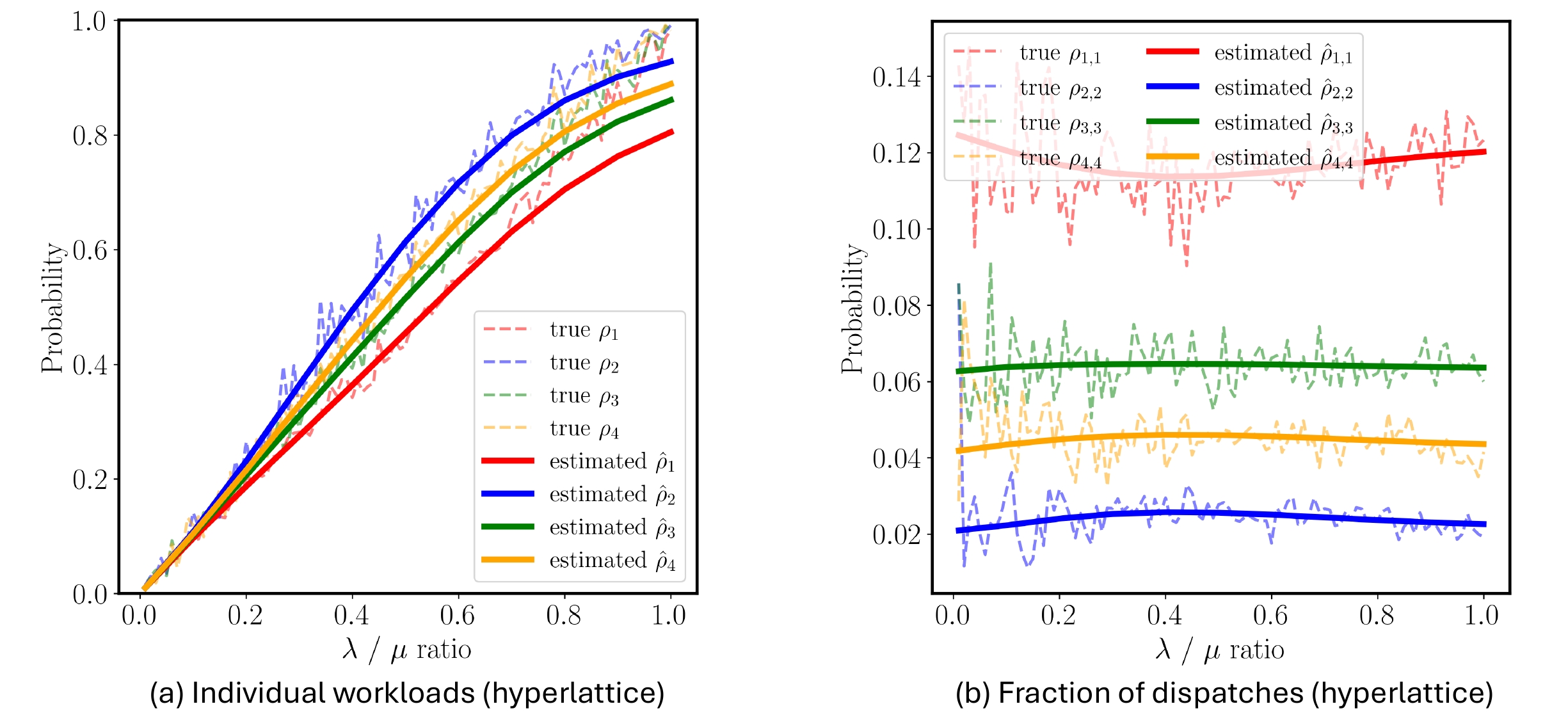}}
    {Performance metrics from queueing models and simulations with heterogeneous service rates. \label{fig:heter_service_rate} }
    {
    (a) The individual workloads $\rho_i$ and (b) the fraction of dispatches to each server's own region $\rho_{i,i}$, respectively, with varying $\lambda / \mu$ ratios. The solid lines represent estimates from the hyperlattice queueing model, while the dashed lines represent simulation results.
    }
\end{figure}
In Figure~\ref{fig:heter_service_rate}, we present the performance metrics of our hyperlattice model, analogous to those in Figure~\ref{fig:exp-setting-eta-alt}, but under heterogeneous service rates.
Recall that in our synthetic experiment, the overlapping region covers the entire service area to facilitate comparison with the hypercube model.
Therefore, we incorporate server-specific travel times to the overlapping region, resulting in different service rates across servers. The results show that the truncated hyperlattice model continues to closely approximate the simulation outcomes, demonstrating its robustness under heterogeneous service conditions.

 \begin{figure}[!htbp]
\FIGURE
    {\includegraphics[width=0.9\linewidth]{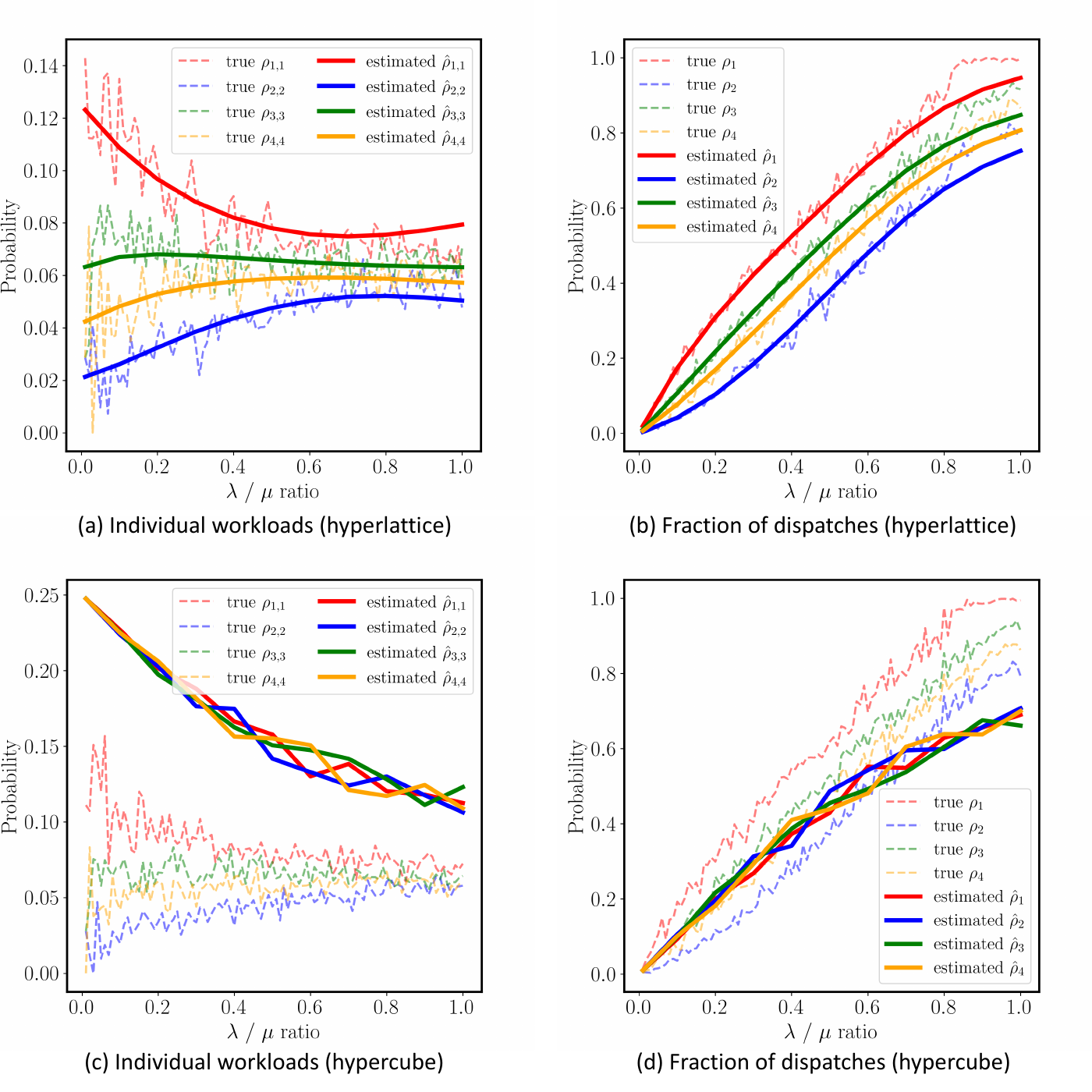}}
    {Performance metrics under server-specific priority dispatch policy.\label{fig:priority} }
    {
    (a) individual workloads $\rho_i$ and (b) fraction of dispatches to its own region $\rho_{i,i}$ with varying $\lambda/\mu$ ratio. The solid lines represent the estimated performance measures using hyperlattice queueing models, whereas the dashed lines represent the proper performance measures obtained from simulation runs. (c) and (d) present the corresponding results for Larson's hypercube queueing model \citep{larson1974hypercube}.
    }
\end{figure}

In Figure~\ref{fig:priority}, we compare the performance of our hyperlattice model with Larson's hypercube model under a server-specific priority dispatch policy, where a predefined preference list is assigned across servers. Specifically, we use the same experimental setup as in Figure~\ref{fig:exp-setting-eta-alt}, but modify the dispatch policy to follow a fixed preference order of $[1, 3, 4, 2]$ across the entire overlapping service region, e.g. server 1 has the highest priority, server 3 the second, server 4 the third, and server two the lowest. The plot demonstrates that (1) our hyperlattice model is capable of handling priority-based dispatch policies, and (2) under such policies, it accurately approximates the performance metrics of interest.

\end{document}

%% file: tables/workload.tex
\begin{table}[htbp]
\centering
\setlength{\tabcolsep}{6pt}
\caption{Workload allocation among police units across beats under different districting plans.
}
\begin{tabular}{c|cccccccccccc|c}
\hline
\multirow{2}{*}{Plan}
& \multicolumn{12}{c|}{Workload of police units}
& \multirow{2}{*}{$\sigma$} \\
\cline{2-13}
& 1 & 2 & 3 & 4 & 5 & 6 & 7 & 8 & 9 & 10 & 11 & 12 \\
\hline
a & 0.28 & 0.63 & - & 0.19 & 0.27 & 0.19 & 0.19 & 0.26 & 0.27 & 0.23 & 0.13 & 0.21 & 0.13 \\
b & 0.51 & 0.46 & - & 0.21 & 0.28 & 0.21 & 0.20 & 0.28 & 0.28 & 0.25 & 0.14 & 0.22 & 0.11 \\
c & 0.41 & 0.35 & - & 0.33 & 0.40 & 0.21 & 0.21 & 0.28 & 0.29 & 0.25 & 0.14 & 0.22 & 0.08 \\
\hline
d & 0.27 & 0.63 & - & 0.19 & 0.26 & 0.19 & 0.19 & 0.26 & 0.26 & 0.36 & - & 0.20 & 0.13 \\
e & 0.51 & 0.46 & - & 0.20 & 0.28 & 0.21 & 0.20 & 0.34 & 0.28 & 0.32 & - & 0.22 & 0.10 \\
f & 0.41 & 0.35 & - & 0.33 & 0.40 & 0.21 & 0.24 & 0.31 & 0.29 & 0.29 & - & 0.26 & 0.06 \\
\hline
g
& 0.26 & 0.61 & -
& 0.18 & 0.25 & 0.19
& 0.18 & 0.50 & -
& 0.34 & - & 0.20
& 0.15 \\
h
& 0.50 & 0.46 & -
& 0.35 & 0.41 & 0.20
& 0.20 & 0.34 & -
& 0.31 & - & 0.22
& 0.10 \\
i
& 0.41 & 0.35 & -
& 0.40 & 0.46 & 0.21
& 0.25 & 0.39 & -
& 0.36 & - & 0.26
& 0.08 \\
\hline
$*$ & 0.29 & 0.23 & 0.45 & 0.21 & 0.28 & 0.21 & 0.20 & 0.28 & 0.28 & 0.25 & 0.13 & 0.22 & 0.07 \\
\hline
\end{tabular}
\label{table:workload_table}
\end{table}

%% file: tables/optimal_workload.tex
\begin{table}[htbp]
\centering
\caption{Bottleneck structure and balance-optimal dispatch for the districting plans of Table~\ref{table:workload_table}.
}
\setlength{\tabcolsep}{2.5mm}
\begin{tabular}{c|c|cc|cc}
\hline
Plan & Units & $\sigma$ (uniform $\eta$) & $\sigma$ (balance-optimal $\bar\eta^\ast$) & $\theta^*$ & Maximal bottleneck regions \\
\hline
b & 11 & 0.108 & 0.107 & 2.92 & \{1, 2\} \\
c & 11 & 0.081 & 0.079 & 2.19 & \{1, 2, 4, 5\} \\
e & 10 & 0.102 & 0.102 & 2.92 & \{1, 2\} \\
f & 10 & 0.061 & 0.055 & 2.19 & \{1, 2, 4, 5\} \\
h & 9 & 0.105 & 0.103 & 2.92 & \{1, 2\} \\
i & 9 & 0.080 & 0.064 & 2.27 & \{1, 2, 4, 5, 8, 10\} \\
$*$ & 12 & 0.073 & 0.073 & 2.69 & \{3\} \\
\hline
\end{tabular}
\label{table:optimal_workload}
\end{table}

%% file: tables/mtt_table.tex
\begin{table}[h!]
\centering
\setlength{\tabcolsep}{3.3pt}
\caption{Mean travel times for police units across beats under different districting plans.
}
\begin{tabular}{c|cccccccccccc|c}
\hline
\multirow{2}{*}{Plan}
& \multicolumn{12}{c|}{Mean travel time of police units (minutes)}
& \multirow{2}{*}{MTT} \\
\cline{2-13}
& 1 & 2 & 3 & 4 & 5 & 6 & 7 & 8 & 9 & 10 & 11 & 12 \\
\hline
a & 12.15 & 14.24 & - & 12.26 & 12.07 & 12.20 & 11.81 & 11.94 & 12.01 & 12.09 & 12.16 & 11.82 & 12.54 \\
b & 13.89 & 13.64 & - & 12.26 & 12.07 & 12.20 & 11.81 & 11.94 & 12.01 & 12.09 & 12.16 & 11.82 & 12.59 \\
c & 13.24 & 13.10 & - & 13.75 & 12.62 & 12.20 & 11.81 & 11.94 & 12.01 & 12.09 & 12.16 & 11.82 & 12.55 \\
\hline
d & 12.15 & 14.24 & - & 12.26 & 12.07 & 12.20 & 11.81 & 11.94 & 12.01 & 12.69 & - & 11.82 & 12.61 \\
e & 13.90 & 13.64 & - & 12.26 & 12.07 & 12.20 & 11.81 & 12.41 & 12.01 & 12.44 & - & 11.82 & 12.67 \\
f & 13.24 & 13.10 & - & 13.75 & 12.62 & 12.20 & 12.21 & 12.19 & 12.01 & 12.28 & - & 11.99 & 12.63 \\
\hline
g & 12.15 & 14.24 & - & 12.26 & 12.07 & 12.20 & 11.81 & 12.99 & - & 12.69 & - & 11.82 & 12.80 \\
h & 13.90 & 13.64 & - & 13.89 & 12.93 & 12.20 & 11.81 & 12.41 & - & 12.44 & - & 11.82 & 13.00 \\
i & 13.29 & 13.14 & - & 14.14 & 12.90 & 12.20 & 12.23 & 12.51 & - & 12.84 & - & 12.00 & 12.91 \\
\hline
$*$ & 12.15 & 12.03 & 12.21 & 12.26 & 12.07 & 12.20 & 11.81 & 11.94 & 12.01 & 12.09 & 12.16 & 11.82 & 12.07 \\
\hline
\end{tabular}
\label{table:travel_time_table}
\end{table}

%% file: tables/staffing.tex
\begin{table}[htbp]
\centering
\caption{Minimum staffing for stabilizability under each districting plan.}
\setlength{\tabcolsep}{2.3mm}
\begin{tabular}{c|cccccccccccc|cc|c}
\hline
\multirow{2}{*}{Plan} & \multicolumn{12}{c|}{Number of cars homed at each region's unit} & \multirow{2}{*}{$N^*$} & \multirow{2}{*}{$\theta^*_{N^*}$} & \multirow{2}{*}{$\theta^*$ (1 car/unit)} \\
\cline{2-13}
  & 1 & 2 & 3 & 4 & 5 & 6 & 7 & 8 & 9 & 10 & 11 & 12 & & & \\
\hline
a & 2 & 5 & - & 2 & 2 & 2 & 2 & 2 & 2 & 2 & 1 & 2 & 24 & 0.89 & 4.06 \\
b & 4 & 2 & - & 2 & 2 & 2 & 2 & 2 & 2 & 2 & 1 & 2 & 23 & 0.97 & 2.92 \\
c & 3 & 2 & - & 2 & 2 & 2 & 2 & 2 & 2 & 2 & 1 & 2 & 22 & 0.97 & 2.19 \\
\hline
d & 2 & 5 & - & 2 & 2 & 2 & 2 & 2 & 2 & 3 & - & 2 & 24 & 0.89 & 4.06 \\
e & 4 & 2 & - & 2 & 2 & 2 & 2 & 2 & 2 & 2 & - & 2 & 22 & 0.99 & 2.92 \\
f & 3 & 2 & - & 2 & 2 & 2 & 2 & 2 & 2 & 2 & - & 2 & 21 & 0.97 & 2.19 \\
\hline
g & 2 & 5 & - & 2 & 2 & 2 & 2 & 4 & - & 3 & - & 2 & 24 & 0.89 & 4.06 \\
h & 2 & 4 & - & 2 & 3 & 2 & 2 & 2 & - & 2 & - & 2 & 21 & 0.99 & 2.92 \\
i & 4 & 2 & - & 2 & 2 & 2 & 2 & 2 & - & 2 & - & 2 & 20 & 0.97 & 2.27 \\
\hline
$*$ & 2 & 2 & 3 & 2 & 2 & 2 & 2 & 2 & 2 & 2 & 1 & 2 & 24 & 0.90 & 2.69 \\
\hline
\end{tabular}

\vspace{0.4em}
\begin{minipage}{\textwidth}
\footnotesize
\textit{Note.} For each plan, the entries report the number of cars homed at each region's unit under one optimal allocation. Here, $N^*$ denotes the minimum total number of cars required to achieve $\theta^*<1$, and $\theta^*_{N^*}$ is the corresponding bottleneck ratio.
\end{minipage}

\label{table:staffing}
\end{table}